\documentclass[reqno]{amsart}%[11pt, reqno]{article}
\usepackage{geometry}                % See geometry.pdf to learn the layout options. There are lots.
\usepackage{graphicx}
\usepackage{amssymb}
\usepackage{epstopdf}
\usepackage{amsfonts}
\usepackage{amsthm}
\usepackage{float}
\usepackage{ mathrsfs }
\usepackage{color}
\usepackage{pifont}
\usepackage{hyperref}
\hypersetup{
  colorlinks   = true, %Colours links instead of ugly boxes
  urlcolor     = blue, %Colour for external hyperlinks
  linkcolor    = blue, %Colour of internal links
  citecolor   = red %Colour of citations
}
\usepackage{tikz}
\usetikzlibrary{decorations.markings}
\usepackage{enumerate}

\theoremstyle{plain}
\newtheorem{thm}{Theorem}[section]
\newtheorem{prop}[thm]{Proposition}
\newtheorem{lemma}[thm]{Lemma}
\newtheorem{cor}[thm]{Corollary}

\theoremstyle{remark}
\newtheorem{remark}[thm]{Remark}

\theoremstyle{definition}
\newtheorem{defn}[thm]{Definition}
\newtheorem{ex}[thm]{Example}

\newcommand{\con}{%
  \mathord{
    \mathchoice
    {\raisebox{1ex}{\scalebox{.7}{$\frown$}}}
    {\raisebox{1ex}{\scalebox{.7}{$\frown$}}}
    {\raisebox{.7ex}{\scalebox{.5}{$\frown$}}}
    {\raisebox{.7ex}{\scalebox{.5}{$\frown$}}}
  }
}

\newcommand{\is}{\sqsubseteq}
\newcommand{\sis}{\sqsubset}

\newcommand{\dom}{\mbox{dom}}
\newcommand{\range}{\mbox{range}}

\newcommand{\up}{\mathchoice%
    {\mbox{$\displaystyle\downarrow$}}%
    {\mbox{$\displaystyle\downarrow$}}%
    {\mbox{$\scriptstyle\downarrow$}}%
    {\mbox{$\scriptscriptstyle\downarrow$}}%
} 
\newcommand{\ups}{\mathchoice%
    {\rotatebox[origin=c]{-90}{$\displaystyle\rightarrowtail$}}%
    {\rotatebox[origin=c]{-90}{$\displaystyle\rightarrowtail$}}%
    {\rotatebox[origin=c]{-90}{$\scriptstyle\rightarrowtail$}}%
    {\rotatebox[origin=c]{-90}{$\scriptscriptstyle\rightarrowtail$}}%
} 

\newcommand{\down}{\mathchoice%
    {\mbox{$\displaystyle\uparrow$}}%
    {\mbox{$\displaystyle\uparrow$}}%
    {\mbox{$\scriptstyle\uparrow$}}%
    {\mbox{$\scriptscriptstyle\uparrow$}}%
} 
\newcommand{\downs}{\mathchoice%
    {\rotatebox[origin=c]{90}{$\displaystyle\rightarrowtail$}}%
    {\rotatebox[origin=c]{90}{$\displaystyle\rightarrowtail$}}%
    {\rotatebox[origin=c]{90}{$\scriptstyle\rightarrowtail$}}%
    {\rotatebox[origin=c]{90}{$\scriptscriptstyle\rightarrowtail$}}%
}

\newcommand{\im}{\mbox{im}}
\newcommand{\rank}{\mbox{rank}}
\newcommand{\On}{\mathrm{On}}
\newcommand{\card}{\mathrm{Card}}

\newcommand{\ot}{\mathrm{ot}}
\newcommand{\cmark}{\ding{51}}%
\newcommand{\xmark}{\ding{55}}%
\newcommand{\MM}{\mathcal{F}}
\newcommand{\PP}{\mathcal{A}}
\newcommand{\RR}{\mathcal{L}}
\newcommand{\CC}{\mathcal{Q}}
\newcommand{\AAA}{\mathcal{B}}
\newcommand{\Btree}{\mathsf{B}}
\newcommand{\Qtree}{\mathsf{Q}}
\newcommand{\Atree}{\mathsf{A}}
\newcommand{\Qhattree}{\mathsf{C}}
\newcommand{\col}{\mathbf{c}}

\newcommand{\arity}{\mathbf{a}}

\newcommand{\OA}{\mathcal{C}}
\newcommand{\amrs}{\tilde{\OA}_0}%{(\AAA,\MM)^\RR_0}
\newcommand{\amr}{\tilde{\OA}}%{(\AAA,\MM)^\RR}
\newcommand{\amri}{\tilde{\OA}_\infty}%{(\AAA,\MM)^\RR_\infty}
\newcommand{\PS}{\mathfrak{I}}
\newcommand{\DS}{\mathfrak{D}}
\newcommand{\FS}{\mathfrak{F}}
\newcommand{\scat}{\mathscr{S}}
\newcommand{\scatt}{\mathscr{U}}
\newcommand{\sscat}{\mathscr{M}}
\newcommand{\sscatt}{\mathscr{T}}
\newcommand{\pscat}{\mathscr{P}}
\newcommand{\ctbl}{\mathscr{C}}
\newcommand{\Trank}{\mbox{rank}}
\newcommand{\Srank}{\mbox{rank}_\scatt}
\newcommand{\Troot}[1]{\mbox{root}(#1)}
\newcommand{\Qhat}{2^{<\omega}_\perp}
\newcommand{\pth}{\mathbb{M}}
\newcommand{\SSR}[3]{\mathrm{SSR}(#1;#2,#3)}
\newcommand{\obj}[1]{\mathrm{obj}{(#1)}}
\newcommand{\mor}[1]{\mathrm{hom}{(#1)}}
\newcommand{\homm}{\mathrm{hom}}
\newcommand{\AC}[1]{\mathrm{A}_{#1}}
\newcommand{\CH}[1]{\mathrm{C}_{#1}}
\newcommand{\ra}{\gamma}
\newcommand{\barity}[1]{\mathbf{b}(#1)}
\newcommand{\clos}[1]{\overline{#1}}
\renewcommand{\leq}{\leqslant}
\renewcommand{\geq}{\geqslant}

\definecolor{g0}{gray}{0.4}
\definecolor{g1}{gray}{0.5}
\definecolor{g2}{gray}{0.6}
\definecolor{g3}{gray}{0.7}
\tikzset{->-/.style={decoration={
  markings,
  mark=at position .63 with {\arrow{stealth}}},postaction={decorate}}}
  \tikzset{-<-/.style={decoration={
  markings,
  mark=at position .5 with {\arrow{stealth reversed}}},postaction={decorate}}}

\begin{document}

\title{On better-quasi-ordering classes of partial orders}
\author{Gregory McKay}
\address{School of Mathematics,
University of East Anglia,
Norwich Research Park,
Norwich,
NR4 7TJ,
UK}
\email{g.mckay@uea.ac.uk}
\thanks{This work was carried out as part of the author's Ph.D. thesis at the University of East Anglia in the UK}
%03E05=Other comb set theory
%
%
%
%

\begin{abstract}
We provide a method of constructing better-quasi-orders by generalising a technique for constructing operator algebras that was developed by Pouzet. We then generalise the notion of $\sigma$-scattered to partial orders, and use our method to prove that the class of $\sigma$-scattered partial orders is better-quasi-ordered under embeddability. 
This generalises theorems of Laver, Corominas and Thomass\'{e} regarding $\sigma$-scattered linear orders and trees, countable forests and $N$-free partial orders respectively. In particular, a class of \emph{countable} partial orders is better-quasi-ordered whenever the class of indecomposable subsets of its members satisfies a natural strengthening of better-quasi-order.
\end{abstract}
\maketitle

\tableofcontents
\section{Introduction}\label{Section:Intro}
Some of the most striking theorems in better-quasi-order (bqo) theory are that certain classes of partial orders, often with colourings, are bqo under embeddability. Indeed, the notion of bqo was first used by Nash-Williams to prove that the class $\mathscr{R}$ of rooted trees of height at most $\omega$ has no infinite descending chains or infinite antichains under the embeddability quasi-order \cite{NWInfTrees}. %Laver explored the coloured versions of such trees in \cite{LaverFrOTconj}, expanding Nash-Williams' method, he proved that $\mathscr{R}$ preserves bqo. (That is to say that if $Q$ is bqo, then the class of trees of $\mathscr{R}$ coloured by $Q$ is also bqo under a natural embeddability ordering, see Definition \ref{Defn:Preservesbqo}.)

Another contribution from Nash-Williams comes in \cite{nashwilliamsseqs}. He proved that if $Q$ is bqo, then the class $\tilde{Q}$ of transfinite sequences of members of $Q$ is bqo (see \cite{simpson}). This theorem can be viewed as an embeddability result on a class of coloured partial orders; since it is equivalent to saying that the ordinals \emph{preserve bqo} (which is to say that if $Q$ is bqo, then the class of ordinals coloured by $Q$ is also bqo under a natural embeddability ordering; see Definition \ref{Defn:Preservesbqo}).  In fact Nash-Williams proved a technical strengthening of this statement, equivalent to saying that the ordinals are \emph{well-behaved} (see Definition \ref{Defn:WellBehaved}). This is an important generalisation of bqo that will be crucial in this paper, because it is much more useful than bqo or even preserving bqo when constructing large bqo classes.
%
%, Laved proved that if $Q$ is a bqo, then $\On$ is well-behaved (a technical strengthening of preserving bqo, see Definition \ref{Defn:WellBehaved}) and 

Perhaps the most well known result of this kind comes from Laver, who proved that the class $\sscat$ of $\sigma$-scattered linear orders preserves bqo, a positive result for a generalisation of Fra\"{i}ss\'{e}'s conjecture \cite{LaverFrOTconj}. A few years after his paper on $\sigma$-scattered linear orders, Laver also showed that the class $\sscatt$ of $\sigma$-scattered trees preserves bqo \cite{LaverClassOfTrees}. 
Initially we notice that there should be some connection between the theorems of $\sigma$-scattered linear orders and $\sigma$-scattered trees. In each, we first take all partial orders of some particular type (linear orders, trees) that do not embed some particular order (namely $\mathbb{Q}$ and $2^{<\omega}$% - notice that these are the minimal elements of the increasing unions of smaller structures
). In both cases, the class of countable unions of these objects turn out to preserve bqo.%, and in the first case it is known that they are well-behaved. 

We prove a general theorem of this type (Theorem \ref{Thm:MLPisWB2}) which states that given well-behaved classes $\mathbb{L}$ and $\mathbb{P}$ of linear orders and partial orders respectively, the class of `generalised $\sigma$-scattered partial orders' $\sscat^\mathbb{L}_\mathbb{P}$ will be well-behaved (see Definition \ref{Defn:PPL}). We define our general `scattered' partial orders to be those orders $X$ such that:
\begin{itemize}
	\item Every indecomposable subset of $X$ is contained in $\mathbb{P}$. (See Definition \ref{Defn:Indecomp}.)
	\item Every chain of intervals of $X$ under $\supseteq$ has order type in $\clos{\mathbb{L}}$. (See definitions \ref{Defn:Lclosure} and  \ref{Defn:Interval}.)
	\item No `pathological' order $2^{<\omega}$, $-2^{<\omega}$ or $\Qhat$ embeds into $X$. (See definitions \ref{Defn:2^<omega} and \ref{Defn:Qhat}).
\end{itemize}
Our class $\sscat^\mathbb{L}_\mathbb{P}$ is then the class of \emph{countable increasing unions} or \emph{limits} of such $X$.

Applying this theorem with classes known to be well-behaved yields generalisations of many other known results in this area. %Van Engelen, Miller and Steel proved that the class $\scat$ of scattered linear orders is well-behaved \cite{VMS}, which in turn was generalised to $\sscat$ by K\v{r}\'{i}\v{z} \cite{Kriz}. 
For example, Corominas showed that the class of countable $\ctbl$-trees\footnote{Here $\ctbl$ is the class of countable linear orders.} preserves bqo \cite{Corominas} (see Definition \ref{Defn:LTree}) and Thomass\'{e} showed that the class
of countable $N$-free partial orders preserves bqo \cite{Nfree} (see Definition \ref{Defn:Nfree}). 
We summarise known results as applications of Theorem \ref{Thm:MLPisWB2} in the following table.\footnote{Here $\AC{2}$ and $\CH{2}$ are the antichain and chain of size 2  respectively.} In each case Theorem \ref{Thm:MLPisWB2} tells us that the given class is well-behaved.\footnote{In the cases of $\scatt^{\On}$, $\sscatt^{\On}$ and $\sscatt^{\ctbl}$ the constructed $\sscat^\mathbb{L}_\mathbb{P}$ is actually a larger class of partial orders.}
\begin{center}
  \begin{tabular}{|c| c || c | c | c | c |}
    \hline
Class & Description  & $\mathbb{P}$ & $\mathbb{L}$ & Limits \\ \hline\hline
 	$\mathbb{N}$&Finite numbers under injective maps & $1$, $\AC{2}$  & 1 &\xmark \\ \hline
 	%$\On$ & $\bullet\perp \bullet$  & 1 &\cmark \\ \hline
    $\scat$ & Scattered linear orders \cite{LaverFrOTconj} & $1$, $\CH{2}$  & $\On\cup \On^*$ &\xmark \\ \hline
    $\sscat$ & $\sigma$-scattered linear orders \cite{LaverFrOTconj} & $1$, $\CH{2}$ & $\On\cup \On^*$ &\cmark \\ \hline
    $\scatt^{\On}$ & Scattered trees \cite{LaverClassOfTrees} & $1$, $\CH{2}$, $\AC{2}$ & $\On$ & \xmark \\ \hline
    $\sscatt^{\On}$ & $\sigma$-scattered trees \cite{LaverClassOfTrees} & $1$, $\CH{2}$, $\AC{2}$ & $\On$ & \cmark \\
    \hline
    $\sscatt^{\ctbl}$ & Countable $\mathscr{L}$-trees \cite{Corominas} & $1$, $\CH{2}$, $\AC{2}$ & $\ctbl$ & \cmark \\
    \hline
    $\ctbl_{\{1,\AC{2},\CH{2}\}}$ & Countable $N$-free partial orders \cite{Nfree}& $1$, $\CH{2}$, $\AC{2}$ & $\ctbl$ & \cmark\\
    \hline
  \end{tabular}
\end{center}

Applying this theorem with the largest known well-behaved classes $\mathbb{L}$ and $\mathbb{P}$ gives that some very large classes of partial orders are well-behaved (Theorem \ref{Thm:POTHM}). For example, let $\mathbb{P}$ be the set of indecomposable partial orders of cardinality less than some $n\in \omega$, and $\mathbb{L}=\sscat$. Then for $n>2$ the well-behaved class $\sscat^\mathbb{L}_\mathbb{P}$ contains the $\sigma$-scattered linear orders, $\sigma$-scattered $\sscat$-trees, countable $N$-free partial orders, and generalisations of such objects.

Crucial to the ideas in this paper are those of constructing objects with so called `structured trees'. Put simply, these are trees with some extra structure (usually a partial order) given to the set of successors of each element. Embeddings between structured trees are then required to induce embeddings of this extra structure.

Theorems on structured trees also appear throughout the literature on bqo theory (cf. \cite{Nfree, Corominas, Kriz, PouzetApps, Kruskal, LouveauStR}). The rationale for their usefulness is explained by Pouzet in \cite{PouzetApps}. %As mentioned in this paper, Kruskal proved that the set of finite trees structured with $\omega$ is bqo \cite{Kruskal}. 
His method is to take a `simple' class of objects (e.g. partial orders) and a bqo class of multivariate functions sending a list of objects to a new object (so called `operator algebras'). Closing the class under these functions then yields a new class, which one can prove to be bqo. The crucial step is to show that this construction can be encoded as a structured tree, contained inside a class which is known to preserve bqo. Pouzet's method however was limited in that the structured trees that he used were only `chain-finite' (i.e. those trees for which every chain is finite).

Since then, larger classes of structured trees have been shown to be preserve bqo. In particular, using a modification of the Minimal Bad Array Lemma (see \cite{simpson}), K\v{r}\'{i}\v{z} managed to prove that if $Q$ is well-behaved then $\mathscr{R}_Q$ (the class of $Q$-structured trees of $\mathscr{R}$, see Definition \ref{Defn:StructTrees}) is well-behaved \cite{Kriz}. %Using our construction, we will also prove a bqo theorem on structured trees (Theorem \ref{Thm:TLPWell-Behaved}).
We give a generalisation of Pouzet's method, that incorporates these larger classes of structured trees. This allows for iterating functions over a general linear order and for taking countable limits. Using this coding, these more complex classes can be shown to be bqo (even well-behaved).

% structured with ordinals is bqo. Ultimately, this is what Corominas used to prove that class of $\ctbl$-trees is bqo (see \cite{PouzetApps, Corominas}). 
%
% However  Most recently Thomass\'{e} used that a particular subclass of $\AC{2}$-structured $\ctbl$-trees preserves bqo in order to prove his theorem on $N$-free partial orders \cite{Nfree} (here $\AC{2}$ is the antichain of size $2$). %Namely the countable trees for which any chain has an infimum.
%
%In \cite{PouzetApps}, Pouzet also describes how structured trees can be used to code functions that construct more complicated orders, . 

So we aim to show that some large classes of partial orders are well-behaved, the main theorem being Theorem \ref{Thm:MLPisWB2} and its main application Theorem \ref{Thm:POTHM}. The general method of the proof will be as follows. In Section \ref{Section:OpConstruction} we define an operator algebra construction, similar to ideas explained in \cite{PouzetApps} but with our two generalisations. We also give an example of how to construct partial orders. %The idea is to take a `simple' class of objects and from it generate a new class of objects by iteratively applying certain functions. 
In Section \ref{Section:StructTrees} we encode this construction in terms of structured trees. In Section \ref{Section:WBConstruction} we prove that such a construction, under the correct conditions, will be bqo. In Section \ref{Section:StructuredRTrees}, using the structured tree theorem of K\v{r}\'{i}\v{z} from \cite{Kriz} (Theorem \ref{Thm:Kriz}), in conjunction with our construction theorem, we construct a more general well-behaved class of `$\sigma$-scattered structured $\mathbb{L}$-trees'. This serves to supercharge the construction theorem. In Section \ref{Section:PO} we prove a generalisation of Hausdorff's theorem on scattered order types, which characterises the class of partial orders that we constructed as precisely $\sscat^\mathbb{L}_\mathbb{P}$. This completes the proof of our main theorem, that this class is well-behaved. Finally in Section \ref{Section:Cors} we explain how this result expands all known bqo results on embeddability of coloured partial orders.

\section{Preliminaries}
\subsection{Basic bqo theory}
\begin{defn}
If $A$ is an infinite subset of $\omega$, let $[A]^\omega=\{X\subseteq A\mid  |X|=\aleph_0\}$ and $[A]^{<\omega}=\{X\subseteq A\mid |X|< \aleph_0\}$. We equate $X\in [A]^\omega$ with the increasing enumeration of elements of $X$.
\end{defn}

\begin{defn}
\begin{itemize}
	\item A class $Q$ with a binary relation $\leq_Q$ on $Q$ is called a \emph{quasi-order} whenever $\leq_Q$ is transitive and reflexive.% We will write $\leq$ for $\leq_Q$ when this is unambiguous. 
	\item If $Q$ is a quasi-order with $\leq_Q$ antisymmetric, then we call $Q$ a \emph{partial order}.
	\item For $a,b\in Q$ we write $a<_Q b$ iff $a\leq_Q b$ and $b\not \leq_Q a$. We write $a\perp_Q b$ and call $a$ and $b$ \emph{incomparable} iff $a\not \leq_Q b$ and $b\not \leq_Q a$.% We will write simply $<$ and $\perp$ when this is unambiguous. 
	\item We write $\leq$, $<$ and $\perp$ in place of $\leq_Q$, $<_Q$ and $\perp_Q$ when the context is clear.
%	\item A quasi-order $Q$ is called a \emph{well-quasi-order} (wqo) if $Q$ has no infinite antichains and no infinite descending sequences with respect to $\leq$.
%\end{itemize}
%\end{defn}
%\begin{defn}
%\begin{itemize}
	\item A function $f:[\omega]^\omega \rightarrow Q$ is called a \emph{$Q$-array} (or simply an \emph{array}) if $f$ is continuous (giving $[\omega]^\omega$ the product topology and $Q$ the discrete topology). 
	\item An array $f:[\omega]^\omega \rightarrow Q$ is called \emph{bad} if $\forall X\in [\omega]^\omega$ we have
	$$f(X)\not \leq f(X\setminus\{ \min X\}).$$
	\item An array $f:[\omega]^\omega \rightarrow Q$ is called \emph{perfect} if $\forall X\in [\omega]^\omega$ we have
	$$f(X) \leq f(X\setminus\{ \min X\}).$$
	\item A quasi-order $Q$ is called a \emph{better-quasi-order} (bqo) if there is no bad $Q$-array.
%	\item A $Q$-array $f$ is called \emph{good} if it is not bad.

\end{itemize}

\end{defn}
\begin{remark}
We note that we could replace `continuous' in the definition of a $Q$-array with `Borel measurable' and this would make no difference to the definition of bqo (see \cite{simpson}). We can also consider bad arrays with domain $[A]^\omega$ for some $A\in [\omega]^\omega$.%, since existence of such an array would imply existence of a bad array as in the definition.% We also note that bqo implies wqo, and we will not use wqo again in this paper.
\end{remark}

The following is a well-known Ramsey-theoretic result due to Galvin and Prikry.
\begin{thm}[Galvin, Prikry \cite{GalvinPrikry}]\label{Thm:GalvinPrikry}
Given $X\in [\omega]^\omega$ and a Borel set $B$ in $[X]^\omega$, there exists $A\in [X]^\omega$ such that either $[A]^\omega\subseteq B$ or $[A]^\omega \cap B=\emptyset$.
\end{thm}
\begin{proof}
See \cite{GalvinPrikry} or \cite{simpson}.
\end{proof}
\begin{thm}[Nash-Williams \cite{NWInfTrees}]\label{Thm:BadOrPerfect}
If $f$ is a $Q$-array, then there is $A\in [\omega]^\omega$ such that $f\restriction [A]^\omega$ is either bad or perfect.
\end{thm}
\begin{proof}
Let $B=\{X\in[\omega]^\omega\mid f(X)\leq f(X\setminus \{\min X\})\}$. If $B$ is Borel, then by Theorem \ref{Thm:GalvinPrikry} we will be done. %Since $f$ is continuous we have that for each $X\in [\omega]^\omega$ there is a non-empty finite initial segment $s_X\sis X$ such that $\forall Y\in [\omega]^\omega$ with $s_X\is Y$ we have $f(Y)=f(X)$.
%
%Pick an arbitrary $X\in B$ and let $Y\in [\omega]^\omega$ be such that $s_X\cup s_{X\setminus \{\min X\}}\sis Y$. Then $s_X\sis Y$ and $s_{X\setminus \{\min X\}}\sis Y\setminus \{\min Y\}$. Hence $f(X)=f(Y)$ and $f(X\setminus \{\min X\})=f(Y\setminus \{\min Y\})$, which means that $Y\in B$. So we have proved that for every $X\in B$ there is a finite non-empty initial segment $t=s_X\cup s_{X\setminus \{\min X\}}$ of $X$ such that if $t\sis Y$ then $Y\in B$. Hence $B$ is open, and therefore Borel.
Let $S:[\omega]^\omega\rightarrow [\omega]^\omega$ be the function $S(X)=X\setminus \{\min X\}$ and let $g=f\times (f\circ S)$. Then $g$ is continuous, since $f$ and $S$ are continuous. We also have that $B=g^{-1}(\leq)$, considering the relation $\leq$ as a subset of the discrete space $ Q\times Q$. Therefore $B$ is open and we are done.
\end{proof}
\begin{defn}
Given two sets $x$ and $y$ we write $x\sqcup y$ for the disjoint union of $x$ and $y$. And given a set $X$ of sets, we define $\bigsqcup_{x\in X}x$ as the disjoint union of the sets in $X$.
\end{defn}
\begin{defn}Let $Q_0$ and $Q_1$ be quasi-orders, we define new quasi-orders:
\begin{itemize}
	\item $Q_0\cup Q_1=(Q_0\sqcup Q_1,\leq)$ where for $p,q\in Q_0\cup Q_1$, we have $p\leq q$ iff both $p$ and $q$ are in the same $Q_i$ for $i\in \{0,1\}$, and $p\leq_{Q_i}q$;

	\item $Q_0\times Q_1=\{\langle q_0,q_1\rangle\mid q_0\in Q_0, q_1\in Q_1\}$ where for $\langle p_0,p_1\rangle,\langle q_0,q_1\rangle\in Q_0\times Q_1$ we have $$\langle p_0,p_1\rangle\leq \langle q_0,q_1\rangle \mbox{ iff }(p_0\leq_{Q_0}q_0)\wedge(p_1\leq_{Q_1}q_1).$$
\end{itemize}
\end{defn}
\begin{thm}\label{Thm:U bqo}
If there is a bad $Q_0\cup Q_1$-array $f$ then there is $A\in[\omega]^\omega$ such that $f\restriction [A]^{\omega}$ is either a bad $Q_0$-array, or a bad $Q_1$-array.
\end{thm}
\begin{proof}
Apply Theorem \ref{Thm:GalvinPrikry} with $B=f^{-1}(Q_0)$.
\end{proof}
\begin{thm}[Nash-Williams \cite{NWInfTrees}]\label{Thm:times bqo}
If there is a bad $Q_0\times Q_1$-array $f$, then there is $A\in[\omega]^\omega$ and either a bad $Q_0$-array, or a bad $Q_1$-array $g$ with $\dom(g)=[A]^{\omega}$ and such that $g(X)$ is either the first or second component of $f(X)$ for all $X\in [A]^\omega$.
\end{thm}
\begin{proof}
Define the $Q_0$-array $f_0$ and the $Q_1$-array $f_1$ so that for every $X\in [\omega]^\omega$ we have $$f(X)=\langle f_0(X),f_1(X)\rangle.$$ 
Now apply Theorem \ref{Thm:BadOrPerfect} twice to restrict so that $f_0$ and $f_1$ are both either bad or perfect. Then either we are done or they are both perfect, which contradicts that $f$ was bad.
%By Theorem \ref{Thm:BadOrPerfect}, let $A\in[\omega]^\omega$ be such that $f_0\restriction [A]^\omega$ is bad or perfect. If it is bad then we are done, so suppose that $f_0\restriction [A]^\omega$ is perfect. Now let $B\in[A]^\omega$ be such that $f_1\restriction [B]^\omega$ is bad or perfect. Again if it is bad we are done, so assume that $f_1\restriction [B]^\omega$ is perfect.
%
%Let $X\in [B]^\omega$, we have for both $i\in\{0,1\}$ that $f_i(X)\leq f_i(X\setminus \{\min X\})$. Therefore we have that $f(X)\leq f(X\setminus \{\min X\})$ so $f$ is perfect, which contradicts that $f$ was bad.
\end{proof}

\subsection{Concrete categories}

Usually we will be interested in quasi-ordering classes of partial orders under embeddability, however we can keep the results more general with no extra difficulty by considering the notion of a \emph{concrete category}. The idea is to add a little more meat to the notion of a quasi-order, considering classes of structures quasi-ordered by existence of some kind of embedding. This allows us to generate more complicated orders by colouring the elements of these structures with a quasi-order; enforcing that embeddings must increase values of this colouring. Then we can construct complicated objects from simple objects in a ranked way by iterating this colouring process, and the notion of \emph{well-behaved} allows us to reduce back down through the ranks. We shall now formalise these notions, similar to the definitions of \cite{Kriz} and \cite{wqoforbpl}.
%\begin{defn}
%A \emph{concrete system} is a quasi-ordered class $Q$, such that for any quasi-order $A$, we assign a class $Q(A)$ with a function $J:Q(A)\rightarrow \mathcal{P}(A)$. %blah blah 
%\end{defn}

%concrete systems can be made from \emph{concrete categories} as is done in \cite{Kriz} and \cite{wqoforbpl}.
\begin{defn}\label{Defn:QOcat}
A \emph{concrete category} is a pair $\mathcal{O}=\langle \obj{\mathcal{O}},\mor{\mathcal{O}}\rangle$ such that:
\begin{enumerate}
	\item \label{item:QOcat1} each $\gamma\in \obj{\mathcal{O}}$ has an associated \emph{underlying set} $U_\gamma$;
	\item \label{item:QOcat2} for each $\gamma,\delta\in \obj{\mathcal{O}}$ there are sets of \emph{embeddings} $\homm_\mathcal{O}(\gamma,\delta)$ consisting of some functions from $U_\gamma$ to $U_\delta$;
	\item \label{item:QOcat3} $\homm_\mathcal{O}(\gamma,\gamma)$ contains the identity on $\gamma$ for any $\gamma\in \mathcal{O}$;
	\item \label{item:QOcat4} for any $\gamma,\delta,\beta\in \mathcal{O}$, if $f\in \homm_\mathcal{O}(\gamma, \delta)$ and $g\in \homm_\mathcal{O}(\delta, \beta)$ then $f\circ g\in \homm_\mathcal{O}(\gamma,\beta)$;
	\item \label{item:QOcat5} $\mor{\mathcal{O}}=\{\homm_\mathcal{O}(\gamma,\delta)\mid \gamma,\delta\in \obj{\mathcal{O}}\}$.
\end{enumerate}
Elements of $\obj{\mathcal{O}}$ are called \emph{objects} and elements of $\mathcal{O}(\gamma,\delta)$ are called $\mathcal{O}$-\emph{morphisms} or \emph{embeddings}. To simplify notation we write $\gamma\in \mathcal{O}$ for $\gamma\in \obj{\mathcal{O}}$ and equate $\gamma$ with $U_\gamma$ .
\end{defn}
\begin{remark}
Similar definitions to \ref{Defn:QOcat} appear in \cite{Kriz}, \cite{wqoforbpl} and \cite{LouveauStR}. The first two enjoy a more category theoretic description, and the last is in terms of structures and morphisms.
\end{remark}

With the following definition, concrete categories will turn into quasi-ordered sets under \emph{embeddability}; (\ref{item:QOcat3}) and (\ref{item:QOcat4}) of Definition \ref{Defn:QOcat} guaranteeing the reflexivity and transitivity properties respectively. This allows us to consider the %wqo and 
bqo properties of concrete categories.
\begin{defn}
For $\gamma,\delta \in \mathcal{O}$, we say that $$\gamma \leq_\mathcal{O} \delta\mbox{ iff }\homm_\mathcal{O}(\gamma,\delta)\neq \emptyset$$ i.e. $\gamma \leq_\mathcal{O} \delta$ iff there is an embedding from $\gamma$ to $\delta$. If $f\in \homm_\mathcal{O}(\gamma,\delta)$ then we call $f$ a \emph{witnessing embedding} of $\gamma\leq_\mathcal{O}\delta$.
\end{defn}
\begin{ex}\label{Ex:POcat}
Let $\obj{\mathcal{P}}$ be the class of partial orders. For any two partial orders $x,y$, let:
\begin{enumerate}
\item  $U_x=x$,
\item $\homm_\mathcal{P}(x,y)=\{\varphi:x\rightarrow y \mid (\forall a,b\in x), a\leq_x b \longleftrightarrow \varphi(a)\leq_y\varphi(b)\},$
\item $\mor{\mathcal{P}}=\{\homm_\mathcal{P}(\gamma,\delta)\mid \gamma,\delta\in \obj{\mathcal{P}}\}$,
\item $\mathcal{P}=\langle \obj{\mathcal{P}},\mor{\mathcal{P}}\rangle$.

\end{enumerate}
The category $\mathcal{P}$ of partial orders with embeddings is then a pragmatic example of a concrete category, and the order $\leq_\mathcal{P}$ is the usual embeddability ordering on the class of partial orders. We keep this example in mind since the majority of concrete categories used in this paper are either subclasses of $\mathcal{P}$ or are derived from $\mathcal{P}$. We note that all $\mathcal{P}$-morphisms are injective.
\end{ex}
%We will focus on partial orders under embeddability, although the theorems will be kept as general as possible, by considering concrete categories as often as possible.
\begin{defn}
Given a quasi-order $Q$ and a concrete category $\mathcal{O}$, we add \emph{colours} to $\mathcal{O}$, by defining the new concrete category $\mathcal{O}(Q)$ %=\langle \obj{\mathcal{O}(Q)},\mor{\mathcal{O}(Q)}\rangle$ 
as follows.
\begin{itemize}
	\item The objects of $\mathcal{O}(Q)$ are pairs $\langle \gamma, c\rangle$ for $\gamma\in \mathcal{O}$ and $c:\gamma\rightarrow Q$.
	
	%$Q$-\emph{coloured} elements of $\mathcal{O}$, $$\obj{\mathcal{O}(Q)}=\{\langle \gamma,c\rangle\mid \gamma\in \mathcal{O}, c:\gamma \rightarrow Q\}.$$

%We order $\mathcal{O}_Q$ by letting $\langle \gamma,\col_{\gamma}\rangle \leq \langle \delta,\col_\delta\rangle$ iff 
	\item For $\langle \gamma,c\rangle\in \obj{\mathcal{O}(Q)}$, we let $V_{\langle \gamma,c\rangle}=V_\gamma$, we call $c$ a $Q$-\emph{colouring} of $\gamma$ and for each $v\in \gamma$, we call $c(v)$ the \emph{colour} of $v$. To simplify notation we equate $\langle \gamma, c \rangle$ with $\gamma$, and write $\col_\gamma=c$ and $\gamma\in \mathcal{O}(Q)$.
	\item We define morphisms of $\mathcal{O}(Q)$ from $\gamma$ to $\delta$ to be embeddings $\varphi:\gamma\rightarrow \delta$ such that 
$$\col_\gamma(x)\leq_Q \col_\delta(\varphi(x))$$
for every $x\in \gamma$.% We call such embeddings \emph{colour preserving}.
%	\item We let
%$\mor{\mathcal{O}(Q)}=\{\mathcal{O}(Q)(\gamma,\delta)\mid \gamma,\delta\in \mathcal{O}_Q\}.$
\end{itemize}
 	Given $\gamma\in \mathcal{O}(Q)$ and $v\in \gamma$, we will sometimes write $\col(v)$, to be read as `the colour of $v$', in place of $\col_\gamma(v)$ when this is unambiguous.

\end{defn}

%
%%%%%%%%%%%%%PICTURE?

%We note that if $Q$ is a concrete category then $Q(A)$ is also a concrete category since it has well-defined objects and morphisms. We then have that $Q(A)(B)=Q(A\times B)$.
We are now able to define the bqo preservation properties mentioned in Section \ref{Section:Intro}, allowing us to pass from bad $\mathcal{O}(Q)$-arrays to bad $Q$-arrays.

\begin{defn}\label{Defn:Preservesbqo}
Let $\mathcal{O}$ be a concrete category, then $\mathcal{O}$ \emph{preserves bqo} iff $$Q\mbox{ is a bqo }\longrightarrow \mathcal{O}(Q)\mbox{ is a bqo.}$$
\end{defn}

Unfortunately, this simple definition fails to be particularly useful. Given a bad $\mathcal{O}(Q)$-array, preservation of bqo ensures the existence of a bad $Q$-array, but no link between these two arrays is guaranteed. The following definition remedies this situation and is extremely important for the rest of the paper.
\begin{defn}\label{Defn:WellBehaved}
Let $\mathcal{O}$ be a concrete category, then $\mathcal{O}$ is \emph{well-behaved} iff for any quasi-order $Q$ and any bad array $f:[\omega]^\omega \rightarrow \mathcal{O}(Q),$ there is an $M\in [\omega]^\omega$ and a bad array $$g:[M]^\omega\rightarrow Q$$ such that for all $X\in [M]^\omega$ there is some $v\in f(X)$ with $$g(X)=\col_{f(X)}(v).$$
We call $g$ a \emph{witnessing bad array} for $f$.
\end{defn}
Warning: this notion of well-behaved is the same as from \cite{Kriz}; it is different from the definition of well-behaved that appears in \cite{wqoforbpl} which is in fact equivalent to Louveau and Saint-Raymond's notion of reflecting bad arrays \cite{LouveauStR}.
%$Q$ coloured partial orders are useful because they allow us to construct more complex bqo orders. Actually, we will always use the stronger notion of well-behaved.
\begin{lemma}\label{Lemma:FinitePosWB}
Let $\mathbb{P}$ be a finite set of finite partial orders, then $\mathbb{P}$ is well-behaved.
\end{lemma}
\begin{proof}
Let $Q$ be an arbitrary quasi-order and let $f$ be a bad $\mathbb{P}(Q)$-array. Then since $\mathbb{P}$ is finite, we can repeatedly apply the Galvin and Prikry Theorem \ref{Thm:GalvinPrikry} to find $A\in [\omega]^\omega$ such that for each $X,Y\in [A]^\omega$, we have that $f(X)$ and $f(Y)$ have the same underlying finite partial order $P$. Then applying Theorem \ref{Thm:BadOrPerfect} at most $|P|$ many times, restrict in turn so that the colours of each point of $f(X)$ give either a bad array or a perfect array. They cannot all be perfect otherwise $f$ would also be perfect on some restriction to an infinite set. Therefore one of these arrays is bad, and this is clearly a witnessing array for $f$.
\end{proof}

\begin{prop}\label{Prop:WB->Preserves}
$\mathcal{O}$ is well-behaved $\longrightarrow$ $\mathcal{O}$ preserves bqo $\longrightarrow \mathcal{O}$ is bqo.
\end{prop}
\begin{proof}
If $\mathcal{O}$ is well-behaved then given a bad $\mathcal{O}(Q)$-array $f$, we have a bad $Q$-array. If $Q$ were bqo this would give a contradiction and hence there is no such bad array $f$.

Now let $1=\{0\}$ be the singleton quasi-order, clearly then $1$ is bqo. Thus if $\mathcal{O}$ preserves bqo then $\mathcal{O}(1)$ is bqo. Clearly $\mathcal{O}(1)$ is isomorphic to $\mathcal{O}$, therefore $\mathcal{O}$ is also bqo.%Suppose there were a bad $\mathcal{O}$-array $g$, then for $X\in [\omega]^\omega$ define $\col_{g(X)}(v)=0$ for all $v\in g(X)$. Then let $g'(X)=\langle g(X),\col_{g(X)}\rangle$, so we have that if $g(X)\not \leq g(X)$ then $g'(X)\not \leq g'(X)$. But then $g'$ is a bad $Q(1)$-array, which is a contradiction since $Q(1)$ was bqo.
\end{proof}
\begin{remark}Note that the converse $\mathcal{O}$ is bqo $\longrightarrow$ $\mathcal{O}$ preserves bqo does not hold. For a counter example let $Z$ be the partial order consisting of points $0_n$ and $1_n$ for $n\in \omega$; ordered so that for $a,b\in Z$, we have $a\leq b$ iff there is some $n\in \omega$ such that $a\in\{0_n,0_{n+1}\}$ and $b=1_n$. Then $\{Z\}$ is clearly bqo but it does not preserve bqo.

It is not known whether or not the other converse holds, i.e. is it the case that $$\mathcal{O}\mbox{ preserves bqo }\longrightarrow \mathcal{O}\mbox{ is well-behaved}?$$ This is an interesting technical question, which was asked by Thomas in \cite{wqoforbpl}.
\end{remark}
%Preserving bqos is an important definition historically, however it will not be used in this paper from now on, we opt instead for well-behaved. 

\begin{defn}\label{Defn:Colours^q}
%=\{\langle \gamma,\col_\gamma\rangle\in Q(A)\}=
We give some notation for the set of all of possible coloured copies of some $\gamma\in \obj{\mathcal{O}}$. Given a concrete category $\mathcal{O}$, a quasi-order $Q$ and $\gamma\in \obj{\mathcal{O}}$ we define $$Q^\gamma=\{\langle \gamma,c\rangle\mid c:\gamma \rightarrow Q\}\subseteq \mathcal{O}(Q).$$
%For $\delta \subseteq \gamma$, and $\gamma_0=\langle \gamma,c\rangle \in Q^\gamma$, define $$\gamma_0 \restriction \delta =\langle \delta,c \restriction \delta\rangle\in Q^{\delta}.$$
If $\gamma_0,\gamma_1\in \mathcal{O}(Q)$ are such that there is some $\gamma\in \mathcal{O}$ with $\gamma_0,\gamma_1\in Q^\gamma$, then we say that $\gamma_0$ and $\gamma_1$ \emph{have the same structure}.
\end{defn}

\subsection{Partial orders}
\begin{defn}
We define $\card$ as the class of cardinals, $\On$ as the class of ordinals and $\On^*=\{\alpha^*:\alpha\in \On\}$, where $\alpha^*$ is a reversed copy of $\alpha$ for every $\alpha\in \On$.
\end{defn}
\begin{thm}[Nash-Williams \cite{nashwilliamsseqs}]\label{Thm:OnWB}
$\On$ is well-behaved.
\end{thm}
\begin{proof}
See \cite{simpson,nashwilliamsseqs}.
\end{proof}
\begin{defn}
If $P$ is a partial order, a \emph{chain} of $P$ is a subset with no incomparable elements. An \emph{antichain} of $P$ is a pairwise incomparable subset.
\end{defn}
\begin{defn}
We let $1=\{0\}$ be the partial order consisting of a single point. For $\kappa\in \card$ we let $\AC{\kappa}$ be the antichain of size $\kappa$. For $n\in \omega$ we let $\CH{n}$ be the chain of length $n$.
\end{defn}
We will now define some notation for traversing partial orders.
\begin{defn}\label{Defn:TraversingPOs}
Let $P$ be a partial order and $x\in P$, we define: 
$$\up x=\{y\in P\mid y\leq x\},\hspace{10pt}
\down x=\{y\in P\mid y\geq x\},$$
$$\ups x=\{y\in P\mid y< x\},\hspace{10pt}
\downs x=\{y\in P\mid y> x\}.$$
For $x,y\in P$, if it exists, we define the meet $x\wedge y$ to be the supremum of $\up x\cap \up y$.
\end{defn}
\begin{defn}
Let $P$ be a partial order and $P'\subseteq P$. Then:
\begin{itemize}
	\item we call $P'$ $\up$-closed if $(\forall x\in P')$, $\up x=\{y\in P\mid y\leq x\} \subseteq P'$,
	\item we call $P'$ $\down$-closed if $(\forall x\in P')$, $\down x=\{y\in P\mid y\geq x\} \subseteq P'$.
\end{itemize}

\begin{defn}\label{Defn:Nfree}
We define the partial order $N=\{0,1,2,3\}$ as follows. For $a,b\in N$ we let $a<b$ iff $a=1$ and $b\in \{0,2\}$ or $a=3$ and $b=2$, %Hence $N$ is the partial order on $4$ points that looks like the letter N 
(see Figure \ref{Fig:N}). A partial order is called \emph{$N$-free} if it contains no subset isomorphic to $N$.
\end{defn}
\begin{figure}
\centering%\vspace{-10pt}
\begin{tikzpicture}[thick]
\draw [->-] (0,0) -- (0,1);
\draw [-<-] (0,1) -- (1,0);
\draw [->-] (1,0) -- (1,1);
\draw [fill] (0,0) circle [radius=0.06];
\draw [fill] (0,1) circle [radius=0.06];
\draw [fill] (1,0) circle [radius=0.06];
\draw [fill] (1,1) circle [radius=0.06];
\node [left] at (0,0) {$0$};
\node [left] at (0,1) {$1$};
\node [right] at (1,0) {$2$};
\node [right] at (1,1) {$3$};
\end{tikzpicture}
\caption{The partial order $N$.}\label{Fig:N}
\label{Fig:Antichain}
\end{figure}
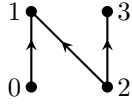
\end{defn}
\begin{defn}\label{Defn:LOBasics}
\begin{itemize}
	\item A \emph{linear order} is a partial order $L$ with no incomparable elements.
	\item A linear order $L$ is \emph{well-founded} if it has no infinite descending sequence.
	\item A linear order $L$ is \emph{scattered} if $\mathbb{Q}\not \leq L$.
	\item A linear order $L$ is \emph{$\sigma$-scattered} iff $L$ can be partitioned into countably many scattered linear orders.
	\item We denote the class of all linear orders as $\mathscr{L}$.
	\item We denote the class of scattered linear orders as $\scat$.
	\item We denote the class of $\sigma$-scattered linear orders as $\sscat$.
	\item We denote the class of countable linear orders as $\ctbl$.

\end{itemize}
\end{defn}
%\begin{defn}
%Let $r$ be a linear order, $j\in r$ and $L=\langle X_i:i\in r\rangle$ be a linearly ordered sequence of sets $X_i$ for $i\in r$. Then we call $X_j$ the \emph{$j$th element} of $L$. We also define $\ot(L)=r$.
%\end{defn}
\begin{defn}
Given $r,r'\in \mathscr{L}$ and linear sequences $k=\langle k_i:i\in r\rangle$ and $k'=\langle k'_j:j\in r'\rangle$, we denote by $\is$ the initial segment relation, and $\sis$ the strict initial segment relation. That is $$k\is k'\mbox{ iff }k=k'\mbox{ or } k=\langle k'_i:i<j\rangle \mbox{ or } k=\langle k'_i:i\leq j\rangle\mbox{ for some } j\in r'$$
and
$k\sis k'\mbox{ iff } k\is k'\mbox{ and } k\neq k'.$
We denote by $k\con k'$ the concatenation of $k$ and $k'$. We also define $\ot(k)=r$, set $k'\setminus k=\langle k'_i:i\in r'\setminus r\rangle$, and call $k_i$ the \emph{$i$th element} of $k$.
\end{defn}
\begin{defn}\label{Defn:Sums}
Let $P$ be a partial order, and for each $p\in P$, let $P_p$ be a partial order. We define the $P$-sum of the $P_p$ denoted by  $\sum_{p\in P}P_p$ as the set $\bigsqcup_{p\in P}P_p$
ordered by letting $a\leq b$ iff 
\begin{itemize}
	\item there is some $p\in P$ such that $a,b\in P_p$ and $a\leq_{P_p}b$, or
	\item there are $p,q\in P$ such that $a\in P_p$, $b\in P_q$ and $p<_{P} q$.
\end{itemize}
\end{defn}

\begin{defn}\label{Defn:Lclosure}
If $\mathbb{L}$ is a class of linear orders, we define $\clos{\mathbb{L}}$ as the least class containing $\mathbb{L}\cup \omega$ and closed under $L$-sums for all $L\in \clos{\mathbb{L}}$.
\end{defn}
\begin{thm}[Hausdorff \cite{Hausdorff}]\label{Thm:Hausdorff}
If $\mathbb{L}=\On\cup\On^*$ then $\clos{\mathbb{L}}=\scat$.
\end{thm}
\begin{proof}
See \cite{Hausdorff, simpson}.
\end{proof}
\begin{defn}\label{Defn:LTree} A partial order $T$ is called a \emph{tree} iff $(\forall t\in T)$, $\up t$ is a well-founded linear order. If $\mathbb{L}$ is a class of linear orders, we call $T$ a $\mathbb{L}$-tree iff every chain of $T$ has order type in $\mathbb{L}$ and for every $x,y\in T$, we have $\up x$ is a linear order and the meet $x\wedge y$ exists.
\end{defn}

Note that $\On$-trees are simply trees, and $\mathscr{L}$-trees are the most general `tree-like' partial orders.%\footnote{Some authors call $\mathscr{L}$-trees \emph{forests}.}

\begin{remark}
We make explicit the definition of $\mathbb{L}$-tree, in order to clarify because in the literature on bqo theory, the term `tree' has varied quite significantly. Indeed, the historical time line of bqo results for trees in the authors' terminology is as follows: Nash-Williams proved that the class of \emph{all} trees is bqo \cite{NWInfTrees}, Laver proved that the class of countable increasing unions of trees that do not embed $2^{<\omega}$ is bqo \cite{LaverClassOfTrees}, then Corominas proved that the class of all countable trees is bqo \cite{Corominas}. This is perplexing since each successive breakthrough seems to be a subclass of the previous! However the differences become clear when we use this notation. Nash-Williams proved that a class of $\omega+1$-trees\footnote{Here $\omega+1$ is the set of its predecessors.} is bqo, Laver proved that a class of $\On$-trees are bqo, and Corominas proved that a class of $\ctbl$-trees is bqo.%\footnote{he actually proved slightly more?}
\end{remark}
\begin{defn}\label{Defn:TreeHeight}\label{Defn:TreeBasics}
For a tree $T$ we call $\alpha\in \On$ the \emph{height} of $T$ %and write $\height(T)=\alpha$, 
iff $\alpha=\sup_{x\in T}\{\ot(\up x)\}$.
\end{defn}
%\begin{defn}\label{Defn:TreeBasics}
%If $T$ is an $\mathscr{L}$-tree, then
%\begin{itemize}
	%\item 
	%$t\in T$ is called a \emph{leaf} of $T$ if there is no $t'\in T$ such that $t'>t$;
	%\item $B\subseteq T$ is called a \emph{branch} of $T$ if $\exists t\in T$ such that $\down t=B$.
%\end{itemize}
%\end{defn}

\begin{defn}\label{Defn:WBScatteredTrees}
Let $\mathbb{L}$ be a class of linear orders and $T$ be an $\mathscr{L}$-tree, then we define as follows:
\begin{itemize}
	\item $t\in T$ is called a \emph{leaf} of $T$ if there is no $t'\in T$ such that $t'>t$.
	\item $T$ is \emph{rooted} iff $T$ has a minimal element, denoted $\Troot{T}$.
	\item $T$ is \emph{chain-finite} iff every chain of $T$ is finite.
\end{itemize}
\end{defn}
\begin{defn}
Given a rooted chain-finite tree $T$, and some $t\in T$ we define inductively\footnote{For the base case we have that $\sup (\emptyset)=0$.} $$\Trank(t)=\sup\{\Trank(s)+1\mid t<_Ts\}.$$ We then define the tree rank of $T$ as $\Trank(T)=\Trank(\Troot{T})$.
\end{defn}

\section{Operator construction}\label{Section:OpConstruction}

In this section we give the definitions required to translate more complicated structured trees (i.e. not necessarily chain-finite) into an operator algebra construction similar to Pouzet's \cite{PouzetApps}.

First we must give the parameters of our construction. We will always let: 
\begin{itemize}
	\item $\CC$ be concrete category;
	\item $\AAA$ be a subset of $\CC$;
	\item $\PP$ be a concrete category;
	\item $\MM$ be a quasi-ordered class of functions $f$ with range in $\CC$;
	\item $\RR$ be a non-empty class of linear orders that is closed under taking non-empty subsets;
	\item $\OA$ denote the whole system $\langle \CC,\AAA,\PP,\MM,\RR\rangle$.
\end{itemize}
Intuitively, $\CC$ is going to be the class of objects for which we will be constructing a bqo subclass; $\AAA$ will be a class of `simple' objects that we will start the construction from; $\PP$ is a class of possible \emph{arities} which we will use to generalise the notion of a multivariate function; $\MM$ is a class of functions which we will be applying to elements of $\CC$ in order to construct more complex elements of $\CC$; and $\RR$ is a class of linear orders which we will allow iteration of functions from $\MM$ over. We will keep these standard symbols when using this construction.

We restrict our attention to such $\CC, \MM, \PP$ and $\RR$, so that for every $f\in \MM$ there is some $\arity(f)\in \PP$ and $\barity{f}\subseteq \arity(f)$ with %$$\AAA^{\arity(f)}\subseteq \dom(f)\subseteq \CC^{\arity(f)}.$$ 
$$\dom(f)=\{a\in \CC^{\arity(f)}\mid (\forall i\in \barity{f}), \col_a(i)\in \AAA\}.$$
(Here $\CC^{\arity(f)}$ is as from Definition \ref{Defn:Colours^q}.) We call $\arity(f)$ the \emph{arity} of $f$. We think of the functions of $\MM$ as having arguments structured by $\arity(f)$. For example, if $\arity(f)$ is a finite linear order of length $n$ then the arguments of $f$ are linearly ordered and $f$ has the form $f(x_1,...,x_n)$; if $\arity(f)$ were an antichain, then the order on the $x_i$ would not matter; and if $\arity(f)$ was the binary tree of height $2$ then we think of $f$ having form \begin{center}
\begin{tikzpicture}
\node at (-1.5,0.2) {$f$};
\node at (-1.1,0.2) {$\Big ($};
\node at (0,0.4) {$x$};
\node at (-0.5,0) {$y$};
\node at (0.5,0) {$z$};
\node at (1.1,0.2) {$\Big )$};
\node [right] at (1.2,-0.05) {.};
\end{tikzpicture}
\end{center} We do this because for some functions it will be more convenient to think of the arguments arranged in some general partial order (particularly the sums of Definition \ref{Defn:Sums}). We include $\barity{f}$ in the definition, since in Section \ref{Section:StructuredRTrees} it will be more convenient to only allow elements of $\AAA$ into some arguments of our functions. All other constructions used in this paper will have $\barity{f}=\emptyset$ for every $f\in \MM$, in which case $\dom(f)=\CC^{\arity(f)}$.

\begin{ex}
Let $Q$ be a quasi-order and $\mathbb{L}$ be some class of linear orders. Now we let
\begin{itemize}
	\item $\CC=\mathscr{L}(Q)$ be the class of all $Q$-coloured linear orders;
	\item $\AAA=Q^1$, i.e. single points coloured by elements of $Q$;
	\item $\PP=\mathbb{L}\cup \omega$;
	\item $\MM$ be the set of $L$-sums for all $L\in\PP$, as defined in Definition \ref{Defn:Sums}, inheriting colours;
	\item $\RR=\{1\}$.
\end{itemize}
We order $\MM$ by letting $\sum_A\leq_\MM \sum_B$ iff $A\leq B$. For $\sum_L\in \MM$ we have $\arity(\sum_L)=L$ and $\barity{\sum_L}=\emptyset$. Now define $\OA_{\mathbb{L}}^Q=\langle \CC,\AAA,\PP,\MM,\RR\rangle$. This relatively simple example will construct the class $\clos{\mathbb{L}}(Q)$, which will ultimately allow us to prove that if $\mathbb{L}$ is well-behaved, then $\clos{\mathbb{L}}$ is well-behaved.
\end{ex}
\begin{ex}\label{Ex:POCONSTRUCTION}
Let $Q$ be a quasi-order; $Q'$ be $Q$ with an added minimal element $-\infty$; let $\mathbb{L}$ be some class of non-empty linear orders; and let $\mathbb{P}$ be some class of partial orders.
Now we let 
\begin{itemize}
	\item $\CC$ be the class all of $Q'$-coloured partial orders;
	\item $\AAA=Q'^1$;
	\item $\PP=\mathbb{P}$;
	\item $\MM$ be the set of $P$-sums for all $P\in \mathbb{P}$, inheriting colours;
	\item $\RR=\clos{\mathbb{L}}$.
\end{itemize}
We order $\MM$ by letting $\sum_A\leq_\MM \sum_B$ iff $A\leq B$. Now define $\OA_{\mathbb{L},\mathbb{P}}^Q=\langle \CC,\AAA,\PP,\MM,\RR\rangle$. This example will construct our class of $Q'$-coloured partial orders, that we will characterise in Section \ref{Section:PO}.
\end{ex}

We wish to generalise Pouzet's operator algebra construction in two ways. The first way will allow us to iterate functions over a linear order from $\RR$, and the second way will allow us to take countable limits.
\subsection{Iterating over $\RR$}

First we will define what is required for iteration. This will allow us to represent complicated functions in terms of simpler ones. For a basic example, suppose we would like to construct the $\omega$-sum as from Definition \ref{Defn:Sums}. Given two linear orders $L_0$ and $L_1$, we can define the simple function $+$, so that $L_0+L_1$ is a copy of $L_0$ followed by a copy of $L_1$. Now suppose we iterate this function; we can easily define $L_0+L_1+L_2$ and $L_0+L_1+L_2+L_3$ and so on. This is easily done finitely many times (which is what Pouzet was doing when he used chain-finite trees \cite{PouzetApps}). But it is possible to iterate $+$ over a more complex linear order. Naturally, its $\omega$ iteration would be $\bigcup_{i\in \omega}L_i$ ordered by $a<b$ if $a\in L_i$, $b\in L_j$, $i<j$ or $i=j$ and $a<_{L_i}b$ (i.e. the $\omega$-sum). We need not stop there though, we could define this iteration over a larger linear order, e.g. its $\mathbb{Q}$ iteration could be defined similarly.

Here $+$ is really the function $\sum_{\CH{2}}$. We think of this function as having arguments arranged in arity $\CH{2}$. At each successive stage of the iteration we apply the next function into argument corresponding to the larger of the two points in $\CH{2}$.\footnote{Note that since we are dealing with multivariate functions, we are required to distinguish an argument in order to know in which position to compose further sums inside.} If we repeatedly compose $\sum_{\CH{2}}$ in this way, we get the functions $\sum_{\CH{3}}$, $\sum_{\CH{4}}$, and so on. If we allowed ourselves to compose over a linear order we could get $\sum_{\omega}$ or $\sum_\mathbb{Q}$ as in the previous paragraph. So in general we want to be able to turn linearly ordered lists of functions, each with a distinguished argument, into a new function that acts as their composition. This gives rise to the definition of a \emph{composition sequence}, which will be a list of functions $f_i$ indexed by a linear order $r$, each with a distinguished argument $s_i$.
\begin{defn}\label{Defn:Iterable1}
We call $\eta$ a \emph{composition sequence} iff $\eta=\langle \langle f_i,s_i\rangle:i\in r\rangle$ where $r\in \RR$, and for all $i\in r$, we have $f_i\in \MM$ and $s_i\in \arity(f_i)$ such that if $i\neq \max(r)$ then $s_i\notin \barity{f}$.

We call $r$ the \emph{length} of the sequence $\eta$.
For $i\in r$ let $$a^\eta_i=\left\{\begin{array}{lcl}
\arity(f_i)\setminus \{s_i\} & \mbox{if} &i \neq \max (r) \\
\arity(f_i) &\mbox{if}& i = \max (r)
\end{array}\right. .$$ 
%When it is clear which $\eta$ we are referring to, we write simply $a_i$ for $a^\eta_i$. 
We also define $A^\eta=\bigsqcup_{i\in r} a^\eta_i$ and $B^\eta=\bigsqcup_{i\in r}\barity{f_i}$ so that $B^\eta\subseteq A^\eta$. For any $j\in r$, let 
$$\eta^-_j=\langle \langle f_i,s_i\rangle:i\leq j\rangle$$ 
$$\eta^+_j=\langle \langle f_i,s_i\rangle:i> j\rangle.$$ 
\end{defn}
For each composition sequence $\eta$ we will define a function $f^\eta$ that will act as the composition of the functions $f_i$ (in order type $r$). We want $f_j$ to be applied to the composition of all of the $f_i$ (for $i>j$) in the argument $s_j$. %(Since these will usually be multivariate functions, if we want to compose functions we need to know which argument we will be composing within.) 
\begin{remark}
We require that for $i\in r$, if $i\neq \max(r)$ then $s_i\notin \barity{f}$, because otherwise it could be that the composition of the $f_j$ for $j>i$ is not allowed into the domain of $f_i$ in position $s_i$.
\end{remark}
\begin{ex}\label{Ex:Heta}
Suppose we want to define iteration of the sums of Definition \ref{Defn:Sums} over a general linear order. 
%(which will be crucial for Section \ref{Section:PO}).
%So suppose we are given a composition sequence 
Let $\eta=\langle \langle f_i,s_i\rangle:i\in r\rangle$ be a composition sequence, %where $r$ is a linear order, and 
and for each $i\in r$, suppose $\arity(f_i)$ is a partial order, $f_i$ is the $\arity(f_i)$-sum, and $s_i\in \arity(f_i)$. We turn this composition sequence into a new function $f^\eta$ that acts as the composition of the $f_i$ as follows.
First define the partial order $H_\eta$ as the set $\bigsqcup_{i\in r} a_i^\eta$ ordered so that for $u,v\in H_\eta$ we let $u<v$ iff $u\in a_i^\eta$, $v\in a_j^\eta$ one of the following occurs:
\begin{itemize}
	\item $i=j$ and $u<_{\arity(f_i)}v$;
	\item $i<j$ and $u<_{\arity(f_i)}s_i$;
	\item $i>j$ and $v>_{\arity(f_j)}s_j$.
\end{itemize}
(See Figure \ref{Fig:Heta}.) We can then define our composition $f^\eta$ to be the $H_\eta$-sum. Notice that for finite $r$, this is equivalent to the usual finite composition of sums, each composed in the argument $s_i$.% We let $f^{\langle \rangle}=\Id_\CC$.
\end{ex}

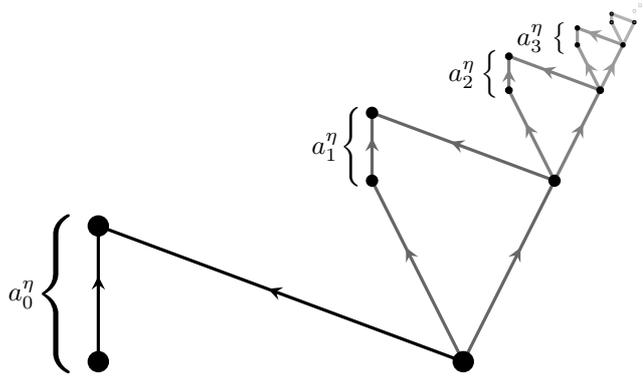
\begin{figure}
\begin{tikzpicture}[very thick, scale=0.3]
%\draw (0,0)--(0,16)--(16,0)--(16,8);

\draw[->-] (0,0)--(0,6);
\draw[-<-] (0,6)--(16,0);
%\draw[->-] (16,0)--(16,8);
\draw[g0, ->-] (16,0)--(12,8);
\draw[g0, ->-] (16,0)--(20,8);
\draw [fill] (0,0) circle [radius=0.4];
\draw [fill] (0,6) circle [radius=0.4];
\draw [fill] (16,0) circle [radius=0.4];
\node [left, scale=2] at (0,3) {$\Bigg \{$};
\node [left] at (-2.3,3) {$a^\eta_0$};

\draw [g0, ->-](12,8)-- (12,11);
\draw [g0, -<-](12,11)--(20,8);
%\draw (20,8)--(20,12);
\draw [g1, ->-] (20,8)--(18,12);
\draw [g1, ->-] (20,8)--(22,12);
\draw [ fill] (12,8) circle [radius=0.2];
\draw [ fill] (12,11) circle [radius=0.2];
\draw [ fill] (20,8) circle [radius=0.2];
\node [left] at (12,9.5) {$\Bigg \{$};
\node [left] at (11,9.5) {$a^\eta_1$};

\draw [g1, ->-] (18,12)--(18,13.5);
\draw [g1, -<-] (18,13.5)--(22,12);
%\draw (22,12)--(22,14);
\draw [g2, ->-] (22,12)--(21,14);
\draw [g2, ->-] (22,12)--(23,14);
\draw [ fill] (18,12) circle [radius=0.1];
\draw [ fill] (18,13.5) circle [radius=0.1];
\draw [ fill] (22,12) circle [radius=0.1];
\node [left] at (18,12.7) {$\Big \{$};
\node [left] at (17,12.7) {$a^\eta_2$};

\draw [g2] (21,14)--(21,14.75);
\draw [g2, -<-] (21,14.75)--(23,14);
%\draw (23,14)--(23,15);
\draw [g3] (23,14)--(22.5,15);
\draw [g3] (23,14)--(23.5,15);
\draw [ fill] (21,14) circle [radius=0.05];
\draw [ fill] (21,14.75) circle [radius=0.05];
\draw [ fill] (23,14) circle [radius=0.05];
\node [left] at (21,14.3) {$\big \{$};
\node [left] at (20,14.3) {$a^\eta_3$};

\draw [g3] (22.5,15)--(22.5,15.375);
\draw [g3] (22.5,15.375)--(23.5,15);
%\draw [g3](23.5,15)--(23.5,15.5);
\draw [ fill] (22.5,15) circle [radius=0.025];
\draw [ fill] (22.5,15.375) circle [radius=0.025];
\draw [ fill] (23.5,15) circle [radius=0.025];

\draw [g3, fill] (24,16) circle [radius=0.01];
\draw [g3, fill] (23.75,15.75) circle [radius=0.01];
\draw [g3, fill] (23.5,15.5) circle [radius=0.01];

\end{tikzpicture}
\caption{The partial order $H_\eta$, for $\eta=\langle \langle \sum_N,3\rangle :i\in \omega \rangle$.}%, and for each $i\in \omega$, $f_i=$ and each $s_i=3\in N$.}
\label{Fig:Heta}
\end{figure}

The next definition allows us, in a general setting, to turn composition sequences into functions that will act as the linear composition of the functions in the composition sequence.

\begin{defn}\label{Defn:Riterable}
$\MM$ is called $\RR$-\emph{iterable} if we distinguish a class $\MM^\RR$ consisting of \emph{composition functions} $f^\eta$, for every composition sequence $\eta=\langle \langle f_i, s_i\rangle:i\in r\rangle$ where $$f^\eta:\{\langle q_u:u\in A^\eta\rangle\in \CC^{A^\eta} \mid (\forall u\in B^\eta),q_u\in \AAA\}\longrightarrow \CC,$$
and $\MM^\RR$ satisfies the following properties.
\begin{enumerate}[(i)]
	\item \label{Item:Riterable4} For all $f_0\in \MM$ and $s_0\in \arity(f_0)$, we have $f^{\langle \langle f_0,s_0\rangle \rangle}=f_0$.% and $f^{\langle \rangle}=\Id_\CC$.	
%	\item \label{Item:Riterable1} $\dom(f^\eta)=\{\langle k_i\restriction a^\eta_i: i\in r\rangle\mid k_i\in \dom(f_i)\}.$
%$\{k: k:r\rightarrow \PP(\CC), k(i)\in \{(z_i,u\restriction z_i): (p_{m(i)},u)\in \dom(m(i))\}\}$
%Thus elements in the domain of $f^\eta$ are sequences of order type $r$ with $i$th element in $\CC^{a_i}$. 
%	\item \label{Item:Riterable2} Given $\langle k_i:i\in r\rangle\in \dom (f^\eta)$ and some $j\in r$, we let $k'_j\in\dom(f_j)\subseteq  \CC^{\arity(f_j)}$ be such that $k'_j\restriction a_j=k_j$ and if $j\neq\max(r)$ then %for the remaining $s\in k^\eta_j$ (ie $s$ is in position $s_j$) we have
% $$\col_{k'_j}(s_j)=f^{\eta^+_j}(\langle k_i:i>j\rangle).$$
%Now set $k'_i=k_i$ whenever $i< j$, then we have
%$$f^\eta(\langle k_i:i\in r\rangle)=f^{\eta^-_j}(\langle k'_i:i\leq j\rangle).$$
	\item \label{Item:Riterable2} Given $\langle q_u:u\in A^\eta\rangle\in \dom (f^\eta)$ and some $j\in r$, set $q_{s_j}=f^{\eta^+_j}(\langle q_u:u\in a^\eta_i, i>j\rangle)$, then $$f^\eta(\langle q_u:u\in A^\eta\rangle)=f^{\eta^-_j}(\langle q_u:u\in A^{\eta^-_j}\rangle).$$

%%%%	so that $\langle q_k:k\in A^{\eta^-_j}\rangle\in \dom(f^{\eta^-_j})$
%%%%	
%%%%	
%%%%	 $k'_j\in\dom(f_j)\subseteq  \CC^{\arity(f_j)}$ be such that $k'_j\restriction a_j=k_j$ and if $j\neq\max(r)$ then %for the remaining $s\in k^\eta_j$ (ie $s$ is in position $s_j$) we have
%%%% $$\col_{k'_j}(s_j)=f^{\eta^+_j}(\langle k_i:i>j\rangle).$$
%%%%Now set $k'_i=k_i$ whenever $i< j$, then we have
%%%%$$f^\eta(\langle k_i:i\in r\rangle)=f^{\eta^-_j}(\langle k'_i:i\leq j\rangle).$$

%	\item \label{Item:Riterable3} For each $\eta$,  $\langle k_i:i\in r\rangle\in \dom(f^\eta)$ and $j\in r$, we have $k^\eta_j\in \dom(f_j)$. 

\end{enumerate}
%For $\zeta$ an initial segment of $r$, we also define $f^{\eta\restriction \zeta}(\zeta,k\restriction \zeta)$. Let $\eta\restriction \zeta=(\zeta, m\restriction \zeta, p\restriction \zeta)$. Now, $k(i)=(z_i,u_i)$ we define $k\restriction \zeta:\zeta\rightarrow P(C)$ %(NOT the usual restriction!) 
%	so that $k\restriction \zeta(i)=(z_i,u'_i)$ where if $u_i(a)=f^\nu(r',k')$ then $u'_i(a)=f^{\nu\restriction n-1}(\min(r',n-1),k'\restriction n-1)$. The $n$'s are decreasing and so is well defined since once we get down to $1$ it is just a function in $M$ applied to a list $1$'s.
%%We will use the notation that if $k\in \dom(f)$ then $k=\langle k_i:i\in r\rangle$.
\end{defn}

So we require that when we split $\eta$ up into initial and final sections $\eta^-_j$ and $\eta^+_j$, the composition will behave as expected (i.e. $f^\eta$ is $f^{\eta^-_j}$ applied to $f^{\eta^+_j}$).
%We think of the functions of $\eta$ being composed in a line of order type $r$, to form $f^\eta$. 
%There is then a remaining argument of $f^\eta$ for each element of each of the $a^\eta_i$, $(i\in r)$, %i.e. the arguments of $f^\eta$ are
%the remaining arguments from each of the domains of the $f_i$, $(i\in r)$.
%these arguments are required to determine the value of the composition $f^\eta$, and so form its domain.
The remaining arguments of $f_i$ are then those in positions corresponding to elements of $a^\eta_i$ (see Figure \ref{Fig:feta_args}), so we consider $f^\eta$ as a multivariate function from $\CC$, which has arguments for each element of $a_i^\eta$, $(i\in r)$, and only allows values admissible into $\dom(f_i)$ (i.e. those respecting that if $u\in \barity{f_i}$ then $q_u\in \AAA$).
\begin{figure}
  
  \centering
    \begin{tikzpicture}
	\node [below right] at (1,5) {$r$};
\draw [thick] (0.1,0.5) -- (1,5);
\draw [fill] (0.8,4) circle [radius=0.06];
\node [left] at (0.8,4) {$i$};
%	\draw  (0.8,4) -- (-0.04,2.3);
	\draw  (0.8,4) -- (1.16,3.3);
	\draw  (0.8,4) -- (1.66,3.3);
	\draw  (0.8,4) -- (2.16,3.3);
	\draw  (0.8,4) -- (2.66,3.3);
	\draw [fill] (0.66,3.3) circle [radius=0.06];
	\node [left] at (0.66,3.3) {$s_i$};
	\node [rotate=90, scale=1.5] at (1.93,3.1) {$\Bigg \{$};
	\node at (1.93,2.75) {$a^\eta_i$};

%	\draw [fill] (0.46,2.3) circle [radius=0.08];
%\draw [fill] (0.6,3) circle [radius=0.08];
	\draw  (0.6,3) -- (-0.04,2.3);
	\draw  (0.6,3) -- (0.96,2.3);
	\draw  (0.6,3) -- (1.46,2.3);
	\draw  (0.6,3) -- (1.96,2.3);
	\draw  (0.6,3) -- (2.46,2.3);
	\draw  (0.6,3) -- (-0.44,2.3);
%	\draw [fill] (0.46,2.3) circle [radius=0.08];

%\draw [fill] (0.4,2) circle [radius=0.08];
	\draw  (0.4,2) -- (-0.24,1.3);
	\draw  (0.4,2) -- (0.76,1.3);
	\draw  (0.4,2) -- (1.26,1.3);
	\draw  (0.4,2) -- (1.76,1.3);
	\draw  (0.4,2) -- (2.26,1.3);
	\draw  (0.4,2) -- (2.76,1.3);
	
%%\draw [fill] (0.2,1) circle [radius=0.08];
%	\draw  (0.2,1) -- (-1.44,0.3);
%	\draw  (0.2,1) -- (-0.94,0.3);
%	\draw  (0.2,1) -- (-0.44,0.3);
%	\draw  (0.2,1) -- (0.56,0.3);
%	\draw  (0.2,1) -- (1.06,0.3);
%	\draw  (0.2,1) -- (1.56,0.3);
%%	\draw  (0.2,1) -- (2.06,0.3);
%%	\draw  (0.2,1) -- (2.56,0.3);
%\node at (-1.5,0.2) {$f$};

\end{tikzpicture}
\caption{The arrangement of the arguments of $f^\eta$.}\label{Fig:feta_args}
\end{figure}
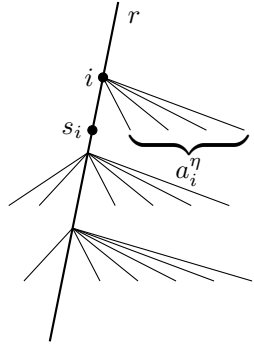

%\begin{remark}
%Given a composition sequence $\eta=\langle \langle f_i,s_i\rangle:i\in r\rangle$ and $k=\langle k_i:i\in r\rangle$, with each $k_i\in \dom(f_i)$; to simplify notation, we write simply $f^\eta(k)$ in place of $f^\eta(\langle k_i\restriction a^\eta_i:i\in r\rangle)$.
%
%\end{remark}
%\begin{defn}
%We call $\MM$ \emph{compatible with} $\CC$ iff whenever $f,g\in \MM$, $x\in \dom(f)$, $y\in \dom(g)$, $f\leq_\MM g$ and $x\leq_{\PP(\CC)} y$ we have $f(x)\leq_\CC g(y)$.
%\end{defn}
\begin{remark}
For a composition sequence $\eta$ of length $r$, notice that elements of $A^\eta$ can be indexed by $i\in r$ and $u\in a^\eta_i$, so we will sometimes write elements of $\dom(f^\eta)$ as $\langle q_{i,u}:i\in r, u\in a^\eta_i\rangle$.
\end{remark}
\begin{lemma}\label{Lemma:PORiterable}
Let $\OA=\OA_{\mathbb{L},\mathbb{P}}^Q=\langle \CC,\AAA,\PP,\MM,\RR\rangle$ for some quasi-order $Q$, some class of linear orders $\mathbb{L}$ and some class of partial orders $\mathbb{P}$. Then $\MM$ is $\RR$-\emph{iterable}.
\end{lemma}
\begin{proof}
Given a composition sequence $\eta$, define $f^\eta=\sum_{H_\eta}$ as in Example \ref{Ex:Heta}. Consider Definition \ref{Defn:Riterable}, clearly $f^\eta$ satisfies (\ref{Item:Riterable4})% and (\ref{Item:Riterable1})
. It remains to show (\ref{Item:Riterable2}), so let $\langle q_u:u\in A^\eta\rangle\in \dom (f^\eta)$ and for some $i\in r$, let $q_{s_i}=f^{\eta^+_i}(\langle q_u:u\in a^\eta_i, j>i\rangle)$. % then $$f^\eta(\langle q_u:u\in A^\eta\rangle)=f^{\eta^-_j}(\langle q_u:u\in A^{\eta^-_j}\rangle).$$ 
%
%We have $f^\eta(\langle k_i:i\in r\rangle)=\sum_{u\in H_\eta}q_u$. %Let $H^+=f^{\eta^+_i}(\langle k_j:j>i\rangle)$, i.e. $H^+$ is the $H_{\eta^+_i}$-sum of the $P_{j,a}$ for $j>i$. 
%
We have that $H_{\eta^+_i}=\bigsqcup_{j>i}a^\eta_j$ and $H_{\eta^-_i}=\{s_i\}\sqcup\bigsqcup_{j\leq i}a^\eta_j$. So that the $H_{\eta^-_i}$-sum of $H_{\eta^+_i}$ in position $s_i$ and single points everywhere else is precisely $H_\eta$. Similarly, we see that $$f^\eta(\langle q_u:u\in A^\eta\rangle)=\sum_{u\in H_\eta}q_u=\sum_{u\in H_{\eta^-_i}}q_u=f^{\eta^-_i}(\langle q_u:u\in a^\eta_i, j\leq i\rangle).$$ So we have (\ref{Item:Riterable2}) as required.
\end{proof}
\begin{defn}
Let $\eta=\langle \langle f_i,s_i\rangle:i\in r\rangle$ and $\nu=\langle \langle f'_i,s'_i\rangle:i\in r'\rangle$ be composition sequences. We define $\eta\leq \nu$ iff there is an embedding $\varphi:r\rightarrow r'$ such that for every $i\in r$ we have $f_i\leq_\MM f'_{\varphi(i)}$ and an embedding $\varphi_i$ witnessing $\arity(f_i)\leq \arity(f'_{\varphi(i)})$, such that whenever $i\neq \max(r)$, we have $\varphi_i(s_i)=s'_{\varphi(i)}$ . If $\eta\leq \nu$ we define $\varphi_{\eta,\nu}:A^\eta\rightarrow A^\nu$ so that when $u\in a^\eta_i$ we have $\varphi_{\eta,\nu}(u)=\varphi_i(u)$.
\end{defn}
\begin{defn}
We call $f\in \MM$ \emph{extensive} if for all $q\in \CC$ and $x\in \dom(f)$ with $i\in x$ such that $\col(i)=q$, we have that $q\leq_\CC f(x)$.
\end{defn}
The idea of the next definition is to express the notion that applying `short' lists of `small' functions to `small' objects will give a smaller result than applying `long' lists of `large' functions to `large' objects.
\begin{defn}\label{Defn:InfExtensive}
We call $\OA=\langle \CC,\AAA,\PP,\MM,\RR\rangle$ \emph{infinitely extensive} iff:
\begin{itemize}
	\item $\MM$ is $\RR$-iterable;
	\item every $f\in \MM$ is extensive;
	\item for any two composition sequences $\eta\leq \nu$ and any $k=\langle q_u:u\in A^\eta\rangle\in \dom(f^\eta)$, $k'=\langle q'_u:u\in A^\nu\rangle\in \dom(f^\nu)$, if for all $u\in A^\eta$ we have embeddings $\varphi_u$ witnessing $q_u\leq q'_{\varphi_{\eta,\nu}(u)}$, then we have a corresponding embedding $\varphi_{\eta,\nu}^{k,k'}$ witnessing $f^\eta(k)\leq_\CC f^\nu(k')$.
\end{itemize}
%with $k_i\in \dom(f_i)$ and $k'_j\in \dom(g_j)$ for each $i\in r$, $j\in r'$. If there is an embedding $\varphi:r\rightarrow r'$ such that $f_i\leq_\MM g_{\varphi(i)},$
%and further if for each $i\in r$ there are embeddings $\varphi_i:\arity(f_i)\rightarrow \arity(g_{\varphi(i)})$ that witness $$k_i\leq_{\PP(\CC)} k'_{\varphi(i)}$$
%and are such that $\varphi_i(s_i)=s'_{\varphi(i)}$ whenever $i\neq\max(r)$; then we have
%$$f^\eta(k)\leq_\CC f^\nu(k').$$
%$$\langle \langle f_i,k_i\rangle : i\in r\rangle \leq \langle \langle g_i,k_i\rangle : i\in r'\rangle$$
%when considered as elements of $\RR(\MM\times \PP(\CC))$ then
\end{defn}
%\begin{defn}
%Let $\mathcal{O}$ be a concrete category and $x,y,z\in \mathcal{O}$ with $x\leq y \leq z$ and $\mu:x\rightarrow y$, $\varphi_0:x\rightarrow z$ and $\varphi_1:y\rightarrow z$ embeddings. We say that $\varphi_1$ extends $\varphi_0$ over $\mu$ iff for all $a\in x$, $$\varphi_0(a)=\varphi_1\circ \mu(a).$$
%\end{defn}
\begin{lemma}\label{Lemma:LOinfExt}
Let $\mathbb{L}$ be a class of linear orders and $Q$ be an arbitrary quasi-order, then $\OA^Q_{\mathbb{L}}$ is infinitely extensive.
\end{lemma}
\begin{proof}
Let $\OA^Q_{\mathbb{L}}=\langle \CC,\AAA,\PP,\MM,\RR\rangle$. Since $\RR=\{1\}$ it is clear that $\MM$ is $\RR$-iterable. We have that every $f\in \MM$ is extensive, because if we take the sum of some linear orders, then each of the linear orders embeds into the sum. Now suppose as in Definition \ref{Defn:InfExtensive} we have composition sequences
$$\eta=\langle \langle \sum_A,s\rangle\rangle \mbox{ and }\nu=\langle \langle \sum_B,s'\rangle\rangle$$ with $\eta\leq \nu$ and $k=\langle q_u:u\in A^\eta\rangle\in \dom(f^\eta)$, $k'=\langle q'_u:u\in A^\nu\rangle\in\dom(f^\nu)$ such that $q_u\leq q'_{\varphi_{\eta,\nu}(u)}$ for each $u\in A^\eta$, with $\varphi_u$ a witnessing embedding.

So we have that $f^\eta(k)=\sum_{u\in A}q_u$ and $f^\nu(k)=\sum_{u\in B}q'_u$. Moreover, $\varphi_{\eta,\nu}$ is an embedding from $A$ to $B$. So we define $\varphi_{\eta,\nu}^{k,k'}:f^\eta(k)\rightarrow f^\nu(k')$ so that for $a\in f^\eta(k)$, with $a\in q_u$ we have $$\varphi_{\eta,\nu}^{k,k'}(a)=\varphi_u(a)\in q'_{\varphi_{\eta,\nu}(u)}\subseteq f^\nu(k').$$
Then $\varphi_{\eta,\nu}^{k,k'}$ is clearly an embedding. Hence $\OA^Q_{\mathbb{L}}$ is infinitely extensive.
%%Now suppose that $\hat{k}=\langle \hat{q}_u:u\in A^\eta\rangle$ is such that $q_u\leq \hat{q}_u$ for all $u\in A^\eta$, with some witnessing embedding $\mu_u$. Suppose also that $\hat{q}_u\leq q'_{\varphi(\eta,\nu)}$ for each $u\in A^\eta$, with $\psi_u$ a witnessing embedding, and that $\varphi_u=\psi_u\circ \mu_u$. Then for $a\in f^\eta(k)$, with $a\in q_u$ we have $$\varphi^{k,k'}_{\eta,\nu}(a)=\varphi_u(a)=\psi_u\circ\mu_u(a)\in q'_{\varphi_{\eta,\nu}(u)}$$
%%and $\mu^{k,k'}_{\eta,\nu}(a)=\mu_u(a)\in \hat{q}_u$, so that $$\psi^{\hat{k},k'}_{\eta,\nu}\circ \mu^{k,k'}_{\eta,\nu}(a)=\psi_u\circ \mu_u(a)\in q'_{\varphi_{\eta,\nu}(u)}.$$
%%Thus we have verified the conditions of Definition \ref{Defn:Extendible}, and $\OA$ is extendible.
%%%$\varphi$ witnesses $\hat{k}\leq_{\PP(\CC)} k'$, in such a way that each embedding given by $\hat{x}_{i}\leq y_{\varphi(i)}$ extends the embedding given by $x_{i}\leq y_{\varphi(i)}$ that we had before. Then extending $\varphi$ to $\hat{\varphi}$ using the embeddings from $\hat{x}_{i}\leq y_{\varphi(i)}$ will make $\hat{\varphi}$ an extension of $\varphi'$. Thus $\OA$ is extendible.
\end{proof}

\begin{lemma}\label{Lemma:RealInfExtensive}
Suppose $\OA$ is infinitely extensive. Let $\eta$ be a composition sequence and $k=\langle q_u :u\in A^\eta\rangle\in \dom(f^\eta)$. Then for any $u\in A^\eta$, we have $q_u\leq f^\eta(k)$.
\end{lemma}
\begin{proof}
Let $\eta=\langle \langle f_i,s_i\rangle : i\in r\rangle$. Pick some $u\in A^\eta$ and let $i\in r$ be such that $u\in a^\eta_i$. Set $\eta'=\langle\langle f_i,s_i\rangle\rangle$ and $q_{s_i}=f^{\eta^+_i}(\langle q_v:v\in a^\eta_j, j>i\rangle)$. %Could simplify this by instead of using u just setting this final colour to the minimum elememnt (but then we have to say there is a minimum element at some point).
%Let $f_i$ be the $i$th member of the sequence $\eta$. Now we set $\eta'=\langle f_i\rangle$ and $z\in \dom(f_i)$ such that $z\restriction a^\eta_i=k_i$, and $\col_z(s_i)=x(u,s_i)$.
Since $f_i\in \MM$ we have $f_i$ is extensive, therefore $$q_u\leq f_i(\langle q_v:v\in \arity(f_i)\rangle)= f^{\eta'}(\langle  q_v:v\in A^{\eta'}\rangle).$$ Then by a simple application of the infinite extensivity of $\OA$, we see that $$f^{\eta'}(\langle  q_v:v\in A^{\eta'}\rangle) \leq f^{\eta^-_i}(\langle q_v:v\in A^{\eta^-_i}\rangle).$$
Now by $\RR$-iterability of $\MM$, we have 
$$f^{\eta^-_i}(\langle q_v:v\in A^{\eta^-_i}\rangle)=f^\eta(k).$$
So that $q_u\leq f^\eta(k)$, which gives the lemma.
\end{proof}

\begin{defn}\label{Defn:amrs}
When $\MM$ is $\RR$-iterable we define $\amrs\subseteq \CC$ as the smallest class containing $\AAA$ and closed under applying $f^\eta$ for any composition sequence $\eta$.
\end{defn}
%We now also fix some $\AAA\subseteq \CC$ throughout this section.

\begin{ex}\label{Ex:ScatLO}
Following Definitions \ref{Defn:LOBasics} and \ref{Defn:Sums}. Let $\CC$ be the class of linear orders; $\AAA=\{1\}\subseteq \CC$; $\PP=\On\cup \On^*$ and $\MM=\{\sum_\alpha\mid \alpha\in \PP\}$ be the class of ordinal sums and reversed ordinal sums; finally set $\RR=\{1\}$. Then $\amrs=\scat$ the class of scattered linear orders. This is precisely Theorem \ref{Thm:Hausdorff}, Hausdorff's theorem on scattered order types \cite{Hausdorff}.
\end{ex}
\begin{remark}\label{Rk:amrs} We can construct $\amrs$ level by level. Starting with $\AAA$; at successor stages applying $f^\eta$ (for every composition sequence $\eta$) to every element of the previous level; and at limit stages taking unions. This allows the definition of an ordinal rank of an element of $\amrs$. So for $x\in \amrs$, we denote by $\rank(x)$ the least $\alpha\in \On$ such that $x$ appears at level $\alpha$ of this construction.
\end{remark}
\begin{defn}
We define $\CC_\alpha=\{x\in \amrs\mid \rank(x)=\alpha\}$ and $\CC_{<\alpha}=\{x\in \amrs\mid \rank(x)<\alpha\}$. For each composition sequence $\eta$ we also define, $$\dom(f^\eta)_{<\alpha}=\{\langle q_u:u\in A^\eta\rangle\in \dom(f^\eta)\mid (\forall u\in A^\eta), \rank(q_u)<\alpha\}.$$
\end{defn}
%We now give two classical examples of possible $\amrs$.

\subsection{Limits}

Now we will define limits, which will allow us to extend the construction further. For example, when we choose our parameters so that $\amrs$ gives us scattered linear orderings (Example \ref{Ex:ScatLO}), taking limits will give us the $\sigma$-scattered linear orderings of \cite{LaverFrOTconj}. %; if we were to choose the parameters so that $\amrs$ gives the scattered trees (Example \ref{Ex:ScatTrees}), then taking limits would give the $\sigma$-scattered trees of \cite{LaverClassOfTrees}.% First we need to define the kind of sequences of which we can take limits.
The next definition will allow us to chain together many embeddings eventually allowing us to produce embeddings between limits.
\begin{defn}\label{Defn:Extendible}
We call $\OA$ \emph{extendible} iff $\OA$ is infinitely extensive and satisfies the following property. 
For any
%\begin{itemize}
	two composition sequences $\eta\leq \nu$,
	and any $k_0=\langle q^0_u:u\in A^\eta\rangle\in\dom(f^\eta)$,
	$k_1=\langle q^1_u:u\in A^\eta\rangle\in\dom(f^\eta)$,
	and $k=\langle q_u:u\in A^\nu\rangle\in \dom(f^\nu)$;
if for all $u\in A^\eta$ we have:

\begin{itemize}
	\item $q^0_u\leq q^1_u$ with $\mu_u$ a witnessing embedding; 
	\item $q^0_u\leq q_{\varphi_{\eta,\nu}(u)}$ with $\varphi_u$ a witnessing embedding; 
	\item $q^1_u\leq q_{\varphi_{\eta,\nu}(u)}$ with $\psi_u$ a witnessing embedding; 
	\item $\varphi_u=\psi_u\circ \mu_u$;
\end{itemize}	
then we have $$\varphi_{\eta,\nu}^{k_0,k}=\psi_{\eta,\nu}^{k_1,k}\circ\mu_{\eta,\eta}^{k_0,k_1}.$$
% $\eta,\nu,k,k',\varphi$ and $\varphi_i$ are as in Definition \ref{Defn:InfExtensive}. Let $\hat{k}=\langle \hat{k}_i:i\in r\rangle$ be such that for each $i\in r$ we have $\hat{k}_i\in \dom(f_i)$ so that $\hat{k}_i$ has the same structure as $k_i$. Suppose that the colours of $\hat{k}_i$ increased from of colours of $k_i$, i.e. $\forall e\in k_i$ we have $\col_{k_i}(e)\leq \col_{\hat{k}_i}(e)$. Let $\mu_i$ be the embedding from $k_i$ to $\hat{k}_i$, just by taking the identity on the structure.
%
%Suppose that for each $i\in r$ there are $\hat{\varphi}_i$ witnessing $\hat{k}_i\leq_{\PP(\CC)}k'_{\varphi(i)}$, and for each $i\in r$ and each $e\in k_i$, the embedding witnessing $$\col_{\hat{k}_i}(e)\leq \col_{k'_{\varphi(i)}}(\hat{\varphi}_i(e))$$ derived from $\hat{\varphi}_i$ extends the corresponding embedding witnessing $$\col_{k_i}(e)\leq \col_{k'_{\varphi(i)}}(\varphi_i(e))$$ derived from $\varphi_i$ over the embedding witnessing $\col_{k_i}(e)\leq \col_{\hat{k}_i}(e)$ derived from $\mu_i$.
%
%We call $\OA$ \emph{extendible} if, under these assumptions, the embedding given by $f^\eta(\hat{k})\leq f^\nu(k')$ extends the embedding given by $f^\eta(k)\leq f^\nu(k')$ over the embedding given by $f^\eta(k)\leq f^\eta(\hat{k})$. (Existence of these embeddings are given by the fact that $\OA$ is infinitely extensive.)
\end{defn}
\begin{thm}
Let $\OA=\OA_{\mathbb{L},\mathbb{P}}^Q=\langle \CC,\AAA,\PP,\MM,\RR\rangle$ for some quasi-order $Q$, some class of linear orders $\mathbb{L}$ and some class of partial orders $\mathbb{P}$. Then $\OA$ is extendible.
\end{thm}
\begin{proof}
By Lemma \ref{Lemma:PORiterable} we know that $\MM$ is $\RR$-iterable. If we take the sum of some partial orders, then each of these partial orders embeds into the sum and therefore every $f\in \MM$ is extensive.

Suppose now as in Definition \ref{Defn:InfExtensive} that we have composition sequences
$$\eta=\langle \langle f_i,s_i\rangle:i\in r\rangle \mbox{ and }\nu=\langle \langle g_i,s'_i\rangle:i\in r'\rangle$$ with $\eta\leq \nu$ and $k=\langle q_u:u\in A^\eta\rangle\in \dom(f^\eta)$ and $k'=\langle q'_u:u\in A^{\nu}\rangle\in \dom(f^{\nu})$ such that $q_u\leq q'_{\varphi_{\eta,\nu}(u)}$ for each $u\in A^{\eta}$, with $\varphi_u$ a witnessing embedding. 

% Suppose also that there is an embedding $\varphi:r\rightarrow r'$ such that $$f_i\leq_\MM f_{\varphi(i)},$$
%and further for each $i\in r$ is an embedding $\varphi_i:\arity(f_i)\rightarrow \arity(g_{\varphi(i)})$ that witnesses $k_i\leq_{\PP(\CC)} k'_{\varphi(i)}$
%and is such that $\varphi_i(s_i)=s'_{\varphi(i)}$. If we can show that under these assumptions we have $f^\eta(k)\leq f^\nu(k')$, then we have that $\OA$ is infinitely extensive.

We have that $f^\eta(k)=\sum_{u\in H_\eta}q_u$ and $f^\nu(k')=\sum_{u\in H_\nu}q'_u$. Moreover $\varphi_{\eta,\nu}$ is an embedding from $H_\eta$ to $H_\nu$. So we define $\varphi_{\eta,\nu}^{k,k'}:f^\eta(k)\rightarrow f^\nu(k')$ so that for $a\in f^\eta(k)$, with $a\in q_u$ we have $$\varphi_{\eta,\nu}^{k,k'}(a)=\varphi_u(a)\in q'_{\varphi_{\eta,\nu}(u)}.$$
Which is clearly an embedding, hence $\OA$ is infinitely extensive.

Now suppose that $\hat{k}=\langle \hat{q}_u:u\in A^\eta\rangle$ is such that $q_u\leq \hat{q}_u$ for all $u\in A^\eta$, with some witnessing embedding $\mu_u$. Suppose also that $\hat{q}_u\leq q'_{\varphi(\eta,\nu)}$ for each $u\in A^\eta$, with $\psi_u$ a witnessing embedding, and that $\varphi_u=\psi_u\circ \mu_u$. Then for $a\in f^\eta(k)$, with $a\in q_u$ we have $$\varphi^{k,k'}_{\eta,\nu}(a)=\varphi_u(a)=\psi_u\circ\mu_u(a)\in q'_{\varphi_{\eta,\nu}(u)}$$
and $\mu^{k,k'}_{\eta,\nu}(a)=\mu_u(a)\in \hat{q}_u$, so that $$\psi^{\hat{k},k'}_{\eta,\nu}\circ \mu^{k,k'}_{\eta,\nu}(a)=\psi_u\circ \mu_u(a)\in q'_{\varphi_{\eta,\nu}(u)}.$$
Thus we have verified the conditions of Definition \ref{Defn:Extendible}, and $\OA$ is extendible.
%
%
%%%%%
%%%%%Let $\psi:H_\eta\rightarrow H_\nu$ be defined such that when $a\in a^\eta_i$ we have $\psi(a)=\varphi_i(a)$. Since $\varphi_i(s_i)=s'_{\varphi(i)}$ we have that $\psi(a)\in a^\nu_{\varphi(i)}\subseteq H_\nu$. Then $\psi$ is an embedding, since $\varphi$ and the $\varphi_i$ $(i\in r)$ were all embeddings.
%%%%%
%%%%%Since $\varphi_i$ witnesses $k_i\leq_{\PP(\CC)} k'_{\varphi(i)}$ we have that $P_{i,a}\leq P'_{\varphi(i),\varphi_i(a)}$. Since $\psi$ maps the point indexed by $\langle i,a\rangle $ to $\langle \varphi(i),\varphi_i(a)\rangle$ we can extend $\psi$ to $\psi':f^\eta(k)\rightarrow f^\nu(k')$, using the embeddings given by $P_{i,a}\leq P'_{\varphi(i),\varphi_i(a)}$. This gives an embedding, and hence $f^\eta(k)\leq f^\nu(k')$ as required. 
%%%%%
%%%%%Now suppose that $\hat{k}$ has the same structure as $k$ but with larger arguments, ie each of the $\hat{P}_{i,a}$ given by $\hat{k}$ is such that $P_{i,a}\leq \hat{P}_{i,a}$. Suppose also that $\varphi$ witnesses $\hat{k}_i\leq_{\PP(\CC)} k'_{\varphi(i)}$, in such a way that each embedding given by $\hat{P}_{i,a}\leq P'_{\varphi(i),\varphi_i(a)}$ extends the embedding given by $P_{i,a}\leq P'_{\varphi(i),\varphi_i(a)}$ that we had before. Then extending $\psi$ to $\hat{\psi}$ using the embeddings from $\hat{P}_{i,a}\leq P'_{\varphi(i),\varphi_i(a)}$ will make $\hat{\psi}$ an extension of $\psi'$. Thus $\OA$ is extendible.
\end{proof}
\begin{remark}\label{Rk:composition}
Let $x\in \amrs$ with $\rank(x)=\alpha$. So for some composition sequence $\eta$, we have that $x=f^\eta(\langle q_u:u\in A^\eta \rangle)$, with $q_u\in \CC_{<\alpha}$ for all $u\in A^\eta$. %We notice that for fixed $\eta$, the length of the sequence $r$ and the arities $\arity(f_i)$ in the domain of $f^\eta$ are fixed. Hence, we can consider $f^\eta$ as a multivariate function from $\CC$, since we just need to determine the colours of the arities, (elements of $\CC_{<\alpha}$) in order to evaluate the value of $f^\eta$. 
Applying the same reasoning to each of the $q_u$, it is then possible to view $x$ as a composition of some more composition functions, applied to further lower ranked elements or \emph{fragments} (see Figure \ref{Fig:CompArgs}). 
\begin{figure}
  
  \centering
    \begin{tikzpicture}
	\node [left] at (1,5) {$f^\eta$};
\draw [thick] (0.1,0.5) -- (1,5);
\draw [fill] (0.8,4) circle [radius=0.06];
\node [left] at (0.8,4) {$i$};
	\draw  (0.8,4) -- (0.16,3.3);
	\draw  (0.8,4) -- (1.16,3.3);
	\draw  (0.8,4) -- (1.66,3.3);
	\draw  (0.8,4) -- (2.16,3.3);
	\draw  (0.8,4) -- (2.66,3.3);
	\draw [fill] (2.16,3.3) circle [radius=0.06];
	\node [below left] at (2.16,3.3) {$f^\nu$};
%	\node [left] at (0.66,3.3) {$s_i$};
%	\node [rotate=90, scale=1.5] at (1.93,3.1) {$\Bigg \{$};
%	\node at (1.93,2.8) {$a_i$};

%	\draw [fill] (0.46,2.3) circle [radius=0.08];
%\draw [fill] (0.6,3) circle [radius=0.08];
	\draw  (0.6,3) -- (-0.04,2.3);
	\draw  (0.6,3) -- (0.96,2.3);
	\draw  (0.6,3) -- (1.46,2.3);
%	\draw  (0.6,3) -- (1.96,2.3);
%	\draw  (0.6,3) -- (2.46,2.3);
	\draw  (0.6,3) -- (-0.44,2.3);
%	\draw [fill] (0.46,2.3) circle [radius=0.08];

%\draw [fill] (0.4,2) circle [radius=0.08];
	\draw  (0.4,2) -- (-0.24,1.3);
	\draw  (0.4,2) -- (0.76,1.3);
	\draw  (0.4,2) -- (1.26,1.3);
	\draw  (0.4,2) -- (1.76,1.3);
%	\draw  (0.4,2) -- (2.26,1.3);
%	\draw  (0.4,2) -- (2.76,1.3);
	
\draw [thick] (2.16,3.3) -- (2.92,-0.5);
	\draw  (2.4,2) -- (2.08,1.3);
	\draw  (2.4,2) -- (3.08,1.3);
	\draw  (2.4,2) -- (3.58,1.3);
	\draw  (2.4,2) -- (4.08,1.3);
	
	\draw  (2.6,1) -- (1.78,0.3);
	\draw  (2.6,1) -- (2.28,0.3);
	\draw  (2.6,1) -- (3.28,0.3);
	\draw  (2.6,1) -- (3.78,0.3);
%	\draw  (2.6,1) -- (4.28,0.3); fullstop

\end{tikzpicture}
\caption{The arrangement of the composition of $f^\eta$ and $f^\nu$. Here $f^\nu$ is one of the arguments of $f_i$.}\label{Fig:CompArgs}
\end{figure}
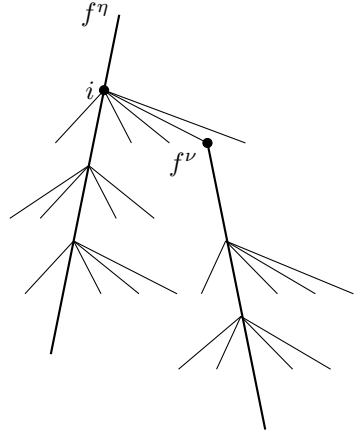

We repeat this process, splitting fragments into more fragments. Because the ranks are well-founded we will eventually obtain $g$, a composition of many $f^\eta$, and an element $d$ of $\dom(g)$, which is a configuration of elements of $\AAA$ such that $g(d)=x$. Since $g$ is just a composition of $f^\eta$, it will usually be possible to compose further still.
% So we define $\Dm(f^\eta)=x'\in \{\col(k_i):i\in r\}$ which would be the domain of this function
%
%Now if $k=\langle k_i:i\in r\rangle$ then for each $x'\in \{\col(k_i):i\in r\}$, we have $x\in\CC_{<\alpha}$. So it must be that $x=f^{\eta'}(k')$ for some $\eta'$ and $k'$. Moreover, we can compose $f^\eta$ and $f^{\eta'}$ in the sense that we had before, placing .
\end{remark}
We now make precise the notion of composing many $f^\eta$.
\begin{defn}\label{Defn:Composing}
Let $\PS$ be a tree of finite sequences of elements of $\bigcup \RR\times \bigcup \PP$ under $\is$, and let $\DS$ be the set of leaves of $\PS$. Suppose that $\PS$ has a $\is$-minimal element,\footnote{For this definition we denote the minimal element of $\PS$ by $\langle \rangle$, but in general it may not be $\langle \rangle$.} and that we have $\FS=\{\eta(\vec{p}):\vec{p}\in \PS\setminus \DS\}$ with each $\eta(\vec{p})$ a composition sequence. We write $r(\vec{p})$ for the length of $\eta(\vec{p})$. Now suppose that for each $\vec{p}\in \PS$ we have
$$\vec{p}\con \langle i,u\rangle\in \PS\mbox{ iff }i\in r(\vec{p}) \mbox{ and }u\in a^{\eta(\vec{p})}_i$$
then we call $\FS$ an \emph{composition set} and $\PS$ a set of \emph{position sequences} of $\FS$. If additionally $\PS$ is a chain-finite tree under $\is$ then we call $\FS$ \emph{admissible}.

If $\FS$ is admissible we define the \emph{composition of $\FS$}, a new function $g^\FS:\CC^\DS\rightarrow \CC$. To do so, we determine the value of $g^\FS(\langle q_{\vec{p}}:\vec{p}\in \DS\rangle )$. When $\vec{p}\in \PS\setminus \DS$ we define inductively $$k^{\vec{p}}=\langle k_{i,u}^{\vec{p}}:i\in r(\vec{p}),u\in a^{\eta(\vec{p})}_i\rangle\in \dom(f^{\eta(\vec{p})}),$$ such that for each $i\in r(\vec{p})$ and each $u\in a_i^{\eta(\vec{p})}$ we have
$$k^{\vec{p}}_{i,u}=
\left \{ \begin{array}{lcl}
f^{\eta(\vec{p}\con \langle i,u\rangle)}(k^{\vec{p}\con \langle i,u\rangle}) & : & \vec{p}\con\langle i,u\rangle\notin \DS \\
q_{\vec{p}\con\langle i,u\rangle} & : & \vec{p}\con\langle i,u\rangle\in \DS
\end{array}
\right ..
$$
Since $\PS$ was a chain-finite tree, these $k^{\vec{p}}$ are well-defined. Now when $\FS\neq \emptyset$, we define  $$g^\FS(\langle q_{\vec{p}}:\vec{p}\in \DS\rangle )=f^{\eta(\langle \rangle)}(k^{\langle \rangle}),$$
and when $\PS=\DS=\{\langle \rangle\}$ so that $\FS=\emptyset$, we define $g^\FS(\langle x \rangle)=x$ for every $x\in \CC$.

\end{defn}
%\begin{remark}\label{Rk:CompositionID}
%Notice that when $\PS=\DS=\{\langle \rangle\}$ and $\FS=\emptyset$, we have $g^\FS(\langle x \rangle)=x$ for every $x\in \CC$, (because $f^{\langle \rangle}=\Id_\CC$).
%\end{remark}

\begin{defn}
Let $\FS$ be an admissible composition set and $\PS$ be a set of position sequences of $\FS$. For $\vec{p}\in \PS$ we define $$\FS(\vec{p})=\{\eta(\vec{u})\in \FS\mid \vec{p}\is \vec{u}\}$$ and $\PS(\vec{p})=\{\vec{u}\in \PS\mid \vec{p}\is \vec{u}\}$. We notice that $\FS(\vec{p})$ is an admissible composition set and $\PS(\vec{p})$ is a set of position sequences of $\FS(\vec{p})$. We also define $\DS(\vec{p})$ to be the set of leaves of $\PS(\vec{p})$.
\end{defn}
\begin{defn}
Let $\FS$ be an admissible composition set. We call $g^{\FS}$ a \emph{decomposition function} for $x\in \amrs$, whenever there is some $\langle q_{\vec{p}}:\vec{p}\in \DS\rangle \in \dom (g^{\FS})$, such that $(\forall \vec{p}\in \DS)$, $q_{\vec{p}}\in \AAA$ and
$$x=g^{\FS}(\langle q_{\vec{p}}:\vec{p}\in \DS\rangle).$$

\end{defn}
\begin{prop}\label{Prop:Reduction}
Let $x\in \CC_\alpha$ then for some composition sequence $\eta$ and $k\in \dom(f^\eta)_{<\alpha}$, we have $x=f^\eta(k)$.
\end{prop}
\begin{proof}
Clear by Remark \ref{Rk:amrs}.
\end{proof}

\begin{lemma}\label{Lemma:decompfn}
For any $x\in \amrs$, there is a decomposition function $g$ for $x$.
\end{lemma}
\begin{proof}
Let $x\in \amrs$, then we will define a decomposition function $g$ for $x$ inductively on the rank of $x$ as follows. If $\rank(x)=0$ then $x\in \AAA$ so set $\FS=\emptyset$ and $\PS=\DS=\{\langle \rangle\}$, so that $g^\FS(\langle x\rangle)=x$ is a decomposition function for $x$.

Suppose for induction that for every $q\in \CC_{<\alpha}$ there is a decomposition function for $x$. If $\rank(x)=\alpha>0$, then $x=f^\eta(k)$ for some composition sequence $\eta$ of length $r$, and $$k=\langle q_{i,u}:i\in r,u\in a^\eta_i\rangle\in \dom(f^\eta)_{<\alpha}.$$ So by the induction hypothesis, for each $i\in r$ and $u\in a^\eta_i$, we see that $q_{i,u}=g^{\FS_{i,u}}(d_{i,u})$ for some admissible composition set $\FS_{i,u}=\{\eta_{i,u}(\vec{p})\mid \vec{p}\in \PS_{i,u}\}$ with $\PS_{i,u}$ a set of position sequences for $\FS_{i,u}$ and some $d_{i,u}=\langle d^{i,u}_{\vec{p}}:\vec{p}\in \DS_{i,u}\rangle$, where $\DS_{i,u}$ is the set of leaves of $\PS_{i,u}$.
Let $$\PS=\{\langle i,u\rangle\con \vec{p}:i\in r, u\in a^\eta_i, \vec{p}\in \FS_{i,u}\}\cup \{\langle \rangle\},$$ and let $\DS$ be the set of leaves of $\PS$. We set $\eta(\langle \rangle)=\eta$ and for $i\in r$ and $u\in a^\eta_i$, we set $\eta(\langle i,u\rangle\con \vec{p})=\eta_{i,u}(\vec{p})$. Now let $\FS=\{\eta(\vec{p})\mid \vec{p}\in \PS\}$, so that clearly $\FS$ is an admissible composition set. Finally we set $d_{\langle i,u\rangle\con \vec{p}}=d^{i,u}_{\vec{p}}$ and $d=\langle d_{\vec{p}}:\vec{p}\in \DS\rangle$ now we have by construction
$$g^\FS(d)=f^\eta(\langle g^{\FS_{i,u}}(d_{i,u}):i\in r, u\in a^\eta_i\rangle)=f^\eta(k)=x.$$
Thus $g^\FS$ is a decomposition function for $x$, which completes the induction.
\end{proof}
\begin{defn}
We call a decomposition function \emph{standard} if it can be constructed by the method of Lemma \ref{Lemma:decompfn}.
\end{defn}
%\begin{defn}
%Given $x\in \amrs$, let $g$ be as in Remark \ref{Rk:composition}, we call $g$ a \emph{decomposition function} for $x$.
%
%Each $k\in \dom(g)$, is a structure of nested colourings of chains from $\RR$ and arities from $\PP$, finally coloured by elements of $\CC$. We call these colours from $\CC$ \emph{arguments} of $k$, and the sequence of elements of $\RR$ and $\PP$ that codes where to find a given argument its \emph{position}.
%\end{defn}

\begin{defn}\label{Defn:LimitSequence}
Let $(x_n)_{n\in \omega}$ be a sequence of elements of $\amrs$. Suppose that $\CC$ has a minimal element $q_0$ and for every $n\in \omega$, there is a standard decomposition function $g_n=g^{\FS_n}$ for $x_n$, and some $k_n=\langle k_n^{\vec{p}}:\vec{p}\in \DS_n\rangle\in \dom(g^{\FS_n})$ with $g_n(k_n)=x_n$. Let $\PS_n$ be the set of position sequences of $\FS_n$ and $\DS_n$ be the set of leaves of $\PS_n$. Then we call $(x_n)_{n\in \omega}$ a \emph{limiting sequence} if there are such $\PS_n$, $\FS_n$, $\DS_n$ and $k_n$ such that the following properties hold for every $n\in \omega$.
\begin{enumerate}
\item \label{Item:LimitSeq1} $\PS_{n}$ is a $\up$-closed subset of $\PS_{n+1}$.
\item \label{Item:LimitSeq2} $\FS_{n+1}=\{\eta_n(\vec{p})\mid \vec{p}\in \PS_{n}\setminus \DS_{n}\}$.
\item \label{Item:LimitSeq3} $\eta_n(\vec{p})\is \eta_{n+1}(\vec{p})$ for every $\vec{p}\in \PS_{n}\setminus \DS_{n}$.
\item \label{Item:LimitSeq4} If $\vec{u}\in \DS_m$ for all $m\geq n$, then for all $m\geq n$ we have $k_{n}^{\vec{u}}=k_m^{\vec{u}}$.
\item \label{Item:LimitSeq5} If $\vec{u}\in \DS_n$ and $\vec{v}\in \DS_m$ with $n<m$ and $\vec{u}\sis \vec{v}$, then $k_n^{\vec{u}}=q_0$, and $\forall \vec{p}\in \DS_{n+1}$ such that $\vec{u}\is \vec{p}$ we have in fact $\vec{u}\sis \vec{p}$.
%\item If there is an argument $q$ of $k$ (so $q\in \AAA$) that is different from an argument $q'$ of $k'$, ($q'\in \CC$) with $q$ and $q'$ in the same position; then $q=q_0$.
%\item $x_n\leq x_{n+1}$ for each $n\in \omega$, with $\mu_n$ fixed distinguished embeddings.
\end{enumerate}
\end{defn}

Limiting sequences $(x_n)_{n\in \omega}$ are those sequences with an increasing construction; that is to say that $x_{n+1}$ is produced by \emph{the same} set of functions as $x_n$ applied to increasingly more functions. We could perhaps be slightly less restrictive in the definition, particularly in conditions (\ref{Item:LimitSeq4}) and (\ref{Item:LimitSeq5}), however these conditions simplify some of the work later on and do not reduce the set of limits we can produce.\footnote{This is the case at least in the applications used within this paper.} This condition determines the characteristics of limits, in particular it enforces that our bqo class of $\sigma$-scattered trees will be those covered by \emph{$\up$-closed} scattered trees (as opposed to any scattered trees).

\begin{remark}\label{Rk:Mu}
Suppose $\OA$ is infinitely extensive, let $(x_n)_{n\in \omega}$ be a limiting sequence, and $\eta=\eta(\langle \rangle)$. So for each $n\in \omega$ we have $x_n=f^{\eta}(k^n)$ and $x_{n+1}=f^{\eta}(k^{n+1})$, for some $k^n=\langle q^n_u:u\in A^{\eta}\rangle$, $k^{n+1}=\langle q^{n+1}_u:u\in A^{\eta}\rangle$. It can be seen by an easy induction\footnote{Using that $q_0$ was minimal and repeatedly applying Lemma \ref{Lemma:RealInfExtensive}.} that $q^n_u\leq q^{n+1}_u$ with $\psi_u$ a witnessing embedding, for all $u\in A^{\eta}$. So to simplify notation, we let $\mu_n=\psi_{\eta,\eta}^{k^n,k^{n+1}}$. 

Note that in all of the applications in this paper, we consider $x_n$ as a subset of $x_{n+1}$, in this case $\mu_n$ just acts as the identity on elements of $x_n$, and we can define the limit to be the union.
%If $\OA$ is infinitely extensive then $x_n\leq x_{n+1}$ follows from the first two properties of Definition \ref{Defn:LimitSequence}, since $q_0$ was a minimal element. In this case we always let $\mu_n$ be the embedding given by infinite extensiveness.
\end{remark}

\begin{defn}\label{Defn:HasLimits}
Suppose that to every limiting sequence $(x_n)_{n\in \omega}$ we associate a unique \emph{limit} $x\in \CC$. %satisfying %$x_n\leq x$ for all $n\in \omega$, % with $x_n\leq x$ for each $n\in \omega$. 
Let $\eta$ be a composition sequence, and for each $n\in \omega$ let $\eta_n\is \eta$ be a composition sequence so that $\eta_n\is \eta_{n+1}$, and $\eta=\bigcup_{n\in \omega} \eta_n$. Also let $k_n=\langle q_u^n:u\in A^\eta\rangle\in \dom(f^{\eta_n})$ be such that for every $u\in A^\eta$ we have $(q_u^n)_{n\in \omega}$ is a limiting sequence with limit $q_u$. We say that $\OA$ \emph{has limits} if the limit of $(f^{\eta_n}(k_n))_{n\in \omega}$ is precisely $f^\eta(\langle q_u:u\in A^\eta\rangle)$ for all such $\eta$, $\eta_n$ and $k_n$ $(n\in \omega)$.
%
% if the limits satisfy the following properties for any $m\in \omega$.
%
%Let $g_m$ be the decomposition function for $x_m$ as in Definition \ref{Defn:LimitSequence}. For $\vec{p}\in \DS(g_m)$ we let $x^{\vec{p}}$ be the limit of $(x^{\vec{p}}_n)_{n\geq m}$ (which is a limiting sequence by Proposition \ref{Prop:LimitSeq}). Then we have $$x=g_m(\langle x^{\vec{p}}:\vec{p}\in \DS(g_m)\rangle ).$$
\end{defn}

\begin{ex}\label{Rk:PoLimits}
Let $\OA=\OA_{\mathbb{L},\mathbb{P}}^Q=\langle \CC,\AAA,\PP,\MM,\RR\rangle$ for some quasi-order $Q$, some class of linear orders $\mathbb{L}$ and some class of partial orders $\mathbb{P}$.

Given a limiting sequence $(x_n)_{n\in \omega}$, we have by Remark \ref{Rk:Mu} that $\mu_n$ was the embedding given by infinite extensiveness of $\OA$ using the construction of each $x_n$. Each $x_{n+1}$ is $x_n$ with some elements coloured by $-\infty$ replaced by a larger partial order. The embedding $\mu_n$ then acts as the identity on everything that is not changed, and maps a point $a$ coloured by $-\infty$ that is replaced, into some element $b$ of the order that replaces it. Since it makes no difference to the order, we equate $a$ and $b$, and consider the underlying set of $x_n$ as a subset of the underlying set of $x_{n+1}$, possibly with some colours changing from $-\infty$. In this way each $\mu_n$ $(n\in \omega)$ becomes the identity map from the underlying set of $x_n$ to the underlying set of $x_{n+1}$.

%If we take a sum over a partial order $P$, where each argument of the sum is non-empty, then we can consider $P$ as a subset of the sum, just by choosing an arbitrary point from each of the parts of the sum corresponding to the arguments (giving a copy of $P$ inside the sum). Hence increasing sequences can be considered as sequences of partial orders with $x_0\subseteq x_1 \subseteq ...$ and we can define their limit to be the countable union.
Now given a limiting sequence $(x_n)_{n\in \omega}$ we have that $x_n\subseteq x_{n+1}$ for all $n\in \omega$. Hence we can define the limit $x$ to be the union of all of the $x_n$. This means that $\OA$ has limits, because an $H_\eta$-sum of limits, is the limit of the $H_\eta$-sums of the elements of the limiting sequence.

\end{ex}
\begin{defn}
Suppose that $\OA$ has limits. Let $\amr\subseteq \CC$ be the class containing all elements of $\amrs$ and all limits of limiting sequences in $\amrs$. We also define $\amri=\amr\setminus \amrs.$
\end{defn}

%\begin{remark}\label{Rk:LimitComposition}
%When $\CC$ has limits and $x\in \amr$, we can try to apply the same construction as in Remark \ref{Rk:composition} to $x$. If $x=f^{\eta}(k_0)$ for a composition sequence $\eta$, then we want to find $g_n$ and $k_n\in \dom(g_n)$ for each $n\in \omega$ such that $g_0=\Id_\CC$, $g_1=f^\eta$, and for any $n\in \omega$, $x=g_n(k_n)$ and $g_{n+1}$ is a composition of $g_n$ with some $f^\nu$. We do this by taking each argument of $k_n$ and writing it as some $f^\nu(k')$, then defining the composition $g_{n+1}$ by composing $g_n$ with each $f^\nu$ in the corresponding position. This composition gives us an element $k_{n+1}$ of the domain of $g_{n+1}$ such that $g_{n+1}(k_{n+1})=x$, and we continue. Now we set $k'_n$ to have the same structure as $k_n$, and to have the same arguments when this argument is in $\AAA$, and if not then we change the argument to $q_0$, the minimal element as from Definition \ref{Defn:LimitSequence}. If we set $x'_n=g_n(k'_n)$ then it is clear from the definition that $(x'_n)_{n\in \omega}$ is a limiting sequence.
%\end{remark}
%\begin{defn}
%We say that $\CC$ has \emph{nice limits} if for any $x\in \amr$ such that we can define $g_n$ $(n\in \omega)$ and $x'_n$ as in Remark \ref{Rk:LimitComposition}, then $x$ is the limit of $(x'_n)_{n\in \omega}$.
%\end{defn}

Finally the following condition allows us to produce embeddings between limits.
\begin{defn}\label{Defn:NiceLimits}
Suppose that $\OA$ has limits. Then we say that $\OA$ has \emph{nice limits} iff for any $x,y\in \amr$ with $x$ the limit of $(x_n)_{n\in \omega}$, and for any $n\in \omega$; if there are embeddings  %$\mu_n:x_n\rightarrow x_{n+1}$ and 
$\varphi_n:x_n\rightarrow y$ such that $\varphi_{n}=\varphi_{n+1}\circ \mu_n$, % (for $\mu_n$ as in Remark \ref{Rk:Mu}), 
then $x\leq y$.\footnote{Here $\mu_n$ is as from Remark \ref{Rk:Mu}.}
\end{defn}
\begin{thm}
Let $\OA=\OA_{\mathbb{L},\mathbb{P}}^Q=\langle \CC,\AAA,\PP,\MM,\RR\rangle$ for some quasi-order $Q$, some class of linear orders $\mathbb{L}$ and some class of partial orders $\mathbb{P}$. Then $\OA$ has nice limits.
\end{thm}
\begin{proof}
Let $x,y\in \amr$ be such that $x$ is the limit of $(x_n)_{n\in \omega}$ and for any $n\in \omega$ we have embeddings $\varphi_n:x_n\rightarrow y$ such that $\varphi_n=\varphi_{n+1}\circ\mu_n$. We want to show that $x\leq y$.

%We have by Remark \ref{Rk:Mu} that $\mu_n$ was the embedding given to us by infinite extensiveness of $\MM$ using the construction of each $x_n$. Each $x_{n+1}$ is $x_n$ with some elements coloured by $-\infty$ replaced by a larger partial order. The embedding $\mu_n$ is then the identity on everything that is not changed, and maps a point $a$ coloured by $-\infty$ that is replaced, into some element $b$ of the order that replaces it. Since it makes no difference to the order, we equate $a$ and $b$, and consider $x_n$ as a subset of $x_{n+1}$, possibly with some colours of elements coloured by $-\infty$ changing. In this way each $\mu_n$ $(n\in \omega)$ becomes the identity map from $x_n$ to $x_{n+1}$.

By Remark \ref{Rk:PoLimits} we consider each $x_n\subseteq x_{n+1}$ and each $\mu_n$ ($n\in \omega$) to be the identity map. So $\varphi_n=\varphi_{n+1}\circ\mu_n$ is equivalent to $a\in x_n$ implies $\varphi_{n+1}(a)=\varphi_n(a)$. Hence it is possible to define $\varphi:x\rightarrow y$ as the union of all of the $\varphi_n$. We claim that $\varphi$ is an embedding.

Let $a,b\in x=\bigcup_{n\in \omega}x_n$, and let $n$ be least such that $a,b\in x_n$ and $\col_{x_n}(a)=\col_x(a)$ (such an $n$ exists since either the colour of $a$ is always $-\infty$ or changes at some $n$, but then stays this colour in $x$). In order to show that $x\leq y$ we need to verify the following properties of $\varphi$.

\begin{enumerate}
%	\item $\varphi$ is injective. If $a\neq b$ then $\varphi(a)=\varphi_n(a)\neq \varphi_n(b)=\varphi(b)$ as required.
	\item $a\leq b$ iff $\varphi_n(a)\leq \varphi_n(b)$ iff $\varphi(a)\leq \varphi(b)$ (since $\varphi_n$ is an embedding).
	\item $\col_x(a)=\col_{x_n}(a)\leq \col_y(\varphi_n(a))=\col_y(\varphi(a))$.
\end{enumerate}
So indeed $\varphi$ is an embedding from $x$ to $y$, and thus $x\leq y$ as required.
\end{proof}

\section{Decomposition trees}\label{Section:StructTrees}
The aim of this section is to encode the elements of $\amr$ in terms of structured trees. This reduces the problem of showing that $\amr$ is bqo to showing that a class of structured trees is bqo.
\subsection{Scattered and structured $\mathscr{L}$-trees}
%%%%%%%% SCAT TREES %%%%%%%%
\begin{defn}\label{Defn:2^<omega}
We define the tree $2^{<\omega}$, (the infinite binary tree of height $\omega$), as the tree of finite sequences of elements of $2=\{0,1\}$ ordered by $\is$. 
\end{defn}
\begin{defn}\label{Defn:WBScatteredTrees}
Let $\mathbb{L}$ be a class of linear orders and $T$ be an $\mathscr{L}$-tree, then we define as follows:
\begin{itemize}
	\item $T$ is \emph{scattered} iff $2^{<\omega}\not \leq T$.
	\item $T$ is \emph{$\mathbb{L}$-$\sigma$-scattered} iff there are countably many $\up$-closed subsets of $T$ that are each scattered $\mathbb{L}$-trees, and every point of $T$ is contained in one of these subsets.
	\item We let $\scatt^\mathbb{L}$ be the class of scattered $\mathbb{L}$-trees.
	\item We let $\sscatt^\mathbb{L}$ be the class of $\mathbb{L}$-$\sigma$-scattered $\mathscr{L}$-trees.
	\item We let $\mathscr{R}$ be the class of rooted $\omega+1$-trees.	
	\end{itemize}
\end{defn}
\begin{remark}
When the context is clear we will call elements of $\sscatt^\mathbb{L}$ $\sigma$-scattered. Notice that elements of $\sscatt^\mathbb{L}$ are not necessarily $\mathbb{L}$-trees as they could contain chains with order type not in $\mathbb{L}$.\footnote{Elements of $\sscatt^\mathbb{L}$ are are still $\mathscr{L}$-trees.} We note also that different classes of linear orders can generate the same classes of $\sigma$-scattered trees, for example $\sscatt^\omega=\sscatt^{\omega+1}$.
\end{remark}

\begin{defn}Given a chain $\zeta$, and $\mathscr{L}$-trees $T^\gamma_i$ for each $i\in \zeta$, and $\gamma< \kappa_i\in \card$, we define the \emph{$\zeta$-tree-sum} of the $T_i^\gamma$ (see Figure \ref{Fig:ZetaTreeSum}) as the set $\zeta\sqcup\bigsqcup_{i\in \zeta, \gamma<\kappa_i}T_i^\gamma$ %partially 
ordered by letting $a\leq b$ iff 
\begin{itemize}
	\item $a,b\in \zeta$ with $a\leq_\zeta b$;
	\item or for some $i\in \zeta$, $\gamma<\kappa_i$ we have $a,b\in T^\gamma_i$ with $a\leq_{T_i^\gamma}b$;
	\item or $a\in \zeta$ and $b\in T_i^\gamma$ for some $i\in \zeta$ with $a<_\zeta i$ and $\gamma<\kappa_i$.
\end{itemize}
\end{defn}
 %$T$ such that there is a chain $\zeta\in\mathbb{L}$ and for each $i\in \zeta$ there is some $T_i\in \scatt^\mathbb{L}_{\alpha}$, with
\begin{defn}
Let $\mathbb{L}$ be a class of linear orders that is closed under finite sums. Let $\scatt^\mathbb{L}_0=\{\emptyset,1\}$, and for $\alpha\in \On$ let $\scatt^\mathbb{L}_{\alpha+1}$ be the class of $\zeta$-tree-sums of trees of $\scatt^\mathbb{L}_{\alpha}$ for $\zeta\in \mathbb{L}$. For limit $\lambda\in \On$ we let $\scatt^\mathbb{L}_\lambda=\bigcup_{\gamma<\lambda} \scatt^\mathbb{L}_\gamma$, and finally set $\scatt^\mathbb{L}_\infty=\bigcup_{\gamma\in \On} \scatt^\mathbb{L}_\gamma$. For $T\in \scatt^\mathbb{L}_\infty$ define the \emph{scattered rank} of $T$, denoted $\Srank(T)$ as the least ordinal $\alpha$ such that $T\in \scatt^\mathbb{L}_\alpha$. (See Figure \ref{Fig:ScatTrees}.)
\end{defn}
\begin{thm}\label{Thm:ScatteredRank}
Let $\mathbb{L}$ be a class of linear orders that is closed under non-empty subsets and finite sums. Then $\scatt^\mathbb{L}=\scatt^\mathbb{L}_\infty$.
\end{thm}
\begin{proof}
%A simple induction shows that $\scatt^\mathbb{L}_\infty\subseteq	 \scatt^\mathbb{L}$. % For $T\in \scatt^\mathbb{L}$, and $\zeta$ a chain of $T$, let ${}^\zeta\downs t=\downs t\setminus \bigcup_{u\in \zeta, u>t}\downs u$. Suppose for contradiction that $T\in \scatt^\mathbb{L}\setminus \scatt^\mathbb{L}_\infty$. For any such $T$, if $\xi$ is an $\up$-closed chain of $T$ then either $\bigcup_{u\in \xi}\down u\notin \scatt^\mathbb{L}_\infty$ or there is some $t\in \xi$ such that ${}^\xi\downs t\notin \scatt^\mathbb{L}_\infty$, since otherwise $T\in \scatt^\mathbb{L}_\infty$.
%The more difficult direction can be seen by showing that if $T\in \scatt^\mathbb{L}\setminus \scatt^\mathbb{L}_\infty$ then $2^{<\omega}\leq T$. 
We leave the proof as an exercise since it is similar to Theorem \ref{Thm:TreesAreAMR}.
%Let $T=T_0$ and suppose that for some $\alpha\in\On$ we have defined some $t_0<...<t_\alpha\in T$ and chains $\xi_0,...,\xi_\alpha\subseteq T$
\end{proof}
%\vspace{-6pt}
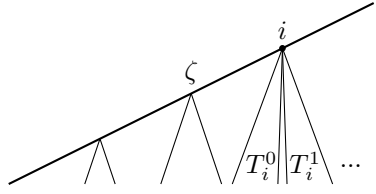
\begin{figure}
\centering

\begin{tikzpicture}[scale=0.6]
\draw [thick] (8,4) -- (0,0);
\node [left, above] at (4,2) {$\zeta$};
\draw (6,3) -- (4.9,0);
\draw (6,3) -- (7.1,0);
\draw (6,3) -- (5.9,0);
\draw (6,3) -- (6.1,0);
\draw [fill] (6,3) circle [radius=0.06];
\node [left, above] at (6,3) {$i$};
\node at (5.54,0.4) {$T_i^0$};
\node at (6.5,0.4) {$T_i^1$};
\node at (7.5,0.4) {$...$};
\draw (2,1) -- (1.666666,0);
\draw (2,1) -- (2.333333,0);
\draw (4,2) -- (3.333333,0);
\draw (4,2) -- (4.666666,0);
\end{tikzpicture}

%\vspace{-5pt}
\caption{A $\zeta$-tree-sum of the $T_i^\gamma$.}
\label{Fig:ZetaTreeSum}
\end{figure}

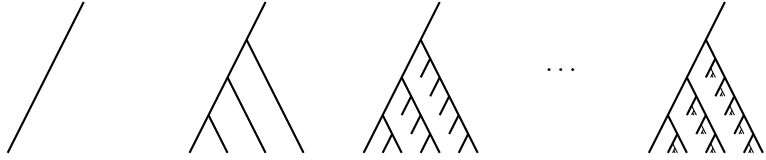
\begin{figure}
\centering

\begin{tikzpicture}[thick, scale=0.5]
\draw (2,4) -- (0,0);
%\node [left] at (1,2) {$\zeta$};
\node [left, above] at (3,0) {};
\end{tikzpicture}
\hspace{15pt}
\begin{tikzpicture}[thick, scale=0.5]
\draw (2,4) -- (0,0);
\draw (1,2) -- (2,0);
\draw (1.5,3) -- (3,0);
\draw (0.5,1) -- (1,0);
\end{tikzpicture}
\hspace{15pt}
\begin{tikzpicture}[thick, scale=0.5]
\draw (2,4) -- (0,0);
\draw (1,2) -- (2,0);
\draw (1.5,1) -- (1.25,0.5);
\draw (1.25,1.5) -- (1,1);
\draw (1.75,0.5) -- (1.5,0);
\draw (1.5,3) -- (3,0);

\draw (1.75,2.5) -- (1.5,2);
\draw (2,2) -- (1.75,1.5);
\draw (2.25,1.5) -- (2,1);
\draw (2.5,1) -- (2.25,0.5);
\draw (2.75,0.5) -- (2.5,0);

\draw (0.5,1) -- (1,0);
\draw (0.75,0.5) -- (0.5,0);
\end{tikzpicture}
\hspace{15pt}
\begin{tikzpicture}[scale=0.5]
\node at (0,2) {$\dots$};
\node at (0,0) {};
\end{tikzpicture}
\hspace{15pt}
\begin{tikzpicture}[scale=0.5]
\draw [thick] (2,4) -- (0,0);
\draw [thick](1,2) -- (2,0);
\draw [thick](1.5,1) -- (1.25,0.5);
\draw [thick](1.25,1.5) -- (1,1);
\draw [thick](1.75,0.5) -- (1.5,0);
\draw [thick](1.5,3) -- (3,0);

\draw [thick](1.75,2.5) -- (1.5,2);
\draw [thick](2,2) -- (1.75,1.5);
\draw [thick](2.25,1.5) -- (2,1);
\draw [thick](2.5,1) -- (2.25,0.5);
\draw [thick](2.75,0.5) -- (2.5,0);

\draw [thick](0.5,1) -- (1,0);
\draw [thick](0.75,0.5) -- (0.5,0);
\draw (0.625,0.25) -- (0.75,0);
\draw (0.6875,0.125) -- (0.625,0);
\draw (0.65625,0.0625) -- (0.6875,0);

\draw (1.625,0.25) -- (1.75,0);
\draw (1.6875,0.125) -- (1.625,0);
\draw (1.65625,0.0625) -- (1.6875,0);

\draw (1.625-0.25,0.25+0.5) -- (1.75-0.25,0+0.5);
\draw (1.6875-0.25,0.125+0.5) -- (1.625-0.25,0+0.5);
\draw (1.65625-0.25,0.0625+0.5) -- (1.6875-0.25,0+0.5);

\draw (1.625-0.5,0.25+1) -- (1.75-0.5,0+1);
\draw (1.6875-0.5,0.125+1) -- (1.625-0.5,0+1);
\draw (1.65625-0.5,0.0625+1) -- (1.6875-0.5,0+1);

\draw (2.625,0.25) -- (2.75,0);
\draw (2.6875,0.125) -- (2.625,0);
\draw (2.65625,0.0625) -- (2.6875,0);

\draw (2.625-0.25,0.25+0.5) -- (2.75-0.25,0+0.5);
\draw (2.6875-0.25,0.125+0.5) -- (2.625-0.25,0+0.5);
\draw (2.65625-0.25,0.0625+0.5) -- (2.6875-0.25,0+0.5);

\draw (2.625-0.5,0.25+1) -- (2.75-0.5,0+1);
\draw (2.6875-0.5,0.125+1) -- (2.625-0.5,0+1);
\draw (2.65625-0.5,0.0625+1) -- (2.6875-0.5,0+1);

\draw (2.625-0.75,0.25+1.5) -- (2.75-0.75,0+1.5);
\draw (2.6875-0.75,0.125+1.5) -- (2.625-0.75,0+1.5);
\draw (2.65625-0.75,0.0625+1.5) -- (2.6875-0.75,0+1.5);

\draw (2.625-1,0.25+2) -- (2.75-1,0+2);
\draw (2.6875-1,0.125+2) -- (2.625-1,0+2);
\draw (2.65625-1,0.0625+2) -- (2.6875-1,0+2);
\end{tikzpicture}
%\vspace{-5pt}
\caption{Scattered $\mathscr{L}$-trees of increasing scattered rank, and a $\sigma$-scattered $\mathscr{L}$-tree.% Here $r$ is a linear order in $\mathbb{L}$.
}
\label{Fig:ScatTrees}
\end{figure}
%\vspace{-10pt}
\begin{defn}\label{Defn:StructTrees}
Let $\mathbb{T}$ be a class of $\mathscr{L}$-trees, and let $\mathcal{O}$ be a concrete category. We define the new concrete category of \emph{$\mathcal{O}$-structured trees of $\mathbb{T}$}, denoted $\mathbb{T}_\mathcal{O}$ as follows. The objects of $\mathbb{T}_\mathcal{O}$ consist of pairs $\langle T,l \rangle$ such that:
\begin{itemize}
	\item $T\in \mathbb{T}$.
	\item $U_{\langle T,l\rangle}=T$.
	\item $l=\{l_v \mid v\in T\}$, where for each $v\in T$ there is some $\ra_v\in \obj{\mathcal{O}}$ such that $$l_v:\downs v \longrightarrow \ra_v $$ %a %surjective? function 
	and if $x,y\in T$ with $x>y>v$ then $l_v(x)=l_v(y)$.%, and moreover, if $l_v(x)=l_v(y)$ then either $x>_Ty$ or $y>_Tx$.
	%\item To simplify notation we will equate $\langle T,l\rangle$ with $T$.
\end{itemize}	
For $\mathcal{O}$-structured trees $\langle T,l\rangle$ and $\langle T',l'\rangle$, we let $\varphi:T\rightarrow T'$ be an embedding whenever:
	\begin{enumerate}
%		\item $\varphi$ is injective,
		\item $x\leq y $ iff $ \varphi(x)\leq \varphi(y)$,
		\item $\varphi(x\wedge y)=\varphi(x)\wedge \varphi(y)$,
		\item \label{Item:StructTrees4} for any $v\in T$, set $\theta:\range(l_v)\rightarrow \range(l'_{\varphi(v)})$ such that $$\theta(l_v(x))=l'_{\varphi(v)}(\varphi(x));$$ then $\theta$ is an embedding of $\mathcal{O}$.

	\end{enumerate}

%When $\mathcal{O}$ is a class of partial orders, we let $\leq_v$ be the order on $\range(l_v)$. Then (\ref{Item:StructTrees4}) is equivalent to: if $x,y\in \downs v$ for some $v\in T$, then $l_v(x)\leq_v l_v(y) \mbox{ iff }$ $$l_{h(v)}(h(x))\leq_{h(v)} l_{h(v)}(h(y)).$$ %(resp. $<_v$), %(resp. $<_{h(v)}$).

We call $l_v(x)$ the \emph{$v$-label} of $x$. To simplify notation, we write $T$ in place of $\langle T,l\rangle$ and always use $l_v(x)$ to denote the $v$-label of $x$, regardless of which $T\in \mathbb{T}_Q$ we are considering (it will be unambiguous since $v\in T$).
\end{defn}

Intuitively $\mathbb{T}_\mathcal{O}$ is obtained by taking $T\in \mathbb{T}$ and for each vertex $v\in T$, and ordering the successors of $v$ by some order in $\mathcal{O}$ as in Figure \ref{Fig:StrucTree}. Embedding for $\mathbb{T}_\mathcal{O}$ is then tree embedding that preserves this ordering on the successors of $v$ for every $v\in \mathbb{T}$.
However, general $\mathscr{L}$-trees may contain points with no immediate successors. To accommodate this, our labelling functions $l_v$ have domain $\downs v$ and we enforce that if $x,y\in T$ with $x>y>v$ then $l_v(x)=l_v(y)$.
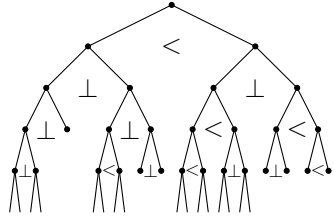
\begin{figure}
  
  \centering
 \begin{tikzpicture}[scale=0.55]
%ATREE

\draw [fill] (9+0,5-6) circle [radius=0.06];
\draw (9+0,5-6) -- (9+-2,4-6);
\draw [fill] (9+-2,4-6) circle [radius=0.06];
\draw (9+0,5-6) -- (9+2,4-6);
\draw [fill] (9+2,4-6) circle [radius=0.06];
\draw (9+2,4-6) -- (9+3,3-6);
\draw [fill] (9+3,3-6) circle [radius=0.06];
\draw (9+-2,4-6) -- (9+-3,3-6);
\draw [fill] (9+-3,3-6) circle [radius=0.06];

\draw (9+2,4-6) -- (9+1,3-6);
\draw [fill] (9+1,3-6) circle [radius=0.06];
\draw (9+-2,4-6) -- (9+-1,3-6);
\draw [fill] (9+-1,3-6) circle [radius=0.06];

\draw (9+1,3-6) -- (9+1.5,2-6);
\draw [fill] (9+1.5,2-6) circle [radius=0.06];
\draw (9+-1,3-6) -- (9+-1.5,2-6);
\draw [fill] (9+-1.5,2-6) circle [radius=0.06];

\draw (9+1,3-6) -- (9+0.5,2-6);
\draw [fill] (9+0.5,2-6) circle [radius=0.06];
\draw (9+-1,3-6) -- (9+-0.5,2-6);
\draw [fill] (9+-0.5,2-6) circle [radius=0.06];

\draw (9+3,3-6) -- (9+3.5,2-6);
\draw [fill] (9+3.5,2-6) circle [radius=0.06];
\draw (9+-3,3-6) -- (9+-3.5,2-6);
\draw [fill] (9+-3.5,2-6) circle [radius=0.06];
\draw (9+3,3-6) -- (9+2.5,2-6);
\draw [fill] (9+2.5,2-6) circle [radius=0.06];
\draw (9+-3,3-6) -- (9+-2.5,2-6);
\draw [fill] (9+-2.5,2-6) circle [radius=0.06];

\draw (9+3.5,2-6) -- (9+3.25,1-6);
\draw [fill] (9+3.25,1-6) circle [radius=0.06];

\draw (9+2.5,2-6) -- (9+2.25,1-6);
\draw [fill] (9+2.25,1-6) circle [radius=0.06];
\draw (9+2.5,2-6) -- (9+2.75,1-6);
\draw [fill] (9+2.75,1-6) circle [radius=0.06];

%\draw (9+-2.5,2-6) -- (9+-2.25,1-6);
%\draw [fill] (9+-2.25,1-6) circle [radius=0.06];
%\draw (9+-2.5,2-6) -- (9+-2.75,1-6);
%\draw [fill] (9+-2.75,1-6) circle [radius=0.06];

\draw (9+0.5,2-6) -- (9+0.25,1-6);
\draw [fill] (9+0.25,1-6) circle [radius=0.06];
\draw (9+0.5,2-6) -- (9+0.75,1-6);
\draw [fill] (9+0.75,1-6) circle [radius=0.06];

\draw (9+-0.5,2-6) -- (9+-0.25,1-6);
\draw [fill] (9+-0.25,1-6) circle [radius=0.06];
\draw (9+-0.5,2-6) -- (9+-0.75,1-6);
\draw [fill] (9+-0.75,1-6) circle [radius=0.06];

\draw (9+-3.5,2-6) -- (9+-3.25,1-6);
\draw [fill] (9+-3.25,1-6) circle [radius=0.06];

\draw (9+3.5,2-6) -- (9+3.75,1-6);
\draw [fill] (9+3.75,1-6) circle [radius=0.06];
\draw (9+-3.5,2-6) -- (9+-3.75,1-6);
\draw [fill] (9+-3.75,1-6) circle [radius=0.06];

\draw (9+1.5,2-6) -- (9+1.25,1-6);
\draw [fill] (9+1.25,1-6) circle [radius=0.06];
\draw (9+-1.5,2-6) -- (9+-1.25,1-6);
\draw [fill] (9+-1.25,1-6) circle [radius=0.06];

\draw (9+1.5,2-6) -- (9+1.75,1-6);
\draw [fill] (9+1.75,1-6) circle [radius=0.06];
\draw (9+-1.5,2-6) -- (9+-1.75,1-6);
\draw [fill] (9+-1.75,1-6) circle [radius=0.06];

\draw (9+0.25,1-6) -- (9+0.125,0-6);
\draw (9+0.25,1-6) -- (9+0.375,0-6);
\draw (9+0.75,1-6) -- (9+0.625,0-6);
\draw (9+0.75,1-6) -- (9+0.875,0-6);

\draw (9+1.25,1-6) -- (9+1.125,0-6);
\draw (9+1.25,1-6) -- (9+1.375,0-6);
\draw (9+1.75,1-6) -- (9+1.625,0-6);
\draw (9+1.75,1-6) -- (9+1.875,0-6);

%\draw (9+2.25,1-6) -- (9+2.125,0-6);
%\draw (9+2.25,1-6) -- (9+2.375,0-6);
%\draw (9+2.75,1-6) -- (9+2.625,0-6);
%\draw (9+2.75,1-6) -- (9+2.875,0-6);
%
%\draw (9+3.25,1-6) -- (9+3.125,0-6);
%\draw (9+3.25,1-6) -- (9+3.375,0-6);
%\draw (9+3.75,1-6) -- (9+3.625,0-6);
%\draw (9+3.75,1-6) -- (9+3.875,0-6);
%
%\draw (9+-0.25,1-6) -- (9+-0.125,0-6);
%\draw (9+-0.25,1-6) -- (9+-0.375,0-6);
%\draw (9+-0.75,1-6) -- (9+-0.625,0-6);
%\draw (9+-0.75,1-6) -- (9+-0.875,0-6);

\draw (9+-1.25,1-6) -- (9+-1.125,0-6);
\draw (9+-1.25,1-6) -- (9+-1.375,0-6);
\draw (9+-1.75,1-6) -- (9+-1.625,0-6);
\draw (9+-1.75,1-6) -- (9+-1.875,0-6);

%\draw (9+-2.25,1-6) -- (9+-2.125,0-6);
%\draw (9+-2.25,1-6) -- (9+-2.375,0-6);
%\draw (9+-2.75,1-6) -- (9+-2.625,0-6);
%\draw (9+-2.75,1-6) -- (9+-2.875,0-6);

\draw (9+-3.25,1-6) -- (9+-3.125,0-6);
\draw (9+-3.25,1-6) -- (9+-3.375,0-6);
\draw (9+-3.75,1-6) -- (9+-3.625,0-6);
\draw (9+-3.75,1-6) -- (9+-3.875,0-6);

\node at (9+0,4-6) {$<$};
\node at (9+2,3-6) {$\perp$};
\node at (9+-2,3-6) {$\perp$};
\node at (9+1,2-6) {$<$};
\node at (9+3,2-6) {$<$};
\node at (9+-1,2-6) {$\perp$};
\node at (9+-3,2-6) {$\perp$};

\node [scale=0.7] at (9+3.5,1-6) {$<$};
\node [scale=0.7] at (9+1.5,1-6) {$\perp$};
\node [scale=0.7] at (9+0.5,1-6) {$<$};
\node [scale=0.7] at (9+2.5,1-6) {$\perp$};

\node [scale=0.7] at (9+-3.5,1-6) {$\perp$};
\node [scale=0.7] at (9+-1.5,1-6) {$<$};
\node [scale=0.7] at (9+-0.5,1-6) {$\perp$};
%\node at (9+-2.5,1-6) {$\perp$};
\end{tikzpicture}
\caption{A tree structured by $\{\CH{2},\AC{2}\}$.}\label{Fig:StrucTree}
\end{figure}

\begin{defn}
Let $T$ be an $\mathcal{O}$-structured $\mathscr{L}$-tree, with $x\in T$ and $p\in \range(l_x)$ then we define
$${}^{p}\downs x=\{y\in T\mid (y>x)\wedge (l_x(y)=p)\}.$$

\end{defn}

%We will end up only considering well-branched structured trees. There are two reasons for this, firstly it is easier to show that these are well-behaved, and secondly because this is all we require for the operator algebra construction of section \ref{Section:OpConstruction}.

It is clear that $\mathbb{T}_Q$ is a concrete category and hence we also have defined the $Q$-coloured, $\mathcal{O}$-structured $\mathscr{L}$-trees of $\mathbb{T}$, denoted $\mathbb{T}_\mathcal{O}(Q)$. % We extend Definitions \ref{Defn:TraversingPOs}, \ref{Defn:TreeHeight}, \ref{Defn:TreeBasics} and \ref{Defn:WBScatteredTrees} to such $\mathscr{L}$-trees in the obvious way.
Finally we mention a theorem of K\v{r}\'{i}\v{z} that is fundamental to the results of this paper. % (from Section \ref{Section:LO} onwards).

\begin{thm}[K\v{r}\'{i}\v{z}, \cite{Kriz}]\label{Thm:Kriz1}
If $\mathcal{O}$ is a well-behaved concrete category with injective morphisms, then $\mathscr{R}_\mathcal{O}$ is well-behaved.
\end{thm}
\begin{proof}
See \cite{Kriz}.
\end{proof}
\begin{remark}
Louveau and Saint-Raymond proved, using a modification of Nash-Williams' original method, that if $Q$ satisfies a slight weakening of well-behaved (that is stronger than preserving bqo) then $\mathscr{R}_Q$, the class of $Q$-structured trees of $\mathscr{R}$, satisfies this same property \cite{LouveauStR}. They were unable to attain full well-behavedness and Nash-Williams' method seems to be insufficient.
\end{remark}
\begin{thm}\label{Thm:Kriz}
If $\mathcal{O}$ is a well-behaved concrete category with injective morphisms, then $\sscatt_\mathcal{O}^\omega$ is well-behaved.
\end{thm}
\begin{proof}
Notice that any $T\in \sscatt_\mathcal{O}^\omega$ is a tree of height at most $\omega$, but not necessarily rooted. Consider $\{t\mid \forall s\in T, s\not < t\}$, let $\kappa$ be the cardinality of this set, and enumerate its elements as $t_i$ for $i\in \kappa$. Given a quasi-order $Q$, let $\tau:\sscatt_\mathcal{O}^\omega(Q)\rightarrow \On(\mathscr{R}_\mathcal{O}(Q))$ be the function sending $T$ to $\langle \kappa, c\rangle$ where $c(i)=\down t_i$ for each $i\in \kappa$. Thus we have that if $\tau(S)\leq \tau(T)$ then $S\leq T$. So given a bad  $\sscatt^\omega_\mathcal{O}(Q)$-array $f$, we see that $\tau \circ f$ is a bad $\On(\mathscr{R}_\mathcal{O})$-array, then by Theorem \ref{Thm:OnWB} we have a witnessing bad $\mathscr{R}_\mathcal{O}(Q)$-array, and by Theorem \ref{Thm:Kriz1} we have a witnessing bad $Q$-array for $f$.
\end{proof}
\subsection{Encoding with structured $\mathscr{L}$-trees}
%As hinted at before, we aim to prove that $\amr$ is bqo whenever %$\CC$, $\MM$, $\RR$, $\PP$ and $\AAA$ 
%$\OA$ satisfies the correct conditions, all of which we have now defined. 
%
%The general method will be to take an element $x$ of $\amr$, and construct from it a structured $\mathscr{L}$-tree $T_x$ that contains all of the information required to describe how $x$ is built. In particular, we want that if $T_x\leq T_y$, then $x\leq y$, allowing us to reflect bad $\amr$-arrays to bad arrays for the $\mathscr{L}$-trees. Then using a theorem on the well-behavedness of some class of structured $\mathscr{L}$-trees (such as %K\v{r}\'{i}\v{z}'s Theorem in \cite{Kriz})
%Theorem \ref{Thm:Kriz}) we will have our bqo result for $\amr$. % We begin by defining a relatively simple structured tree that codes the $f^\eta$s.
We now want to take an element $x\in \amr$ and construct from it a structured $\mathscr{L}$-tree $T_x$ that contains all of the information required to describe how $x$ is built up from elements of $\AAA$, using functions from $\MM$.
We assume for the rest of this section that $\MM$ is $\RR$-iterable and $\OA$ has limits.

\begin{defn}\label{Defn:Tx}
Let $\FS$ be an composition set with $\PS$ a set of position sequences for $\FS$, and $d=\langle d^{\vec{p}}:\vec{p}\in \DS\rangle\in \dom(g^\FS)$. Suppose that for each $\vec{p}\in \PS$ we have $\eta(\vec{p})=\langle \langle f^{\vec{p}}_i,s^{\vec{p}}_i\rangle:i\in r(\vec{p})\rangle$. We define the $\MM\cup \AAA$-coloured $\PP$-structured $\clos{\RR}$-tree $T^\FS_d$ whose underlying set is $$\{\vec{p}\con \langle i\rangle\mid \vec{p}\in \PS\setminus \DS, i\in r(\vec{p})\}\cup \DS;$$
ordering $\vec{p}\con\langle i\rangle\leq \vec{q}\con \langle j\rangle$ or $\vec{p}\con\langle i\rangle\leq \vec{q}\con \langle j,u\rangle\in \DS$ iff either:
\begin{enumerate}
	\item $\vec{p}=\vec{q}$ and $j\geq i$.
	\item $\vec{p}\sis \vec{q}$ and the first element of $\vec{q}\setminus \vec{p}$ is $\langle j',u'\rangle$ with $j'\geq i$.
\end{enumerate}
We colour all $\vec{p}\con \langle i\rangle\in T^\FS_d$ by letting $$\col(\vec{p}\con \langle i\rangle)=f_i^{\vec{p}},$$
and for $\vec{p}\in \DS\subseteq T^\FS_d$ we let
$$\col(\vec{p})=d^{\vec{p}}.$$
For all $\vec{p}\in \PS\setminus \DS$, $i\in r(\vec{p})$ and $t>\vec{p}\con \langle i\rangle$ we define the labels $l_{\vec{p}\con \langle i\rangle}(t)\in \arity(f_i^{\vec{p}})$ so that
$$l_{\vec{p}\con \langle i\rangle}(t)=\left\{\begin{array}{lcl}
s_i^{\vec{p}}&:&\vec{p}\con \langle j\rangle\is t \mbox{ or } \vec{p}\con \langle j,u\rangle\is t \mbox{ for some }j>i \mbox{ and }u\in a^{\eta(\vec{p})}_j \\
u &:&\vec{p}\con\langle i,u\rangle \is t

\end{array}
\right. .
$$

\end{defn}
\begin{defn}\label{Defn:Tx1}
If $g^\FS$ is a standard decomposition function for $x\in \amrs$ and $d$ is such that $g^\FS(d)=x$, then we call $T^\FS_d$ a \emph{decomposition tree} for $x$.

%Suppose that $\OA$ has limits and 
Let $(x_n)_{n\in \omega}$ be a limiting sequence and $x\in \amri$ be the limit of this sequence. Let $q_0$, $\FS_n$ and $k^{\vec{p}}_n$ be as in Definition \ref{Defn:LimitSequence}. We set $\FS=\bigcup_{n\in \omega}\FS_n$ (so $\FS$ is a composition set). %, we let $\PS$ be a set of position sequences for $\FS$ and $\DS$ be the set of leaves of $\PS$. 
Let $\vec{p}\in \DS$ %then for some $n\in \omega$ we have $\vec{p}\in \FS_n$, and since $\vec{p}\in \DS$ we know that $\vec{p}$ has no successors in any $\FS_m$, $m>n$. I
then there must be some least $m$ such that $k^{\vec{p}}_m\in \DS$, and we let $d^{\vec{p}}=k^{\vec{p}}_m$.\footnote{Note that by (\ref{Item:LimitSeq4}) of Definition \ref{Defn:LimitSequence}, we have $k^{\vec{p}}_m=k^{\vec{p}}_n$ for any $n\geq m$.} %; if no such $m$ exists let $d^{\vec{p}}=q_0$. 
Set $d=\langle d^{\vec{p}}:\vec{p}\in \DS\rangle$. When $\FS$ and $d$ are defined in this way we call $T^\FS_d$ a \emph{decomposition tree} for $x$.
\end{defn}
%\begin{prop}
%For any $x\in \amr$ there is a decomposition tree $T$ for $x$.
%\end{prop}
%\begin{proof}
%If $x\in \amrs$ the result follows by Lemma \ref{Lemma:decompfn}. If $x\in \amri$ we define $T$ as in Definition \ref{Defn:Tx1}.
%\end{proof}
\begin{lemma}\label{Lemma:TreeRanksTheSame}
If $x\in \amr$ then there is a decomposition tree $T=T^\FS_d$ for $x$ such that $T\in \sscatt^{\clos{\RR}}_\PP(\MM\cup \AAA)$. Moreover if $x\in \amrs$ then there is a decomposition tree $T\in \scatt^{\clos{\RR}}_\PP(\MM\cup \AAA)$ for $x$ with $\Srank(T)=\rank(x)$.
\end{lemma}
\begin{proof}
Suppose $x\in \amrs$ then let $g^\FS$ be the decomposition function for $x$ as in Lemma \ref{Lemma:decompfn}, and let $d=\langle q_{\vec{p}}:\vec{p}\in \DS\rangle\in \dom(g^\FS)$ be such that $x=g^\FS(d)$ and $q_{\vec{p}}\in \AAA$ for each $\vec{p}\in \DS$. In this case we set $T=T^{\FS}_d$, so $T$ is a decomposition tree for $x$. We claim that $\Srank(T)=\rank(x)$, and prove this by induction on the rank of $x$.

If $\rank(x)=0$ then $x\in \AAA$ so that $T$ is just a single point coloured by $x$, and hence $\Srank(T)=0$. Suppose that $\rank(x)=\alpha$ and for each $y\in \CC_{<\alpha}$, and each decomposition tree $T'$ for $y$ that was constructed from the decomposition function given by Lemma \ref{Lemma:decompfn}, that we have $\rank(y)=\Srank(T')$. Then let $\zeta=\{\langle i\rangle\mid i\in r(\langle \rangle)\}\subseteq T$ so that $T$ is a $\zeta$-tree-sum of decomposition trees $T_i^u$ for some $q_{i,u}\in \CC_{<\alpha}$ $(i\in \zeta, u\in a^{\eta(\langle \rangle)}_i)$. Then using the induction hypothesis, we have
\begin{eqnarray}
\Srank(T)&=&\sup\{\Srank(T^u_i)+1\mid i\in \zeta, u\in a^{\eta(\langle \rangle)}_i\}\nonumber\\
&=&\sup\{\rank(q_{i,u})+1\mid i\in \zeta, u\in a^{\eta(\langle \rangle)}_i\}\nonumber\\
&=&\rank(x).\nonumber
\end{eqnarray}
This completes the induction, and the lemma holds for all $x\in \amrs$.
If $x\in \amri$ then a decomposition tree $T$ for $x$ was of the form $T=T_d^\FS$ for some $d$ and $\FS=\bigcup_{n\in \omega} \FS_n$, with the $\FS_n$ as from Definition \ref{Defn:LimitSequence}. By what we just proved, we have that the underlying set of $T$ is covered by $\up$-closed subsets consisting of the underling sets of $T_n=T^{\FS_n}_{d_n}\in \scatt^{\clos{\RR}}_\PP(\MM\cup \AAA)$ for some $d_n$; %By Remark \ref{Rk:DecompTreesAlwaysSSCAT} we can always assume that these $T_n$ are in $\scatt^{\clos{\RR}}_\PP(\MM\cup \AAA)$. 
hence $T\in \sscatt^{\clos{\RR}}_\PP(\MM\cup \AAA)$.
\end{proof}

%\begin{remark}
%By a simple induction it is clear that all of the trees mentioned in Definition \ref{Defn:Tx} are well-branched. Also, since each $T^n_x$ is scattered, so is each $S^n_x$, hence $T_x$ is $\sigma$-scattered, ie $T_x\in\mathbb{T}_\PP^{\clos{\RR}}(\MM\cup \AAA)$.
%\end{remark}
%\begin{remark}
%Since $x\in \amr$ always has a decomposition function, when the choice does not matter we just take any composition function and write $T_x$ for $T_x^g$.
%\end{remark}

\begin{lemma}\label{Lemma:x(t,p)}
Let $x\in \amr$ and $T=T^\FS_d$ be a decomposition tree for $x$. For any $t=\vec{p}\con \langle i\rangle \in T$ and $u\in \range(l_t)$, there exists some $x(t,u)\in \amr$ with a decomposition tree $T_{x(t,u)}$ such that $${}^u\downs t=T_{x(t,u)}.$$
Moreover if $x\in \amrs$ and $\vec{p}\neq\langle \rangle$ or $u\neq s^{\langle \rangle}_i$; then $x(t,u)\in \amrs$ with $\rank(x(t,u))<\rank(x)$.
\end{lemma}
\begin{proof}
Let $\PS$ be a set of position sequences for $\FS$ and fix $t=\vec{p}\con \langle i\rangle \in T$ and $u\in \range(l_t)$. We will find $x(t,p)$. First, if $u\neq s_i^{\vec{p}}$ we define:
\begin{itemize}
	\item $\PS_{t,u}=\{\vec{q}\in \PS\mid \vec{p}\con \langle i,u\rangle\is \vec{q}\}$,
	\item $\FS_{t,u}=\{\eta(\vec{q})\mid \vec{p}\con \langle i,u\rangle\is \vec{q}\}\subseteq \FS$,
\end{itemize}
and if $u=s_i^{\vec{p}}$ we define:
\begin{itemize}
	\item $\eta'(\vec{p})=\eta(\vec{p})^+_{i}$, and $\eta'(\vec{q})=\eta(\vec{q})$ whenever $\vec{p}\neq\vec{q}\in \PS$,
	\item $\PS_{t,u}=\{\vec{q}\in \PS\mid (\exists j>i )(\exists v\in a^{\eta(\vec{p})}_j),\vec{p}\con \langle j,v\rangle \is \vec{q} \}\cup \{\vec{p}\}$,
	\item $\FS_{t,u}=\{\eta'(\vec{q})\mid \vec{q}\in\PS_{t,u}\}$.
\end{itemize}
Now let
\begin{itemize}
	\item $\DS_{t,u}$ be the set of leaves of $\PS_{t,u}$,
	\item $d_{t,u}=\langle d^{\vec{q}}:\vec{q}\in \DS_{t,u}\rangle$, (where $d=\langle d^{\vec{q}}:\vec{q}\in \DS\rangle$),
	\item $T_{t,u}=T^{\FS_{t,u}}_{d_{t,u}}$.
\end{itemize}
Then in either case, by construction we have that $T_{t,u}={}^u\downs t$.

If $g^{\FS_{t,u}}(d_{t,u})\in \amrs$ then set $x(t,u)=g^{\FS_{t,u}}(d_{t,u})$, so we have that $T_{t,u}$ is by definition a decomposition tree for $x(t,u)$. Moreover if $\vec{p}\neq \langle \rangle$ or $u\neq s_i^{\langle \rangle}$ then there is some $i_0\in r(\langle \rangle)$ and $u_0\in \range(l_{\langle i_0\rangle})$ such that $T_{t,u}\subseteq {}^{u_0}\downs i_0$ and hence by Lemma \ref{Lemma:TreeRanksTheSame} we have $$\rank(x(t,u))=\Srank(T_{t,u})\leq \Srank({}^{u_0}\downs i_0)<\Srank(T)=\rank(x).$$

%It just remains to find $x(t,u)$ when $x\in \amri$. We will set $x(t,u)$ to be the limit of a limiting sequence constructed as follows. Let $\PS^n_{t,u}$ be the set of sequences from $\PS_{t,u}$ of length at most $\length(\vec{p})+1+n$, let $\FS^n_{t,u}=\{\eta(\vec{q})\mid \vec{q}\in \PS^n_{t,u}\}$, and let $\DS_{t,u}^n$ be the set of leaves of $\PS_{t,u}^n$. Now we let $e^n_{t,u}=\langle e^{\vec{q}}:\vec{q}\in \DS^n_{t,u}\rangle$ where $$e^{\vec{q}}=\left\{\begin{array}{lcl}
%q_0&:&\vec{q}\notin \DS_{t,u}\\
%d^{\vec{q}}&:&\vec{q}\in \DS_{t,u}
%\end{array}\right..$$

It just remains to find $x(t,u)$ when $x\in \amri$. First suppose that $x$ is the limit of $(x_n)_{n\in \omega}$, then consider $x_n(t,u)$ for every $n$ large enough so that $t\in T^{\FS_n}$. Then since $(x_n)_{n\in \omega}$ was a limiting sequence; for some $m\in \omega$ we have that $(x_n(t,u))_{n\geq m}$ is a limiting sequence. We let $x(t,u)$ be its limit,
%
%We have now constructed a limiting sequence $(g^{\FS^n_{t,u}}(e^n_{t,u}))_{n\in \omega}$, we let $x(t,u)$ be its limit, 
then by construction $x(t,u)$ has a decomposition tree equal to $$T_{d_{t,u}}^{\bigcup_{n\in \omega}\FS^n_{t,u}}=T^{\FS_{t,u}}_{d_{t,u}}=T_{t,u}={}^u\downs t.$$
\end{proof}

\begin{defn}
For a given $x\in \amr$ and a decomposition tree $T$ for $x$, let $\zeta$ be a $\up$-closed chain of $T$ that contains no leaf. For each $i\in \zeta$, if $i$ is not the maximum element of $\zeta$ then let $j_i$ be an arbitrary element of $\zeta\cap \down i$, otherwise let $j_i$ be an arbitrary element of $\down i$. We then define $$\eta(\zeta)=\langle \langle \col(i),l_i(j_i)\rangle :i\in \zeta\rangle$$
which is a composition sequence since for each $i\in \zeta$ we have $\col(i)$ is a function of $\MM$, and since the labels below $i$ are elements of $\arity(\col(i))$. We note that the choice of $j_i$ made no difference to $l_i(j_i)$, except possibly when $i$ is maximal, in which case only $s_i$ is ambiguous - but this makes no difference to $f^{\eta(\zeta)}$.

To simplify notation, for each $i\in \zeta$, we let $f^\zeta_i=\col(i)$, $s^{\zeta}_i=l_i(j_i)$ and $a_i^\zeta=a_i^{\eta(\zeta)}$. We also define $k^\zeta=\langle x(i,u):i\in \zeta, u\in a^{\eta(\zeta)}_i\rangle\in \dom(f^{\eta(\zeta)})$, %, where for each $i\in \zeta$ we have $k^\zeta_i\in \CC^{a^{\zeta}_i}$ with $\col(p)=x(i,p)$ for each $p\in a^{\zeta}_i$. 
where $x(i,u)$ is as from Lemma \ref{Lemma:x(t,p)}.
\end{defn}
We now require the following lemma, which will allow us to swap composition sequences.

\begin{lemma}\label{Lemma:ChangeOfPos}
Let $\MM$ be $\RR$-iterable, $x\in \amrs$ have a decomposition tree $T=T_d^\FS$ and let $\zeta$ be a $\up$-closed chain of $T$, that is either maximal or has a maximum element, then
$$x=f^{\eta(\zeta)}(k^\zeta).$$
Moreover if $\OA$ has limits, then the conclusion holds for all $x\in \amr$.
\end{lemma}
\begin{proof}
Let $\xi=\{\langle i\rangle\mid i\in r(\langle \rangle)\}$ be the chain satisfying $\eta(\xi)=\eta(\langle \rangle)$. % Note that we do not lose generality with this assumption, because if we can prove the result for a single $\xi$ and general $\zeta$, then we will have the result for general $\xi$. 
We assume without loss of generality that $\zeta \not \is \xi$, since otherwise the lemma follows from a simple application of $\RR$-iterability (noticing that $\eta(\zeta)=\eta(\xi)^-_{\max(\zeta)}$). We also assume that $\xi\not \is \zeta$ since $\xi\is \zeta$ is impossible unless $\xi$ has a maximum element,\footnote{This is by definition of the order on decomposition trees.} in which case the lemma follows similarly by $\RR$-iterability. Clearly the Lemma holds when $\rank(x)=0$ so suppose for induction that $\rank(x)=\alpha$ and the lemma holds for all $y\in \CC_{<\alpha}$.

By Lemma \ref{Lemma:TreeRanksTheSame} we have $T\in\sscatt^{\clos{\RR}}_\PP(\MM\cup \AAA)$ and thus $\xi\cap \zeta$ has a maximal element $m$. 
 By Proposition \ref{Prop:Reduction}, we have that $x=f^{\eta(\xi)}(k^\xi)$ and $k^\xi\in \dom(f^{\eta(\xi)})_{<\alpha}$.
For each $\chi\in \{\zeta,\xi\}$ and $u\in A^{\eta(\chi)}$ we define %$w^\chi\in \dom(f^{\eta(\chi)^-_m})$ so that $w^\chi=\langle w^\chi_{u}:u\in A^{\eta(\chi)^-_m}_i\rangle$ by letting 
$w^\chi_{u}=x(i,u)$ whenever $i$ is such that $u\in a_i^{\eta(\chi)}$, and %$w^\xi_m\in \CC^{\arity(f_m^\xi)}$ be such that $w^\xi_m\restriction a^{\xi}_m=k^\xi_m$ and %for the remaining $s\in k^\eta_i$ (ie $s$ is in position $s_i$) we have
 $$w^\chi_{s^\chi_m}=f^{\eta(\chi)^+_m}(\langle w^\chi_u:u\in A^{\eta(\chi)^+_m}\rangle).$$
We then set $w^\chi=\langle w^\chi_{u}:u\in A^{\eta(\chi)^-_m}_i\rangle\in \dom(f^{\eta(\chi)})$. %Now set $w^\xi_i=k^\xi_i$ whenever $i<m$. 
%where $w_u^\xi=x(i,u)$ whenever $u\in a_i^{\eta(\zeta)^+_m}$. 
Thus, since $\MM$ is $\RR$-iterable, we have $$f^{\eta(\xi)}(k^\xi)=f^{\eta(\xi)^-_m}(w^\xi) \hspace{10pt}\mbox{ and }\hspace{10pt} f^{\eta(\zeta)}(k^\zeta)=f^{\eta(\zeta)^-_m}(w^\zeta).$$
%We define similarly $w^\zeta\in \dom(f^{\eta(\zeta)})$, and $w^\zeta_u$ $(u\in A^\eta\cup\{s_m^\zeta\})$ with $\zeta$ in place of $\xi$. 
%So by the same argument, we have
%$$f^{\eta(\zeta)}(k^\zeta)=f^{\eta(\zeta)^-_m}(w^\zeta).$$
Clearly $\eta(\xi)^-_m=\eta(\xi\cap \zeta)=\eta(\zeta)^-_m$ since $m$ was the maximal element of $\xi\cap \zeta$. It remains to verify that $w^\xi=w^\zeta$. By definition of $m$, for all $i< m$ and $u\in a^\xi_i=a^\zeta_i$, we have $w_{u}^\xi =x(i,u)=w_{u}^\zeta$. We also have $w_{u}^\xi=x(i,u)=w_{u}^\zeta$ for $u\in a^\xi_m\cap a^\zeta_m$. So the only possible cases in which $w_{u}^\xi\neq w_{u}^\zeta$ are when $u\in \{s^\zeta_m,s^\xi_m\}$, % Since $f^\xi_m=f^\zeta_m$, these functions have the same arity, hence $w_m^\xi$ and $w_m^\zeta$ have the same structure, and it just remains to verify their colours are equivalent.
hence it only remains to verify that they agree here too.
%We have now that $w_m^\xi\restriction a^{\xi}_m=k^\xi_m$ and $w_m^\zeta\restriction a^{\zeta}_m=k^\zeta_m$. So we see that the elements of $w_m^\xi$ and $w_m^\zeta$ share the same colours on the set $$D=\dom(f_m^\xi)\setminus \{s_m^\xi, s_m^\zeta\}$$ since for $p\in D$, considering $p$ as an element of either $w_m^\xi$ or $w_m^\zeta$, in both cases we see that the colour of $p$ is precisely $x(m,p)$. 
%There are now two remaining colours to verify. Firstly, we want that $\col_{w_m^{\xi}}(s^\xi_m)=\col_{w_m^{\zeta}}(s^\xi_m)$, and secondly that $\col_{w_m^{\xi}}(s^\zeta_m)=\col_{w_m^{\zeta}}(s^\zeta_m)$.

Since $\xi \not \is \zeta$ and $\zeta \not \is \xi$ we have $s^\xi_m\neq s^\zeta_m$, and therefore $s^\xi_m\in a^\zeta_m$. Thus 
$$w_{s^\xi_m}^{\zeta}=x(m,s^\xi_m).$$ We also know 
$$w^\xi_{s^\xi_m}=f^{\eta(\xi)^+_m}(\langle w^\xi_u:u\in A^{\eta(\xi)^+_m}\rangle)=f^{\eta(\xi)^+_m}(\langle x(i,u):i\in \xi, i> m, u\in a^{\eta(\xi)^-_m}_i\rangle)$$
and hence by the construction of $x(m,s^\xi_m)$ from Lemma \ref{Lemma:x(t,p)}, we see that 
$$w^\xi_{s^\xi_m}=x(m,s^\xi_m).$$
Hence indeed we have $w^{\xi}_{s^\xi_m}=w^{\zeta}_{s^\xi_m}$. 

Similarly, we have that $w_{s^\zeta_m}^\xi=x(m,s_m^\zeta)$ and $w^\zeta_{s^\zeta_m}=f^{\eta(\zeta)^+_m}(\langle w^\zeta_u:u\in A^{\eta(\zeta)^+_m}\rangle)$. It remains only to show that these are equal. Note that by Lemma \ref{Lemma:x(t,p)}, since $s_m^\zeta\neq s_m^\xi$ we have that $\rank(x(m,s_m^\zeta))<\alpha$. We also have that $\eta(\zeta)^+_m=\eta(\zeta\cap \down m)$ and $\zeta\cap \down m$ is a $\up$-closed chain of ${}^{s_m^\zeta}\downs m$, which is a decomposition tree for $x(m,s_m^\zeta)$. So by the induction hypothesis we have $$w_{s^\zeta_m}^\xi=x(m,s_m^\zeta)=f^{\eta(\zeta)^+_m}(\langle w^\zeta_u:u\in A^{\eta(\zeta)^+_m}\rangle)=w^\zeta_{s^\zeta_m}.$$
So we have verified that $f^{\eta(\xi)^-_m}(w^\xi)=f^{\eta(\zeta)^-_m}(w^\zeta)$, and therefore $x=f^{\eta(\xi)}(k^\xi)=f^{\eta(\zeta)}(k^\zeta)$.

Now suppose that $\OA$ has limits and let $x$ be the limit of $(x_n)_{n\in \omega}$. For each $n\in \omega$ let $T_n$ be the corresponding decomposition tree for $x_n$, such that the underlying set of $T$ is the union of the underlying sets of the $T_n$ $(n\in \omega)$. Then $\zeta=\bigcup_{n\in \omega}\zeta_n$ where for each $n\in \omega$, $\zeta_n$ is a $\up$-closed, chain of $T_n$ which is either maximal or has a maximal element. So by what we just proved, we have that $x_n=f^{\eta(\zeta_n)}(k^{\zeta_n})$. But we know by Definition \ref{Defn:HasLimits}, that the limit of the limiting sequence $(f^{\eta(\zeta_n)}(k^{\zeta_n}))_{n\in \omega}$ is precisely $f^{\eta(\zeta)}(k^\zeta)$. Therefore since limits were unique, we have $f^{\eta(\zeta)}(k^\zeta)= x$ as required.
\end{proof}
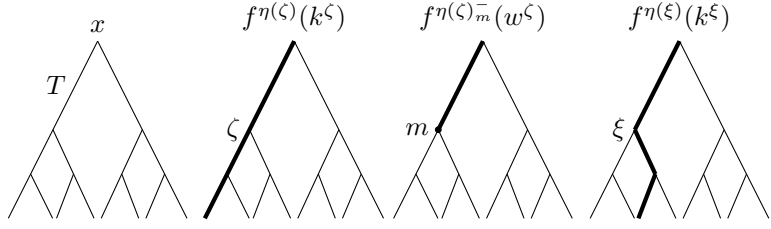
\begin{figure}
  
  \centering
    \begin{tikzpicture}[scale=0.47]
%%%
%%%	\node [left] at (1,5) {$f^\eta$};
%%%\draw [thick] (0.1,0.5) -- (1,5);
%%%\draw [fill] (0.8,4) circle [radius=0.06];
%%%\node [left] at (0.8,4) {$i$};
%%%	\draw  (0.8,4) -- (0.16,3.3);
%%%	\draw  (0.8,4) -- (1.16,3.3);
%%%	\draw  (0.8,4) -- (1.66,3.3);
%%%	\draw  (0.8,4) -- (2.16,3.3);
%%%	\draw  (0.8,4) -- (2.66,3.3);
%%%	\draw [fill] (2.16,3.3) circle [radius=0.06];
%%%	\node [below left] at (2.16,3.3) {$f^\nu$};
%%%%	\node [left] at (0.66,3.3) {$s_i$};
%%%%	\node [rotate=90, scale=1.5] at (1.93,3.1) {$\Bigg \{$};
%%%%	\node at (1.93,2.8) {$a_i$};
\node [above] at (0,5) {$x$};
\draw (0,5) -- (-2.5,0);
\draw (0,5) -- (2.5,0);
\draw (-1.25,2.5) -- (-0.1,0);
\draw (1.25,2.5) -- (0.1,0);
\draw (-1.875,1.25) -- (-1.26,0);
\draw (1.875,1.25) -- (1.35,0);
\draw (-0.675,1.25) -- (-1.15,0);
\draw (0.675,1.25) -- (1.15,0);
\node [left] at (-0.625,3.75) {$T$};

\node [above] at (5.5,5) {$f^{\eta(\zeta)}(k^\zeta)$};
\draw [ultra thick] (5.5+0,5) -- (5.5-2.5,0);
\draw (5.5+0,5) -- (5.5+2.5,0);
\draw (5.5-1.25,2.5) -- (5.5-0.1,0);
\draw (5.5+1.25,2.5) -- (5.5+0.1,0);
\draw (5.5-1.875,1.25) -- (5.5-1.26,0);
\draw (5.5+1.875,1.25) -- (5.5+1.35,0);
\draw (5.5-0.675,1.25) -- (5.5-1.15,0);
\draw (5.5+0.675,1.25) -- (5.5+1.15,0);
\node [left] at (5.5-1.25,2.5) {$\zeta$};
\end{tikzpicture}
\begin{tikzpicture}[scale=0.47]
\node [above] at (0,5-6) {$f^{\eta(\zeta)^-_m}(w^\zeta)$};
\draw (0,5-6) -- (-2.5,0-6);
\draw [ultra thick] (0,5-6) -- (-1.25,2.5-6);
\draw (0,5-6) -- (2.5,0-6);
\draw (-1.25,2.5-6) -- (-0.1,0-6);
\draw (1.25,2.5-6) -- (0.1,0-6);
\draw (-1.875,1.25-6) -- (-1.26,0-6);
\draw (1.875,1.25-6) -- (1.35,0-6);
\draw (-0.675,1.25-6) -- (-1.15,0-6);
\draw (0.675,1.25-6) -- (1.15,0-6);
\draw [fill] (-1.25,2.5-6) circle [radius=0.08];
%\node [left] at (-0.625,3.75-6) {$\eta(\zeta)^-_m$};
\node [left] at (-1.25,2.5-6) {$m$};

\node [above] at (5.5,5-6) {$f^{\eta(\xi)}(k^\xi)$};
\draw (5.5+0,5-6) -- (5.5-2.5,0-6);
\draw [ultra thick] (5.5,5-6) -- (5.5-1.25,2.5-6);
\draw (5.5+0,5-6) -- (5.5+2.5,0-6);
\draw (5.5-1.25,2.5-6) -- (5.5-0.1,0-6);
\draw [ultra thick] (5.5-1.25,2.5-6) -- (5.5-0.675,1.25-6);
\draw (5.5+1.25,2.5-6) -- (5.5+0.1,0-6);
\draw (5.5-1.875,1.25-6) -- (5.5-1.26,0-6);
\draw (5.5+1.875,1.25-6) -- (5.5+1.35,0-6);
\draw [ultra thick] (5.5-0.675,1.25-6) -- (5.5-1.15,0-6);
\draw (5.5+0.675,1.25-6) -- (5.5+1.15,0-6);
\node [left] at (5.5-1.25,2.5-6) {$\xi$};
\end{tikzpicture}
\caption{The method of the proof of Lemma \ref{Lemma:ChangeOfPos}.}
\end{figure}

\section{The construction is bqo}\label{Section:WBConstruction}
Now we aim to prove that if $x,y\in \amr$ have decomposition trees $T_x$ and $T_y$ respectively, then $x\leq y$. This will allow us to reflect bad $\amr$-arrays to bad arrays for the trees. Then we can use a structured tree theorem such as Theorem \ref{Thm:Kriz} or Theorem \ref{Thm:TLPWell-Behaved} in order to show that $\amr$ is bqo.
\subsection{The non-limit case}

%%%%%\begin{lemma}
%%%%%Suppose that $\OA$ has limits, $x\in \amri$ has a decomposition tree $T$, and let $\zeta$ be a $\up$-closed chain of $T$ that is either maximal or has a maximal element, then
%%%%%$$f^{\eta(\zeta)}(k^\zeta)= x.$$
%%%%%\end{lemma}
%%%%%\begin{proof}
%%%%%
%%%%%\end{proof}

%\begin{cor}\label{Cor:TSizeReduction}
%Suppose $\OA$ is infinitely extensive. Let $x\in \amrs$ and $t\in T^n_x$ for some $n\in \omega$ be such that $\col(t)\in \CC$. Then we have that $\col(t)\leq \col(P_x(t))$.
%\end{cor}
%\begin{proof}
%Let $u=P_x(t)$, so by the construction of $T_x$, there is a composition sequence $\eta$, some $k\in \dom(f^\eta)$ with $k=\langle k_j:j<r\rangle$, some $i\in r$, and some $p\in k_i$, such that $$\col(u)=f^{\eta}(k)\mbox{ with }\col(t)=\col(p).$$ %(ie $\eta$ is the central branch of $H(\col(t_1))$).
%So by Lemma \ref{Lemma:RealInfExtensive}, we see that $\col(t)\leq \col(u)$.
%\end{proof}

\begin{thm}\label{Thm:Non-Limit}
Suppose $\OA$ is infinitely extensive. If $x\in \amrs$ and $y\in \amr$ have decomposition trees $T_x$ and $T_y$ respectively such that $T_x\leq T_y$, then $x\leq y$.
\end{thm}
\begin{proof}
Suppose $x,y$ are as in the statement of the theorem. We prove $x\leq y$ by induction on $\rank(x)$. 
So for the base case we assume $\rank(x)=0$, i.e. that $x\in \AAA$, so we have that $T_x$ is just a single point coloured by $x$. If $T_x\leq T_y$, then there must be some $t_0\in T_y$ such that $x\leq \col(t_0)$ and thus necessarily $t_0=\vec{p}\con \langle i,u\rangle$ is a leaf of $T_y$. %Let $n\in \omega$ be least such that $t_0\in T^n_y$. 
%Now let $t_1=P_x(t_0)$, thus we have by Corollary \ref{Cor:TSizeReduction} that $\col(t_0)\leq \col(t_1)$. For each $j< n$ set $t_{j+1}=P_x(t_j)$. Hence we have $t_{i}\in T^{n-i}_y$ for each $i\leq n$ and $$x\leq \col(t)\leq \col(t_1)\leq ...\leq \col(t_n).$$
%But since $T^0_y$ was just a point coloured by $y$, we know that $\col(t_n)=y$ and therefore $x\leq y$ as required. 
Let $\xi=\up t_0$, thus by Lemma \ref{Lemma:ChangeOfPos} we have that $y=f^{\eta(\xi)}(\langle y(j,u):j\in \xi, u\in a^{\xi}_j\rangle)$. Now since $y(\vec{p}\con \langle i\rangle,u)=\col(t_0)$ we have by Lemma \ref{Lemma:RealInfExtensive} that $$x\leq \col(t_0)\leq f^{\eta(\xi)}(k^\xi)=y$$
which gives the base case.

Now suppose that $\rank(x)=\alpha$ and whenever $\rank(x')<\alpha$, $T_{x'}\leq T_{y'}$ and $y'\in \amr$ we have $x'\leq y'$. We set $\eta=\eta(\langle \rangle)$ as from Proposition \ref{Prop:Reduction}, so for some $k=\langle q_u:u\in A^\eta\rangle\in\dom(f^\eta)_{<\alpha}$, we have that $x=f^\eta(k)$. We also let $\chi=\{\langle i\rangle\mid i\in r(\langle \rangle)\}\subseteq T_x$, $\eta=\langle \langle f_i,s_i\rangle:i\in r\rangle$ and %$r$ be the length of $\eta$ and 
$\varphi$ be an embedding witnessing $T_x\leq T_y$. We notice that $\eta=\eta(\chi)$.%We have by Lemma \ref{Lemma:x(t,p)}, that each $t\in T_x\setminus C_x$ and $p\in \range(l_t)$ that $\rank(x(t,s))<\rank(x)$.

For $\langle i\rangle \in \chi$ we denote by $\varphi_i$ the embedding from $\range(l_{\langle i\rangle})$ to $\range(l_{\varphi(\langle i\rangle)})$ induced by the structured tree embedding $\varphi$. We have, since $\varphi$ was an embedding, that for each $p\in \range(l_{\langle i\rangle})$, $$\varphi({}^p\downs \langle i\rangle )\subseteq {}^{\varphi_{ i}(p)}\downs \varphi(\langle i\rangle).$$
By Lemma \ref{Lemma:x(t,p)}, this implies $T_{x(\langle i\rangle ,p)}\leq T_{y(\varphi(\langle i\rangle),\varphi_i(p))}$ where $T_{x(\langle i\rangle ,p)}$ and $T_{y(\varphi(\langle i\rangle),\varphi_i(p))}$ are decomposition trees for $x(\langle i\rangle ,p)$ and $y(\varphi(\langle i\rangle),\varphi_i(p))$ respectively. Moreover, whenever $p\neq s_i$, or $\langle i\rangle$ is the maximum element of $\chi$, we have ${}^p\downs i\cap \chi=\emptyset$; hence $\rank(x(i,p))<\alpha$ by Lemma \ref{Lemma:x(t,p)}. Therefore, by the induction hypothesis, whenever $p\in a^\eta_i$ we have that 
\begin{equation}\label{Eqn:Non-Limit LowerRanks}
x(\langle i\rangle ,p) \leq y(\varphi(\langle i\rangle),\varphi_i(p))
\end{equation}

Now let $$\zeta=\bigcup_{t\in \chi}\up \varphi(t)$$
so that by Lemma \ref{Lemma:ChangeOfPos} we have that $y=f^{\eta(\zeta)}(k^\zeta)$.
%
%
%Let $n$ be least such that there exists some $u\in \zeta$ with $P_y(u)\in T_y^{n}$. For this $u$, we set $y_0=\col(u)$. By repeated applications of Corollary \ref{Cor:TSizeReduction} similarly to the base case, we see that $y_0\leq y$. By Lemma \ref{Lemma:ChangeOfPos}, we have $y_0=f^{\eta(\zeta)}(k^\zeta)$. 
Let $\varphi'$ be the embedding from $\ot(\chi)$ to $\ot(\zeta)$ induced by $\varphi$. It then simple to verify that $\varphi'$ and the $\varphi_i$ witness the fact that $\eta=\eta(\chi)\leq \eta(\zeta)$. Then using (\ref{Eqn:Non-Limit LowerRanks}) and since $\OA$ is infinitely extensive, we have $$x=f^\eta(k)\leq f^{\eta(\zeta)}(k^\zeta)=y$$
as required.
\end{proof}
We immediately obtain the following bqo result.

\begin{cor}
Suppose $\OA$ is infinitely extensive. Also suppose that $\sscatt_\PP^{\clos{\RR}}$ is well-behaved, $\MM$ is bqo and $\AAA$ is bqo. Then $\amrs$ is bqo.
\end{cor}
\begin{proof}
Suppose we had a bad array $f:[\omega]^\omega\rightarrow\amrs$. Let $g:[\omega]^\omega\rightarrow \sscatt_\PP^{\clos{\RR}}(\MM\cup \AAA)$ be defined by letting $g(X)$ be a decomposition tree for $f(X)$. Then by Theorem \ref{Thm:Non-Limit}, $g$ must be bad. Thus since $\sscatt_\PP^{\clos{\RR}}$ is well-behaved there must be a bad $\MM\cup \AAA$-array, which is a contradiction of Theorem \ref{Thm:U bqo} since $\MM$ and $\AAA$ were bqo.
\end{proof}

We now present a simple application of Theorem \ref{Thm:Non-Limit}. We will show that if a class of linear orders $\mathbb{L}$ is well-behaved, then $\clos{\mathbb{L}}$ is well-behaved. This will allow us to give nicer descriptions of the classes that we will construct in section \ref{Section:PO}, as well as expanding on some classical results. So for the rest of this subsection we let $\mathbb{L}$ be a class of linear orders and $Q$ be a quasi-order.
%We now define the parameters $\CC, \AAA, \MM, \PP$ and $\RR$, so that the construction will give us the class $\clos{\mathbb{L}}(Q)$ for a given quasi-order $Q$, and we aim to use Theorem \ref{Thm:Limit}. We include $Q$-colourings because this will allow us to show that the class $\clos{\mathbb{L}}$ is well-behaved (Theorem \ref{Thm:LbarWB}). So, throughout this section we let:
%\begin{enumerate}
%	\item $Q$ be an arbitrary quasi-order;
%	\item $Q'$ be $Q$ with an added minimal element $-\infty$;
%	\item $\mathbb{L}$ be some class of linear orders;
%	\item $\CC=\mathscr{L}(Q')$ be the class of all $Q'$-coloured linear orders;
%	\item $\AAA=Q'^1$;
%	\item $\PP=\mathbb{L}\cup \omega$;
%	\item $\MM$ be the set of $L$-sums for all $L\in\PP$, as defined in Definition \ref{Defn:Sums}, inheriting colours;
%	\item $\RR=\{1\}$.
%\end{enumerate}
\begin{defn}
We define $\mathbb{L}_0=\mathbb{L}\cup \omega$, and for $\alpha\in \On$, $$\mathbb{L}_{\alpha+1}=\left\{\sum_{i\in L}L_i\mid L\in \mathbb{L}_0,L_i\in \mathbb{L}_\alpha\right\}$$ and for limit $\lambda$, $\mathbb{L}_\lambda=\bigcup_{\gamma<\lambda}\mathbb{L}_\gamma$ finally set $\mathbb{L}_\infty=\bigcup_{\alpha\in \On}\mathbb{L}_\alpha$.
For $x\in \mathbb{L}_\infty$ we say that $\rank_{\mathbb{L}}(x)=\alpha$ iff $\alpha$ is least such that $x\in \mathbb{L}_\alpha$.
\end{defn}
\begin{prop}\label{Prop:Linfty=amrs}
Let $\OA=\OA^Q_\mathbb{L}$, then $\mathbb{L}_\infty(Q)=\amrs$.
\end{prop}
\begin{proof}
This is by construction, considering Remark \ref{Rk:amrs}.
\end{proof}
\begin{lemma}\label{Lemma:Lbar=Linfty}
$\clos{\mathbb{L}}=\mathbb{L}_\infty$.
\end{lemma}
\begin{proof}
Clearly $\mathbb{L}\cup \omega\subseteq \mathbb{L}_\infty\subseteq \clos{\mathbb{L}}$. So if we can show that $\mathbb{L}_\infty$ is closed under sums, we will have the lemma. So let $y\in \mathbb{L}_\infty$, and $y_j \in \mathbb{L}_\infty$ for each $j\in y$, and we claim that $\sum_{j\in y}y_j\in \mathbb{L}_\infty$. 
%Then since $y_i\in \mathbb{L}_\infty$ we have $y_i=\sum_{j\in y_i^0}y_{\langle i,j\rangle}$ for some $y_{\langle i,j\rangle}$ of rank less than $\rank(y)$, and $y_i^0\in \mathbb{L}_0$. Then $x=\sum_{i\in y^0} \sum_{j\in y^0_i}y_{\langle i,j\rangle}$, so that $x\in \mathbb{L}_\infty$ whenever we have that for each $i\in y^0$ and $j\in y_i^0$ we have $y_{\langle i,j\rangle}\in \mathbb{L}_\infty$. Repeating this process, for each sequence $s$ writing each $y_s=\sum_{j\in y^0_s}y_{s\con \langle j\rangle }$ for some $y_{s\con \langle j\rangle }$ of lower rank than $y_s$ and $y^0_s\in \mathbb{L}_0$. After each step we have reduced the lemma to show that the $y_s\in \mathbb{L}_\infty$ for each of the longest sequences $s$. Whenever $s\is t$, we have $\rank_\mathbb{L}(y_s)>\rank_\mathbb{L}(y_t)$ so the induction stops when the $y_s$ are in $\mathbb{L}_0\subseteq \mathbb{L}_\infty$, this gives the lemma.
Let $y_{\langle \rangle}=y$ and suppose inductively that we have defined $y_s\in \mathbb{L}_\infty\setminus \mathbb{L}_0$ for some sequence $s$. Then we can find $y'_s\in \mathbb{L}_0$ and for each $i\in y'_s$, we find $y_{s\con \langle i\rangle}\in \mathbb{L}_\infty$ with $\rank_\mathbb{L}(y_{s\con \langle i\rangle})<\rank_\mathbb{L}(y_s)$; such that $y_s=\sum_{i\in y'_s}y_{s\con\langle i\rangle}$. Thus $\sum_{j\in y_s}y_j=\sum_{i\in y'_s}\sum_{j\in y_{s\con \langle i\rangle}}y_j$. So if every $y_{s\con \langle i\rangle}\in \mathbb{L}_0$ then $\sum_{i\in y_s}y_i\in \mathbb{L}_\infty$. Since the ranks are well-founded $y_s\in \mathbb{L}_0$ for all of the longest sequences $s$ generated by this induction. So by induction, $\sum_{j\in y}y_j\in \mathbb{L}_\infty$, which gives the lemma.
\end{proof}

\begin{thm}\label{Thm:LbarWB}
If $\mathbb{L}$ is well-behaved, then $\clos{\mathbb{L}}$ is well-behaved.
\end{thm}
\begin{proof}
Let $\OA=\OA^Q_{\mathbb{L}}=\langle \CC,\AAA,\PP,\MM,\RR\rangle$. By Lemma \ref{Lemma:Lbar=Linfty} and Proposition \ref{Prop:Linfty=amrs} we have that $\clos{\mathbb{L}}(Q')=\amrs$. Suppose we have a bad $\clos{\mathbb{L}}$-array $f$, then define $g(X)$ to be a decomposition tree for $f(X)$. First we note that by Definition \ref{Defn:Tx}, it is clear that the leaves of the tree $g(X)$ are coloured by some element of $\AAA$ used to construct $f(X)$. Thus the colours of the linear order $f(X)$ are colours of colours of elements of $g(X)$. Now using Lemma \ref{Lemma:LOinfExt} we can apply Theorem \ref{Thm:Non-Limit}, in order to see that $g$ is a bad array with range in $\scatt^{\clos{\RR}}_{\PP}(\MM\cup \AAA)$.

We know that $\clos{\RR}=\omega$ and hence is well-behaved. We also have that $\PP=\mathbb{P}$ is well-behaved and has injective morphisms; hence by Theorem \ref{Thm:Kriz}, there is a witnessing bad $\MM\cup \AAA$-array $g'$. Since the order on $\MM$ is isomorphic to the order on $\mathbb{L}$, we know that $\MM$ is bqo. Hence by Theorem \ref{Thm:U bqo}, we can find a witnessing bad $\AAA$-array $f'$. Since the colours from each $g(X)$ that were in $\AAA$ were points of $f(X)$ we now know that $f'$ is a witnessing bad array for $f$. Then since $\AAA=Q^1$ this clearly gives a witnessing bad $Q$-array just passing to colours. Hence any bad $\clos{\mathbb{L}}(Q)$-array admits a witnessing bad $Q$-array, i.e. $\clos{\mathbb{L}}$ is well-behaved.
\end{proof}
\begin{remark}
Without much more difficulty, using we could show that the class of countable unions of elements of $\clos{\mathbb{L}}$ is also well-behaved, by considering limiting sequences $(x_n)_{n\in \omega}$ to be such that $x_n\subseteq x_{n+1}$ for each $n\in \omega$ and defining limits as the countable union. Then it is relatively simple to verify that $\OA$ is infinitely extensive and has nice limits. The result then follows similarly to Theorem \ref{Thm:LbarWB} (using Theorem \ref{Thm:Limit} in place of Theorem \ref{Thm:Non-Limit}). We omit the proof because this is implied by Theorem \ref{Thm:MLPisWB2}.
\end{remark}
\subsection{The limit case}

\begin{thm}\label{Thm:Limit}%\label{Lemma:LimitChains}
Let $\OA$ be extendible and have nice limits. If $x,y\in\amr$ have decomposition trees $T_x$ and $T_y$ respectively such that $T_x\leq T_y$, then $x\leq y$.
\end{thm}
\begin{proof}
%We note that since every decomposition function $g_n$ for the sequence $(x_n)_{n\in \omega}$ was a composition of some $f^\eta$; by extending the limiting sequence if necessary, we can assume that each $g_{n+1}$ is $g_n$ composed with just one more $f^\eta$ in every argument.
By Theorem \ref{Thm:Non-Limit}, it only remains to check that the theorem holds when $x\in \amri$. So suppose that $x$ is a limit of the limiting sequence $(x_n)_{n\in \omega}$.
For each $n\in \omega$ let $T_{n}=T^{\FS_n}_{d_n}$ be a decomposition tree for $x_n$. Since $T_{n}\subseteq T_x\leq T_y$ and each $x_n\in \amrs$ we have by Theorem \ref{Thm:Non-Limit} that $x_n\leq y$ for every $n$. Let $\varphi_n:x_n\rightarrow y$ be the embedding granted by infinite extensivity that was used in the final line of the proof of Theorem \ref{Thm:Non-Limit}.

By the definition of a limiting sequence, we see that for each $n\in \omega$ there are $k,k'\in \dom(g_n)$ such that $x_n=g_n(k)$ and $x_{n+1}=g_n(k')$. If $k=\langle q_u:u\in \DS_n\rangle$ and $k'=\langle q'_u:u\in \DS_n\rangle$ then for each $u\in \DS_n$ we have $q_u\leq q'_u$ since either $q_u=q'_u$ or $q_u$ is the minimum element $q_0$. Set $\eta=\eta(\langle \rangle)\in \FS$, then we can inductively find, using $k$ and $k'$, elements $k^n,k^{n+1}$ of $\dom(f^\eta)$ such that $x_n=f^\eta(k^n)$, $x_{n+1}=f^\eta(k^{n+1})$. Thus we have that if $k^n=\langle q_u^n:u\in A^\eta\rangle$ and $k^{n+1}=\langle q^{n+1}_u:u\in A^\eta\rangle$ then $q^n_u\leq q^{n+1}_u$, by repeated applications of Lemma \ref{Lemma:RealInfExtensive}. So let $\psi_u$ be the embedding witnessing $q^n_u\leq q^{n+1}_u$ for $u\in A^\eta$, granted by this induction.
Then since $\OA$ is extendible, we see that $$\varphi_n=\varphi_{n+1}\circ \psi_{\eta,\eta}^{k^n,k^{n+1}}=\varphi_{n+1}\circ \mu_n$$ with $\mu_n$ as from Remark \ref{Rk:Mu}. Therefore, since $\CC$ has nice limits we have $x\leq y$.
%Then if embedding that we used in the final line of the proof of Theorem \ref{Thm:Non-Limit} to extend the previous embedding. So for every $n\in \omega$ we have embeddings $\varphi_n:x_n\rightarrow y$ with $\varphi_{n+1}$ extending $\varphi_n$. Therefore, since $\CC$ has nice limits we have $x\leq y$.
%since there are  that still embed into the larger object, hence since $\CC$ is extendible we can go back inductively, and we see that the embedding given in the final line of the proof of Theorem \ref{Thm:Non-Limit} is extended to the next step. Ie for every $n\in \omega$ we have embeddings $h_n:x_n\rightarrow y$ with $h_{n+1}$ extending $h_n$. Thus since $\CC$ has nice limits we have that $x\leq y$. Similarly $y\leq x$.
\end{proof}

Again we immediately obtain a bqo result, we also mention some further corollaries.
\begin{cor}\label{Cor:LimitBqo}
Let $\OA$ be extendible and have nice limits. Suppose that $\sscatt_\PP^{\clos{\RR}}$ is well-behaved, $\MM$ is bqo and $\AAA$ is bqo. Then $\amr$ is bqo.
\end{cor}
\begin{proof}
Suppose we had a bad $\amr$-array $f$. Let define the bad $\sscatt_\PP^{\clos{\RR}}(\MM\cup \AAA)$-array $g$ by letting $g(X)$ be a decomposition tree for $f(X)$. Then by Theorem \ref{Thm:Limit}, $g$ must be bad. Thus since $\sscatt_\PP^{\clos{\RR}}$ is well-behaved there must be a bad function to $\MM\cup \AAA$, which is a contradiction of Theorem \ref{Thm:U bqo} since $\MM$ and $\AAA$ were bqo.
\end{proof}
\begin{cor}[Laver \cite{LaverFrOTconj}]\label{Cor:sscat}
$\sscat$, the class of $\sigma$-scattered linear orders is bqo.
\end{cor}
\begin{proof}
Set $\CC, \AAA, \MM, \PP$ and $\RR$ as in Example \ref{Ex:ScatLO}. For a limiting sequence $(x_n)_{n\in \omega}$, we consider $x_n\subseteq x_{n+1}$ and define the limit to be the union, in this way $\mu_n$ acts as the identity on elements of $x_n$. We then have that $\amr=\sscat$. We have $\PP=\On\cup \On^*$ and $\RR=\{1\}$ hence $\clos{\RR}=\omega$ so since $\PP$ has injective morphisms and by theorems \ref{Thm:OnWB}, \ref{Thm:U bqo} and \ref{Thm:Kriz} we know $\sscatt_\PP^{\clos{\RR}}$ is well-behaved. By Corollary \ref{Cor:LimitBqo} it just remains to show that $\OA$ is extendible and has nice limits. Since $\RR=\{1\}$ %the former is simple. Clearly $\CC$ is a concrete category as it is a class of partial orders, s
%Since we are using sums as functions, i
it is easily verified that $\OA$ is extendible. Finally if we have embeddings $\varphi_n$ for $n\in \omega$ as in the definition of nice limits, then we let $\varphi$ be the union of the $\varphi_n$, this clearly satisfies the definition of nice limits. The result now follows from Corollary \ref{Cor:LimitBqo}.
\end{proof}
\begin{cor}[Laver \cite{LaverClassOfTrees}]
$\sscatt^{\On}$, the class of $\sigma$-scattered trees is bqo.
\end{cor}
\begin{proof}

Let $\CC$ be the class of all trees, $\AAA=\{\emptyset\}\subseteq \CC$, $\PP=\On$ and $\RR=\{1\}$. We let $\MM=\{c_\alpha:\alpha\in \On\}\cup \{d_\kappa:\kappa\in \card\}$, ordered so that $c_\alpha<c_\beta$ and $d_\alpha<d_\beta$ for $\alpha<\beta$; where we define:
\begin{itemize}
	\item $c_\alpha:\CC^\alpha\rightarrow \CC$ such that $c_\alpha(\langle T_\gamma:\gamma< \alpha \rangle)$ is %a chain of order type $\alpha$ with the tree $T_\gamma$ placed below the $\gamma$th element of $\alpha$ for each $\gamma<\alpha$ (and incomparable to the $\gamma+1$th);
%	the set $\alpha\cup \bigcup_{\gamma<\alpha}T_\gamma$ ordered by letting $a\leq b$ iff $a,b\in \alpha$, $a\leq_\alpha b$ or $a\in \alpha$, $b\in T_\gamma$, $a\leq_\gamma b$ or $a,b\in T_\gamma$ for some $\gamma$ and $a\leq_{T_\gamma}b$
the $\alpha$-tree-sum of the $T_\gamma$ $(\gamma<\alpha)$.
	\item $d_\kappa:\CC^\kappa\rightarrow \CC$ such that $d_\kappa(\langle T_\gamma:\gamma<\kappa\rangle)$ is the disjoint union of the $T_\gamma$.
\end{itemize}
Then $\amrs$ is $\scatt^\On$, the class of scattered trees (as in \cite{LaverClassOfTrees}) by Theorem \ref{Thm:ScatteredRank}. %This can be seen because it is possible to show that if $T\in \CC\setminus \amrs$ then $2^{<\omega}\leq T$. We leave the proof of this as an exercise since it is similar to Theorem \ref{Thm:TreesAreAMR}.
Define the limit of a limiting sequence to be the union so that $\amr=\sscatt^{\On}$. Now proceed similarly to Corollary \ref{Cor:sscat}.
\end{proof}
%\begin{remark}
%Without too much more effort, we could in fact show that $\sscat$ is infinitely bqo or even well-behaved, however we keep things simple since this will be implied by Theorem \ref{Thm:LbarWB}.
%\end{remark}
%The next three sections will consist of applications of Theorem \ref{Thm:Non-Limit} and Theorem \ref{Thm:Limit}.
%\section{Closing classes of linear orders under sums}\label{Section:LO}
The following theorem is our main application of Theorem \ref{Thm:Limit}, which is the crucial step in showing that the large classes of partial orders in Section \ref{Section:PO} are well-behaved.

\begin{thm}\label{Thm:MLPisWB}
Let $Q$ be a quasi-order, $\mathbb{L}$ be a class of linear orders, $\mathbb{P}$ be a bqo class of partial orders and $\OA=\OA_{\mathbb{L},\mathbb{P}}^Q$. If $\sscatt^{\clos{\mathbb{L}}}_{\mathbb{P}}$ is well-behaved, then any bad $\amr$-array admits a witnessing bad $Q$-array.
\end{thm}

\begin{proof}
Let $\OA=\langle \CC,\AAA,\PP,\MM,\RR\rangle$. Suppose we have a bad $\amr$-array $f$ then for each $X\in[\omega]^\omega$, let $g(X)$ be a decomposition tree for $f(X)$. First we note that by Definition \ref{Defn:Tx}, it is clear that the leaves of the tree $g(X)$ are coloured by some element of $\AAA$ used to construct $f(X)$. Thus the colours of the partial order $f(X)$ are colours of colours of leaves of $g(X)$. Now by Theorem \ref{Thm:Limit} we have that $g$ is a bad $\scatt^{\clos{\RR}}_{\PP}(\MM\cup \AAA)$-array.

Since we assumed that $\sscatt^{\clos{\mathbb{L}}}_{\mathbb{P}}$ is well-behaved, there is a witnessing bad $\MM\cup \AAA$-array $g'$. Since the order on $\MM$ is isomorphic to the order on $\mathbb{P}$, we know that $\MM$ is bqo. Hence we can restrict $g'$ as in the proof of Theorem \ref{Thm:U bqo}, to find a witnessing bad $\AAA$-array $f'$. Since the colours from each $g(X)$ that were in $\AAA$ were points of $f(X)$ we now know that $f'$ is a witnessing bad array for $f$. Then since $\AAA=Q'^1$ we can restrict again to be inside the complement of $f'^{-1}(-\infty)$ in order to find a witnessing bad $Q^1$-array. This clearly gives a witnessing bad $Q$-array just passing to colours. Hence any bad $\amr$-array admits a witnessing bad $Q$-array.
\end{proof}

\section{Constructing structured $\mathscr{L}$-trees}\label{Section:StructuredRTrees}
The aim of this section is to show that whenever $\mathbb{L}$ is a class of linear orders and $\mathbb{P}$ is a class of partial orders both of which are well-behaved, then we have $\sscatt^{\clos{\mathbb{L}}}_\mathbb{P}$ is also well-behaved. We can then combine this result with Theorem \ref{Thm:MLPisWB}. The proof will consist of an application of Theorem \ref{Thm:Limit}, but now we will construct structured trees instead of partial orders.
\begin{defn}
We define $\mathbb{E}\subseteq\clos{\mathbb{L}}(\mathbb{P}(\AC{2}))$ such that whenever $\langle E_i:i\in r\rangle \in \mathbb{E}$ and $i\in r$ %we have that
%\begin{itemize}
	%\item if $i<\max(r)$ then 
	there is a unique $e\in E_i$ with $\col(e)=1$. 
%	\item and if $i=\max(r)$ then for every $e\in E_i$ we have $\col(e)=0$.
%\end{itemize}

We define from $\mathbb{E}$ the new concrete category $\mathbb{E}'$ of $\AC{2}$-coloured $\mathbb{P}$-structured $\clos{\mathbb{L}}$-trees. Let $E\in \mathbb{E}$ be such that $E=\langle E_i:i\in r\rangle$, %then let $$Z_i=\{e\in E_i:\col(e)=0\}.$$ For each such $E\in \mathbb{E}$ 
for each $i\in r$ we let 
\begin{itemize}
	\item $E'_i=E_i$ if $i$ is not the maximum element of $r$,
	\item $E'_i=E_i\cup\{\varepsilon\}$ if $i$ is the maximum element of $r$.
%	\item $E'_i=E_i\cup\{\sigma\}$ if $i$ is the minimum element of $r$.
\end{itemize}
Then we define $$E'=\bigsqcup_{i\in r}E'_i$$ and let $\mathbb{E}'=\{E':E\in \mathbb{E}\}$. %We take this disjoint union in such a way to make the $E'$ distinct for distinct $E$.
We order $E'$ as in Figure \ref{Fig:E'} by letting %$\sigma$ be a minimum, 
$\varepsilon>e\in E'$ iff $\col(e)=1$; and for $d,e\in E'$ we have $d<e$ iff $d\neq \varepsilon$ and either
\begin{itemize}
	\item $d,e\in E_i$, $\col(d)=1$, $\col(e)=0$.
	\item $d\in E_i$, $e\in E_j$, $i<j$ and $\col(d)=1$.
\end{itemize}
The colouring of elements of the $E_i$ induces a colouring of $E'$ (with colours in $\AC{2}$), we also set %$\col(\sigma)=1$ and 
$\col(\varepsilon)=0$ when $\varepsilon\in E'$. 
For each $i\in r$ and $d\in E_i$ with $\col(d)=1$ and $d<e$ we define the label $l_d(e)$ to be $d$ if $\col(e)=1$ and $e$ otherwise. This gives us a $\mathbb{P}$-structured $\mathbb{L}$-tree $E'$. We call the chain of elements of $E'$ coloured by $1$ the \emph{central chain} of $E'$.

We define morphisms of $\mathbb{E}'$ to be maps induced by embeddings of $\mathbb{E}$ (we can also find a place for the possibly new maximum $\varepsilon$). Thus morphisms of $\mathbb{E}'$ will also be structured tree embeddings.
\end{defn}
\begin{figure}
%  \vspace{-10pt}
  \centering
  
\begin{tikzpicture}
\node [above] at (0.5,4.5) {$E$};
\draw [fill] (0,4) circle [radius=0.04];
\draw [fill] (0.2,4) circle [radius=0.04];
\draw [fill] (0.4,4) circle [radius=0.04];
\draw [fill] (0.6,4) circle [radius=0.04];
%\draw [fill] (0.8,4) circle [radius=0.04];
\draw [fill] (1,4) circle [radius=0.04];
\draw (0,4) -- (1,4);
\node [left] at (0,4) {$E_0$};
\draw [red, fill] (0.8,4) circle [radius=0.04];

\draw [fill] (0,3) circle [radius=0.04];
\draw [fill] (0.4,2.9) circle [radius=0.04];
\draw [fill] (0.6,3.05) circle [radius=0.04];
\draw [fill] (0.8,3) circle [radius=0.04];
\draw [fill] (1,3.1) circle [radius=0.04];
\draw (0,3) -- (0.2,3);
\draw (0.2,3) -- (0.4,2.9);
\draw (0.4,2.9) -- (0.6,3.05);
\draw (0.6,3.05) -- (0.8,3);
\draw (0.8,3) -- (1,3.1);
\draw [red, fill] (0.2,3) circle [radius=0.04];
\node [left] at (0,3) {$E_1$};

\draw [fill] (0,2) circle [radius=0.04];
\draw [fill] (0.2,2) circle [radius=0.04];
\draw [fill] (0.4,2.1) circle [radius=0.04];
\draw [fill] (0.6,2) circle [radius=0.04];
%\draw [fill] (0.8,2) circle [radius=0.04];
\draw [fill] (1,2) circle [radius=0.04];
\draw (0.4,2.1) -- (0.2,2);
\draw (0.4,2.1) -- (0.6,2);
\draw [red, fill] (0.8,2) circle [radius=0.04];
\node [left] at (0,2) {$E_2$};

\draw [fill] (0,1) circle [radius=0.04];
\draw [fill] (0.2,1) circle [radius=0.04];
%\draw [fill] (0.4,1) circle [radius=0.04];
\draw [fill] (0.6,1) circle [radius=0.04];
\draw [fill] (0.8,1) circle [radius=0.04];
\draw [fill] (1,1) circle [radius=0.04];
\draw (0,1) -- (0.4,1);
\draw [red, fill] (0.4,1) circle [radius=0.04];
\node [left] at (0,1) {$E_3$};

\draw [fill] (0,0) circle [radius=0.04];
\draw [fill] (0.2,0.1) circle [radius=0.04];
\draw [fill] (0.2,-0.1) circle [radius=0.04];
%\draw [fill] (0.4,0) circle [radius=0.04];
\draw [fill] (0.6,0) circle [radius=0.04];
\draw [fill] (0.8,0) circle [radius=0.04];
\draw [fill] (1,0) circle [radius=0.04];
\draw (0,0) -- (0.2,0.1);
\draw (0,0) -- (0.2,-0.1);
\draw (0.4,0) -- (0.2,0.1);
\draw (0.4,0) -- (0.2,-0.1);
\draw (0.8,0) -- (1,0);
\node [left] at (0,0) {$E_4$};
\draw [red, fill] (0.4,0) circle [radius=0.04];
%\end{tikzpicture}
%
%\hspace{40pt}%%%%%%%%%%%%%%%%%%%%%%%%%%%%%%%%%%%%%%%%%%%%%%%%%%%%%%%%%%%%%%%%%%%%%%%%%%%%
%
%
%
%
%
%
%
%
%
%%%%%%%%%%%%%%%%%%%%%%%%%%%%%%%%%%%%%%%%%%%%%%%%%%%%%%%%%%%5
%\begin{tikzpicture}
\node [above] at (3+0.5,4.5) {$E'$};
\draw [lightgray] (3+0,4) -- (3+1,4);
\draw [fill] (3+0,4) circle [radius=0.04];
\draw [fill] (3+0.2,4) circle [radius=0.04];
\draw [fill] (3+0.4,4) circle [radius=0.04];
\draw [fill] (3+0.6,4) circle [radius=0.04];
%\draw [fill] (0.8,4) circle [radius=0.04];
\draw [fill] (3+1,4) circle [radius=0.04];

\node [left] at (3+0,4) {$E'_0$};

\draw (3+0.8,4.4) -- (3+0,4);
\draw (3+0.8,4.4) -- (3+0.2,4);
\draw (3+0.8,4.4) -- (3+0.4,4);
\draw (3+0.8,4.4) -- (3+0.6,4);
\draw [blue] (3+0.8,4.4) -- (3+0.8,4);
\draw (3+0.8,4.4) -- (3+1,4);
\draw [red, fill] (3+0.8,4.4) circle [radius=0.04];

\draw [blue] (3+0.8,4) -- (3+0.2,3.3);

\draw [lightgray] (3+0,3) -- (3+0.2,3);
\draw [lightgray] (3+0.2,3) -- (3+0.4,2.9);
\draw [lightgray] (3+0.4,2.9) -- (3+0.6,3.05);
\draw [lightgray] (3+0.6,3.05) -- (3+0.8,3);
\draw [lightgray] (3+0.8,3) -- (3+1,3.1);

\draw [fill] (3+0,3) circle [radius=0.04];
\draw [fill] (3+0.4,2.9) circle [radius=0.04];
\draw [fill] (3+0.6,3.05) circle [radius=0.04];
\draw [fill] (3+0.8,3) circle [radius=0.04];
\draw [fill] (3+1,3.1) circle [radius=0.04];

\draw (3+0.2,3.3) -- (3+0,3);
\draw (3+0.2,3.3) -- (3+0.4,2.9);
\draw (3+0.2,3.3) -- (3+0.6,3.05);
\draw (3+0.2,3.3) -- (3+0.8,3);
\draw (3+0.2,3.3) -- (3+1,3.1);

\draw [blue] (3+0.2,3.3) -- (3+0.2,3);
\draw [blue] (3+0.2,3) -- (3+0.8,2.4);

\draw [red, fill] (3+0.2,3.3) circle [radius=0.04];
\node [left] at (3+0,3) {$E'_1$};
\draw [lightgray] (3+0.4,2.1) -- (3+0.2,2);
\draw [lightgray] (3+0.4,2.1) -- (3+0.6,2);
\draw [fill] (3+0,2) circle [radius=0.04];
\draw [fill] (3+0.2,2) circle [radius=0.04];
\draw [fill] (3+0.4,2.1) circle [radius=0.04];
\draw [fill] (3+0.6,2) circle [radius=0.04];
%\draw [fill] (0.8,2) circle [radius=0.04];
\draw [fill] (3+1,2) circle [radius=0.04];

\draw (3+0.8,2.4) -- (3+0,2);
\draw (3+0.8,2.4) -- (3+0.2,2);
\draw (3+0.8,2.4) -- (3+0.4,2.1);
\draw (3+0.8,2.4) -- (3+0.6,2);
\draw [blue] (3+0.8,2.4) -- (3+0.8,2);
\draw (3+0.8,2.4) -- (3+1,2);

\draw [red, fill] (3+0.8,2.4) circle [radius=0.04];

\node [left] at (3+0,2) {$E'_2$};
\draw [lightgray] (3+0,1) -- (3+0.4,1);
\draw [fill] (3+0,1) circle [radius=0.04];
\draw [fill] (3+0.2,1) circle [radius=0.04];
%\draw [fill] (3+0.4,1) circle [radius=0.04];
\draw [fill] (3+0.6,1) circle [radius=0.04];
\draw [fill] (3+0.8,1) circle [radius=0.04];
\draw [fill] (3+1,1) circle [radius=0.04];

\draw [blue] (3+0.8,2) -- (3+0.4,1.4);

\draw (3+0.4,1.4) -- (3+0,1);
\draw (3+0.4,1.4) -- (3+0.2,1);
\draw [blue] (3+0.4,1.4) -- (3+0.4,1);
\draw (3+0.4,1.4) -- (3+0.6,1);
\draw (3+0.4,1.4) -- (3+0.8,1);
\draw (3+0.4,1.4) -- (3+1,1);

\draw [red, fill] (3+0.4,1.4) circle [radius=0.04];
\node [left] at (3+0,1) {$E'_3$};
\draw [lightgray] (3+0,0) -- (3+0.2,0.1);
\draw [lightgray] (3+0,0) -- (3+0.2,-0.1);
\draw [lightgray] (3+0.4,0) -- (3+0.2,0.1);
\draw [lightgray] (3+0.4,0) -- (3+0.2,-0.1);
\draw [lightgray] (3+0.8,0) -- (3+1,0);
\draw [fill] (3+0,0) circle [radius=0.04];
\draw [fill] (3+0.2,0.1) circle [radius=0.04];
\draw [fill] (3+0.2,-0.1) circle [radius=0.04];

\node [below] at (3+0.4,0) {$\varepsilon$};
\draw [fill] (3+0.6,0) circle [radius=0.04];
\draw [fill] (3+0.8,0) circle [radius=0.04];
\draw [fill] (3+1,0) circle [radius=0.04];

\node [left] at (3+0,0) {$E'_4$};

\draw [blue] (3+0.4,1) -- (3+0.4,0.4);

\draw (3+0.4,0.4) -- (3+0,0);
\draw (3+0.4,0.4) -- (3+0.2,0.1);
\draw (3+0.4,0.4) -- (3+0.2,-0.1);
\draw [blue] (3+0.4,0.4) -- (3+0.4,0);
\draw (3+0.4,0.4) -- (3+0.6,0);
\draw (3+0.4,0.4) -- (3+0.8,0);
\draw (3+0.4,0.4) -- (3+1,0);

\draw [blue, fill] (3+0.4,0) circle [radius=0.04];
\draw [red, fill] (3+0.4,0.4) circle [radius=0.04];
\end{tikzpicture}

\caption{The order on $E'$, for given $E\in \mathbb{E}$.}\label{Fig:E'}
\end{figure}
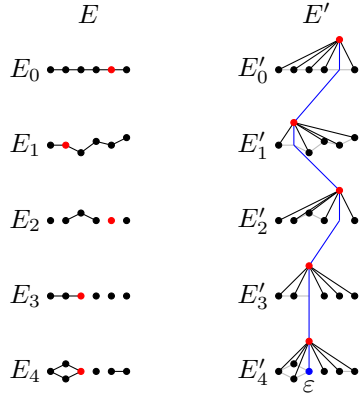
We will now define the parameters of our construction. Throughout this section, we let:
\begin{enumerate}
	\item $Q$ be an arbitrary quasi-order;
	\item $Q'$ be $Q$ with an added minimal element $-\infty$;
	\item $\mathbb{P}$ be an arbitrary concrete category;
	\item $\mathbb{L}$ be an arbitrary class of linear orders;
	\item $\CC=\sscatt^{\clos{\mathbb{L}}}_\mathbb{P}(Q')$;
	\item $\AAA=Q^1\cup \{\emptyset\}\subseteq \CC$;
	\item $\PP=\mathbb{E}'$;
	\item $\RR=\{1\}$.
\end{enumerate}
It remains to define $\MM$, now we will define its functions.
\begin{defn}
For each $E=\langle E_i:i\in r\rangle \in \mathbb{E}$ we define the function $S_E:\dom(S_E)\rightarrow \CC$ as follows. %by letting $\dom(S_E)\subseteq\CC^{E'}$ be so that for $k\in \dom(S_E)$ and $e\in E'$ we have $\col(e)=1$ implies that the argument of $k$ in position $e$ must be in $\AAA$, i.e. 
First let $\arity(S_E)=E'$ and $\barity{S_E}=\{e\in E'\mid \col(e)=1\}$ be the central chain of $E'$. % If $\col_{E'}(e)=0$ then we allow the argument $e$ of $k$ to be any element of $\CC$.
Define the function $\tau:\dom(S_E)\rightarrow \CC^{E'}$ so that $\tau(\langle x_i:i\in E'\rangle)=\langle x'_i:i\in E'\rangle$ where $x'_i$ is a single point coloured by $-\infty$ if $x_i=\emptyset$ and $i$ is on the central chain of $E'$; and $x'_i=x_i$ otherwise.
Now we let $S_E=\sum_{E'}\circ \tau$. % be the $E'$-sum as defined in Definition \ref{Defn:Sums}, except that when we have an argument equal to $\emptyset$ along the central chain, instead of replacing the point by $\emptyset$, we replace the point by the singleton partial order coloured by $-\infty$. 
We inherit labels and colours in this sum. %maybe give a real defn?
We define $S_E\leq_\MM S_F$ iff $E\leq_{\mathbb{E}} F$.
\end{defn}
Throughout this section we will also let $\MM=\{S_E:E\in \mathbb{E}\}$.
%\begin{defn}
%Given a composition sequence $\eta=\langle \langle S_{E_i},s_i\rangle: i\in n\rangle$, we want to define $f^{\eta}$ as the composition of the $S_{E_i}$ as would be expected. So we let $Z^\eta_i$ be the structure $E'_i\setminus s_i$ whenever $i<n-1$ and $Z^\eta_{n-1}=E'_{n-1}$. Now let $E^\eta=\bigsqcup_{i<n}Z^\eta_i$, and $d,e\in E^\eta$. We let $d<e$ iff for $d\in Z^\eta_i$ and $e\in Z^\eta_j$ either
%\begin{itemize}
%	\item $i=j$ and $d<_{Z^\eta_i}e$, or
%	\item $i<j$ and $d<_{E_i}s_i$.
%\end{itemize}
%So for $k\in \dom(f^\eta)$ and $e\in E^{\eta}$ we let $i_e$ be such that $e\in Z^\eta_{i_e}$ and we let $T_e$ be the argument of $k_{i_e}$ in position $e$. We the define $f^\eta(k)$ be the sum over $E^\eta$ of each of the $T_e$.
%\end{defn}
\begin{remark}\label{Rk:TreeSubsetLimits}
Clearly the minimal element of $\CC$ is $\emptyset$. Therefore %composing decomposition functions in a limiting sequence will %either
%just add more branches onto a tree.
elements of limiting sequences are constructed by replacing $\emptyset$ arguments with larger $\mathscr{L}$-trees. So
 %, or will change a colour from $-\infty$ to some element of $Q$. However if the latter occurs in a limiting sequence, the limit we construct would be the same as if we just applied the function to an element of $\AAA$ instead of composing with another function. 
%Hence %we can assume without loss of generality that 
limiting sequences will be sequences of coloured and structured $\mathscr{L}$-trees, whose underlying sets are each $\up$-closed subsets of the next (as in Figure \ref{Fig:ScatTrees}). Thus it is possible to define limits as unions of such sequences.% (If the colour of an element were to change from $-\infty$ to an element of $Q$, we just take the element of $Q$ to be the colour of this point in the union.)
\end{remark}
%\begin{remark}\label{Rk:GoAwayEmptyset}
%Notice that when constructing unions of a limiting sequence, if we applied a function to $\emptyset$ at some stage $n$ and this argument will eventually be replaced by a larger tree (say stage $n+m$), then without loss of generality we can always assume that this tree was added at the first stage possible (stage $n+1$), since doing so makes no difference to the union.
%\end{remark}

We now have two goals. Firstly we will show that $\amr=\sscatt^{\clos{\mathbb{L}}}_\mathbb{P}(Q')$, and secondly we will verify the conditions of Theorem \ref{Thm:Limit}. Using this theorem we will then see that $\sscatt^{\clos{\mathbb{L}}}_\mathbb{P}$ is well-behaved whenever $\mathbb{L}$ and $\mathbb{P}$ are well-behaved.

\subsection{The constructed trees are $\sigma$-scattered}
Now we aim to show that $\amr=\sscatt^{\clos{\mathbb{L}}}_\mathbb{P}(Q')$. 

%For $S,T\in \sscatt^{\clos{\mathbb{L}}}_\mathbb{P}(Q')$ $S$ embeds into the underlying tree of $T$ (with no colours or structure). We write $2^{<\omega}\leq T$ to mean $2^{<\omega}\leq' T$.
\begin{defn}
Let $T\in \sscatt^{\clos{\mathbb{L}}}_{\mathbb{P}}(Q')$, $\xi$ be a chain of $T$ and $t\in \xi$. We define $l_t(\xi)$ to be $\emptyset$ whenever $t$ is the maximal element of $\xi$, and if there is some $t'\in \xi$ with $t'>t$ then we let $l_t(\xi)=\{l_t(t')\}$. This is well-defined by the definition of $l_t$.
\end{defn}

\begin{lemma}\label{Lemma:CombReduce}
Suppose that $T\in \sscatt^{\clos{\mathbb{L}}}_{\mathbb{P}}(Q')\setminus \amrs$, and $\xi$ is an $\up$-closed chain of $T$ that contains no leaves. Then either $\bigcap_{j\in\xi}\down j\notin \amrs$ or there is some $t\in \xi$ and $p\in \range(l_t)\setminus l_t(\xi)$ such that ${}^p\downs t\notin \amrs$.
\end{lemma}
\begin{proof}
Let $T$ and $\xi\subseteq T$ be as described. Suppose $\bigcap_{j\in\xi}\down j\in \amrs$ and there are no such $t$ and $p$. Pick a maximal chain $\zeta$ of $\bigcap_{j\in\xi}\down j$, and let $\xi'=\zeta\cup \xi$. Then for each $i\in \ot(\xi')$ let $\xi'_i$ be the $i$th element of $\xi'$ and we define $E_i$ as an $\AC{2}$-coloured copy of $\range(l_{\xi'_i})\in \mathbb{P}$, that we colour as follows. When $\ot(\xi')$ has a maximum element $m$, we let every colour of $E_m$ be $0$. For non-maximal $i\in \ot(\xi')$ we let $u\in E_i$ be coloured with $1$ iff there is some $j\in \ot(\xi')$ with $i<j$ and $l_{\xi'_i}(\xi'_j)=u$. We set $E=\langle E_i:i\in \ot(\xi')\rangle$, then we have that $E\in \mathbb{E}$.

Now for $a\in E'$ we define $K_a\in \sscatt^{\clos{\mathbb{L}}}_\mathbb{P}(Q')$ as follows. If $a$ is the $i$th point on the central chain of $E'$ and $\col_T(\xi'_i)\neq -\infty$ we let $K_a$ be a single point coloured by $\col_T(\xi'_i)$, if $\col_T(\xi'_i)=-\infty$ then we let $K_a=\emptyset$, hence $K_a\in \AAA$. Otherwise $a$ is not on the central chain of $E'$ and we can let $i$ be largest such that the $i$th point $w$ of the central chain of $E'$ is $<a$ (there was such a largest element, by definition of the order on $E'$). Now let $u\in \range(l_w)$ be such that $l_w(a)=u$. Then we set $$K_a={}^u\downs \xi'_i\subseteq T.$$
We then have by construction that $$S_E(\langle K_a:a\in E'\rangle)=T.$$
%since after applying the function $S_E$, that the central chain of $E'$ will have precisely the same colours and labels as the maximal chain $\xi'$, and each of the trees that we add underneath this chain corresponded exactly to the same sub-tree of $T$ below the same point of $\xi$ with the same label.

But we know that each $K_a$ was in $\amrs$ by our original assumption. Hence since $\amrs$ was closed under applying any function in $\MM$ it must be that $T\in \amrs$ which is a contradiction.
\end{proof}

\begin{lemma}\label{Lemma:NonScatteredTreeSplitting}
Suppose that $T\in \sscatt^{\clos{\mathbb{L}}}_{\mathbb{P}}(Q')\setminus \amrs$, then there is some $u\in T$ and two disjoint $\down$-closed subtrees $S_1,S_2\subseteq \downs u$ such that $S_1,S_2\notin \amrs$.
\end{lemma}
\begin{proof}
%Suppose for contradiction that $T\in \sscatt^{\clos{\mathbb{L}}}_{\mathbb{P}}(Q')\setminus \amrs$ and for any $u\in T$ and any two disjoint $\down$-closed subtrees $S_1,S_2\subseteq \downs u$ of $T$ either $S_1\in \amrs$ or $S_2\in \amrs$.

Let $T\in \sscatt^{\clos{\mathbb{L}}}_{\mathbb{P}}(Q')\setminus \amrs$. Pick a chain $\xi_0$ of $T$ that contains no leaves and is maximal. Then by Lemma \ref{Lemma:CombReduce} there is some $t_0\in \xi_0$ and $p_0\in \range(l_{t_0})\setminus l_{t_0}(\xi_0)$ such that ${}^{p_0}\downs t_0\notin \amrs$. Suppose for induction that for some $\alpha\in\On$ we have defined $t_0<...<t_\alpha\in T$ and $p_0,...,p_\alpha$ such that for any $\gamma<\alpha$ we have $l_{t_\gamma}(t_\alpha)=p_\gamma$ and ${}^{p_\gamma}\downs t_\alpha\notin \amrs$. We now apply Lemma \ref{Lemma:CombReduce} to a maximal chain $\xi_\alpha$ of ${}^{p_\alpha}\downs t_\alpha$ that contains no leaves, whence we find $t_{\alpha+1}\in {}^{p_\alpha}\downs t_\alpha$ and $p_{\alpha+1}$ such that ${}^{p_{\alpha+1}}\downs t_{\alpha+1}\notin \amrs$, in other words $t_{\alpha+1}$ and $p_{\alpha+1}$ satisfy the induction hypothesis.

Now suppose for some limit $\lambda\in \On$ we have defined such $t_\gamma$ and $p_\gamma$ for every $\gamma<\lambda$. Let $\zeta_\lambda=\bigcup_{\gamma<\lambda}\up t_\gamma$, then $\zeta_\lambda$ is an $\up$-closed chain of $T$ that contains no leaves because every element has some $t_\gamma$ larger than it. By Lemma \ref{Lemma:CombReduce} either $\bigcap_{j\in\zeta_\lambda}\down j\notin \amrs$ or there is some $t\in \zeta_\lambda$ and $p\in \range(l_t)\setminus l_t(\zeta_\lambda)$ such that ${}^p\downs t\notin \amrs$. Suppose the latter, then let $\delta$ be least such that $t<t_\delta$. Since $t_\delta\in \zeta_\lambda$ we have that $l_t(t_\delta)\in l_t(\zeta_\lambda)\neq \emptyset$. Now we also had that $p\notin l_t(\zeta_\lambda)$, hence the trees $S_1={}^p\downs t$ and $S_2={}^{p_\delta}\downs t_\delta$ are disjoint. Moreover each is clearly $\down$-closed and we also have $S_1,S_2\notin \amrs$ we now set $u=t$ and we are done. %We have therefore found a contradiction of our earlier assumption.

Finally suppose that $\bigcap_{j\in\zeta_\lambda}\down j\notin \amrs$, then pick a maximal chain $\xi_\lambda$ of $\bigcap_{j\in\xi}\down j$, so by Lemma \ref{Lemma:CombReduce} there is some $t_\lambda\in \xi_\lambda$ and $p_\lambda\in \range(l_{t_\lambda})\setminus l_{t_\lambda}(\xi_\lambda)$ such that ${}^{p_\lambda}\downs t_\lambda\notin \amrs$, and we can continue the induction. Now the induction must stop at some point, otherwise for all limit $\lambda\in \On$ it must be that $\bigcap_{j\in\zeta_\lambda}\down j\neq \emptyset$, and therefore $T$ is a proper class! This contradiction ensures that we will always find $u$, $S_1$ and $S_2$ as required.
\end{proof}

\begin{thm}\label{Thm:TreesAreAMR}
$\amrs=\scatt^{\clos{\mathbb{L}}}_\mathbb{P}(Q')$.
\end{thm}
\begin{proof}
First we will show that $\amrs\subseteq \scatt^{\clos{\mathbb{L}}}_\mathbb{P}(Q')$. It is clear that $\amrs$ consists of $Q'$-coloured, $\mathbb{P}$-structured $\clos{\mathbb{L}}$-trees, so we prove by induction on the rank of $x\in \amrs$ that $2^{<\omega}\not \leq x$. Clearly if $x\in \AAA$ then $2^{<\omega}\not \leq x$, so we have the base case.

So let $x\in \CC_{\alpha}$ for some $\alpha\in \On$. Then for some $E\in \mathbb{E}$ we have that $x=S_E(\langle x_i:i\in E'\rangle)$ with $x_i\in \CC_{<\alpha}$ for each $i\in E'$. Hence by the induction hypothesis, $x$ is a $\zeta$-tree-sum of the scattered trees $x_i$, $(i\in E')$ of lower rank for some chain $\zeta\subseteq x$.

Suppose $2^{<\omega}\leq x$ and let $\psi$ be a witnessing embedding. Then either $\range(\psi)$ is contained entirely inside the chain $\zeta$ (which is a contradiction) or there is some $z\in 2^{<\omega}$ with $\psi(z)$ inside $x_i$ for some $i\in E'$. But each $x_i$ is $\down$-closed, hence $\psi(\down z)\subseteq x_i$. But for any point $t\in 2^{<\omega}$ we have that $\down t$ is a copy of $2^{<\omega}$, and therefore $2^{<\omega}\leq x_i$ which contradicts that $x_i$ was scattered. Therefore $2^{<\omega}\not \leq x$, and hence we have $\amrs\subseteq \scatt^{\clos{\mathbb{L}}}_\mathbb{P}(Q')$.

For the other direction we want to show $\amrs\supseteq \scatt^{\clos{\mathbb{L}}}_\mathbb{P}(Q')$. So suppose that $T\in \scatt^{\clos{\mathbb{L}}}_\mathbb{P}(Q')\setminus\amrs$, we will show that $2^{<\omega}\leq T$. First we set $T^{\langle \rangle}=T$. Now suppose that we have already defined $T^s\notin \amrs$ for some $s\in 2^{<\omega}$. We apply Lemma \ref{Lemma:NonScatteredTreeSplitting} to $T^s$ to obtain corresponding $u_s$, and $S^s_1, S^s_2$ $\notin \amrs$. Now let $T^{s\con \langle 0\rangle}=S^s_1$ and $T^{s\con \langle 1 \rangle }=S^s_2$.

We now define the embedding $\varphi:2^{<\omega}\rightarrow T$ by letting $\varphi(s)=u_s$. We claim that this is a tree embedding. For any $s\in 2^{<\omega}$ we have that the $\down$-closed trees $T^{s\con \langle 0\rangle}$ and $T^{s\con \langle 1\rangle}$ are disjoint, hence $u_{s\con \langle 0\rangle}$ and $u_{s\con \langle 1\rangle}$ are incomparable elements of $T$. It remains to check that for $i\in \{0,1\}$ we have $u_s< u_{s\con \langle i\rangle}$, but this is clear since we had that $$u_{s\con \langle i\rangle}\in T^{s\con \langle i\rangle}=S^s_i\subseteq \downs u_{s}.$$
So indeed $\varphi$ is an embedding and $2^{<\omega}\leq T$. Therefore $T\notin \scatt^{\clos{\mathbb{L}}}_\mathbb{P}(Q')$, and we conclude that $\amrs\supseteq \scatt^{\clos{\mathbb{L}}}_\mathbb{P}(Q')$. So we have shown both directions, which completes the proof.
\end{proof}

\subsection{The class $\sscatt^{\clos{\mathbb{L}}}_\mathbb{P}$ of structured trees is well-behaved}
\begin{lemma}\label{Thm:TreesAritiesWellBehaved}
If $\mathbb{L}$ and $\mathbb{P}$ are well-behaved, then $\PP$ is well-behaved and $\MM$ is bqo.
\end{lemma}
\begin{proof}
Suppose $\mathbb{L}$ and $\mathbb{P}$ are well-behaved. We have that $\PP=\mathbb{E}'$ and $\mathbb{E}\subseteq \clos{\mathbb{L}}(\mathbb{P}(\AC{2}))$. Consider the function $$\tau:\mathbb{E}'(Q)\rightarrow \clos{\mathbb{L}}(\mathbb{P}(\AC{2}\times Q))\times Q'$$ 
defined as follows. First if $\hat{E}'\in \mathbb{E}'(Q)$, (so $\hat{E}'$ is an $\AC{2}\times Q$-coloured copy of $E'\in \mathbb{E}'$), then let $\hat{E}$ be
%
%$\langle \hat{E},\sigma\rangle$ where $\hat{E}$ is 
a coloured copy of $E=\langle E_i:i\in r\rangle\in \mathbb{E}$ such that $\col_{\hat{E}_i}(z)=\col_{\hat{E'}}(z)$ for every $i\in r$ and $z\in E_i$.
If we added $\varepsilon$ to $E'$ then we let $\sigma=\col(\varepsilon)$, otherwise we let $\sigma=-\infty$. Then we define $\tau(\hat{E'})=\langle \hat{E},\sigma\rangle$.
Now if $\tau(E')\leq \tau(F')$ then there is an embedding of $\mathbb{E}$ that increases the colours that are in $Q$. Hence in this case $E'\leq F'$.

So suppose there were a bad $\mathbb{E}'$-array $f$, hence $\tau\circ f$ is a bad $\clos{\mathbb{L}}(\mathbb{P}(\AC{2}\times Q))\times Q'$-array and we notice that $\tau\circ f(X)$ has all of its colours in $Q$ from $f(X)$, and all its colours in $Q'$ either from $f(X)$ or $-\infty$. By Theorem \ref{Thm:times bqo} we either get either a witnessing bad $Q'$-array or $\clos{\mathbb{L}}(\mathbb{P}(\AC{2}\times Q))$-array. In the first case applying the Galvin and Prikry Theorem \ref{Thm:GalvinPrikry}, restricting to be inside the complement of $f^{-1}(-\infty)$ would give a bad $Q$-array which would be witnessing for $f$. 

So we have a witnessing bad $\clos{\mathbb{L}}(\mathbb{P}(\AC{2}\times Q))$-array. Since $\mathbb{L}$ is well-behaved, we have by Theorem \ref{Thm:LbarWB} that $\clos{\mathbb{L}}$ is well-behaved, and hence since $\mathbb{P}$ is also well-behaved, we have a witnessing bad $\AC{2}\times Q$-array $f'$. But each of the colours in $Q$ of $f'(X)$ were from $f(X)$, hence using Theorem \ref{Thm:times bqo} we find a witnessing bad $Q$-array for $f$. Therefore $\PP=\mathbb{E}'$ is well-behaved. Clearly this means that $\mathbb{E}$ is bqo, and hence also $\MM$ is bqo.
\end{proof}
We now verify the conditions of Theorem \ref{Thm:Limit}.%  note that that $\CC=\mathbb{T}^\mathbb{L}_\mathbb{P}(Q')$ is a concrete category.
%\begin{lemma}
%$\MM$ is $\RR$-iterable.
%\end{lemma}
%\begin{proof}
%Trivial since $\RR=\{1\}$.
%\end{proof}
\begin{lemma}\label{Lemma:TreesEveryfExtensive}
Every $f\in \MM$ is extensive.
\end{lemma}
\begin{proof}
Let $S_E\in \MM$ and consider $x=S_E(\langle T_u:u\in E'\rangle)$. Either $u\in E'$ is on the central chain of $E'$ and $T_u$ is the minimal element $\emptyset$; or a copy of $T_u$ appears precisely in $x$.
%
%If we apply $S_E$ to some $k$ with an argument $T$, then either we add $T$ below the chain central chain of $E'$ or $T\in \AAA$ and we replace a point of $E'$ by the point $T$. In either case it is clear that $T\leq S_E(k)$.
\end{proof}
\begin{thm}\label{Thm:TreeInfExtensive}
$\OA$ is extendible.
\end{thm}
\begin{proof}
Since $\RR=\{1\}$ it is clear that $\MM$ is $\RR$-iterable. By Lemma \ref{Lemma:TreesEveryfExtensive} we know that every $f\in \MM$ is extensive. %Our life is made easy since $\RR=\{1\}$. 
So suppose as in Definition \ref{Defn:InfExtensive} that we are given composition sequences $$\eta=\langle \langle S_E,s\rangle\rangle \mbox{ and }\nu=\langle \langle S_F,s'\rangle\rangle$$ with $\eta\leq \nu$ and $k=\langle q_u:u\in A^\eta\rangle\in \dom(S_E)$, $k'=\langle q'_u:u\in A^\nu\rangle\in \dom(S_F)$ such that $q_u\leq q'_{\varphi_{\eta,\nu}(u)}$ for each $u\in A^\eta$ with $\varphi_u$ a witnessing embedding. 
So we have that $f^\eta(k)=S_E(k)$ and $f^\nu(k)=S_F(k')$. We also have that $\varphi_{\eta,\nu}$ is an embedding from $E'=A^\eta$ to $F'=A^\nu$. We define $\varphi^{k,k'}_{\eta,\nu}:f^\eta(k)\rightarrow f^\nu(k')$ as follows.

If $u$ is on the central chain of $E'$, then $u\in \barity{S_E}$, so that $q_u\in \AAA$. We see that $\varphi_{\eta,\nu}$ sends elements of the central chain of $E'$ to the central chain of $F'$ since it was induced from an embedding of $\mathbb{E}$, thus also $q'_{\varphi_{\eta,\nu}(u)}\in \AAA$. Let $a$ be the point of $f^\eta(k)$ that corresponds to $u$, i.e. either it is the single point of $q_u$ in the sum or $q_u=\emptyset$ and $a$ is a single point coloured by $-\infty$. In this case we let $\varphi^{k,k'}_{\eta,\nu}(a)$ be the point $b\in f^\nu(k')$ that corresponds to $\varphi_{\eta,\nu}(u)$. Notice that the embedding $\varphi_u$ gives us that $\col(a)\leq \col(b)$.

Suppose $u$ is not on the central chain of $E'$ and that $a\in q_u\subseteq f^\eta(k)$, then we let $$\varphi^{k,k'}_{\eta,\nu}(a)=\varphi_u(a)\in q_{\varphi_{\eta,\nu}(u)}\subseteq f^\nu(k').$$
Then clearly $\varphi^{k,k'}_{\eta,\nu}$ is a structured tree embedding, since $\varphi_{\eta,\nu}$, and each of the $\varphi_u$ ($u\in A^\eta$) are embeddings. Therefore $\OA$ is infinitely extensive.
%Then we can use $\varphi$ to embed $E'$ into $F'$ since these were the corresponding arities, then since $\varphi$ witnesses $k\leq_{\PP(\CC)} k'$, we can extend $\varphi$ so that each corresponding subtree of $S_E(k)$ that was an argument of the sum, is mapped via the embedding given by $k\leq_{\PP(\CC)} k'$ into a corresponding sub-tree in $S_F(k')$, giving us an embedding $\varphi'$. Thus $S_E(k)\leq S_F(k')$, and $\OA$ is infinitely extensive.

Now suppose that $\hat{k}=\langle \hat{q}_u:u\in A^\eta\rangle$ is such that $q_u\leq \hat{q}_u$ for all $u\in A^\eta$, with some witnessing embedding $\mu_u$. Suppose also that $\hat{q}_u\leq q'_{\varphi(\eta,\nu)}$ for each $u\in A^\eta$, with $\psi_u$ a witnessing embedding, and that $\varphi_u=\psi_u\circ \mu_u$. Then for $a\in f^\eta(k)$, with $a\in q_u$, we have $$\varphi^{k,k'}_{\eta,\nu}(a)=\varphi_u(a)=\psi_u\circ\mu_u(a)\in q'_{\varphi_{\eta,\nu}(u)}$$
and $\mu^{k,k'}_{\eta,\nu}(a)=\mu_u(a)\in \hat{q}_u$, so that $$\psi^{\hat{k},k'}_{\eta,\nu}\circ \mu^{k,\hat{k}}_{\eta,\eta}(a)=\psi_u\circ \mu_u(a)\in q'_{\varphi_{\eta,\nu}(u)}.$$
It remains to check when $q_u=\emptyset$ and $a$ is the point corresponding to $u$. In this case $\varphi^{k,k'}_{\eta,\nu}(a)$ is the point $b\in f^\nu(k')$ corresponding to $\varphi_{\eta,\nu}(u)$. Now $\mu^{k,\hat{k}}_{\eta,\eta}(a)$ is the point of $f^\eta(\hat{k})$ corresponding to $u$, so that $\psi^{\hat{k},k'}_{\eta,\nu}\circ \mu^{k,\hat{k}}_{\eta,\eta}(a)=b$. Thus we have verified the conditions of Definition \ref{Defn:Extendible}, and therefore $\OA$ is extendible.
%
%Suppose now that $\hat{k}$ has the same structure as $k$ but with larger arguments, i.e. each subtree of $S_E(\hat{k})$ that was an argument of this sum is larger than the same argument in $S_E(k)$. Suppose also that $\varphi$ witnesses $\hat{k}\leq_{\PP(\CC)} k'$, and each of the embeddings from arguments of $\hat{k}$ into an arguments of $k'$ given by this embedding extends the corresponding embedding given by $k\leq_{\PP(\CC)} k'$. 
%
%So we extend $\varphi$ in the same way as before so that each corresponding subtree of $S_E(\hat{k})$ that was an argument of the sum, is mapped via the embedding given by $\hat{k}\leq_{\PP(\CC)} k'$ into a corresponding subtree in $S_F(k')$, this gives an embedding $\hat{\varphi}$. It is clear then that the embedding induced from $\clos{\varphi}$ of the sum structure is the same as for $\varphi'$, and that each of the embeddings induced between sub-trees found as arguments of the sum are extended. Hence the embedding $\hat{\varphi}$ extends the embedding $\varphi'$ over the natural embedding given by $S_E(k)\leq S_E(\hat{k})$. Thus $\OA$ is extendible.
%It is also clear here that $\CC$ is extendible since when defining a larger embedding in the same way, we could just use the extended embeddings given to us, this would give an extension of our embedding.
\end{proof}

%Now we aim to show that $\CC$ is limit extensive.

\begin{thm}
$\OA$ has nice limits.
\end{thm}
%\begin{proof}
%We notice that the construction of Remark \ref{Rk:LimitComposition} consists of writing a $\sigma$-scattered tree as a sum over a maximal chain of smaller trees, then for every one of these smaller trees doing the same. Since we chose a maximal branch below every point and with every label, we have that for choice of these chains, the union of all of them will give the original tree. This is because every scattered tree of the original construction 
%\end{proof}
%\begin{proof}
%Just take the embedding from $x$ to $y$ as the union of all of the $h_n$.
%\end{proof}
\begin{proof}
First it is easily verified that $\OA$ has limits. Let $x,y\in \amr$ be such that $x$ is the limit of $(x_n)_{n\in \omega}$ and for any $n\in \omega$ we have embeddings $\varphi_n:x_n\rightarrow y$ such that $\varphi_n=\varphi_{n+1}\circ \mu_n$. We want to show that $x\leq y$.

We have by Remark \ref{Rk:Mu} that $\mu_n$ was the embedding given by the infinite extensiveness of $\MM$, using the construction of each $x_n$. By the same argument as Remark \ref{Rk:TreeSubsetLimits}; the underlying set of $x_n$ is a $\up$-closed subset of the underlying set of $x_{n+1}$ where we consider $\mu_n:x_n\rightarrow x_{n+1}$ acting as the identity on $x_n$.
Now we have that $\varphi_n=\varphi_{n+1}\circ \mu_n$ is equivalent to saying that when $a\in x_n$ we have $\varphi_{n+1}(a)=\varphi_n(a)$. Hence it is possible to define $\varphi:x\rightarrow y$ as the union of all of the $\varphi_n$. We claim that $\varphi$ is an embedding.

Let $a,b\in x=\bigcup_{n\in \omega}x_n$, and let $n$ be least such that $a,b\in x_n$. Let $l_a^x$ be the $a$-label of $x$, and $l_a^{x_{n+1}}$ be the $a$-label of $x_{n+1}$. If $b> a$ we define $\theta:\range(l^x_a)\rightarrow \range(l_{\varphi(a)})$ such that if $l^x_a(b)=u$ then $\theta(u)=l_{\varphi(a)}(\varphi(b))$ (these are labels from $x$). In order to show that $x\leq y$ we verify the following properties of $\varphi$ (see Definition \ref{Defn:StructTrees}).
\begin{enumerate}
	%\item $\varphi$ is injective. Indeed, if $x\neq y$ then $\varphi(x)=\varphi_n(x)\neq \varphi_n(y)=\varphi(y)$.
	\item $a\leq b$ iff $\varphi_n(a)\leq \varphi_n(b)$ iff $\varphi(a)\leq \varphi(b)$ (since $\varphi_n$ is an embedding).
	\item $\varphi(a\wedge b)=\varphi(a)\wedge \varphi(b)$. Note that each $x_n$ is $\up$-closed in $x$. So that $\varphi_n(a\wedge b)$ is defined, and hence since $\varphi_n$ is an embedding, we have $$\varphi(a\wedge b)=\varphi_n(a\wedge b)=\varphi_n(a)\wedge \varphi_n(b)=\varphi(a)\wedge \varphi(b).$$
	\item If $a \leq b$ then $\theta$ is an embedding. %that if $m$ was least such that $a\in x_m$ then 
	%By Remark \ref{Rk:GoAwayEmptyset} we can assume without loss of generality 
	By property (\ref{Item:LimitSeq5}) of Definition \ref{Defn:LimitSequence}, we have $\range(l_a^x)=\range(l_a^{x_{n+1}})$.
		%If $\mathbb{P}$ is closed under subsets then so is $\mathbb{E}'$ so that 
		So $\theta$ is an $\PP$-morphism since it is a map induced by the embedding $\varphi_{n+1}$. %hmmm may need to change $\emptyset$ to a point coloured by $-\infty$... should all still work!!!!!!!
	\item $\col(a)\leq \col(\varphi(a))$. By property (\ref{Item:LimitSeq4}) of Definition \ref{Defn:LimitSequence}, we have $$\col_x(a)=\col_{x_n}(a)\leq \col_y(\varphi_n(a))=\col_y(\varphi(a)).$$
\end{enumerate}
Hence $\varphi$ is an embedding and $x\leq y$ as required.
\end{proof}

\begin{thm}\label{Thm:TLPWell-Behaved}
Suppose that $\mathbb{L}$ and $\mathbb{P}$ are well-behaved, then $\sscatt^{\clos{\mathbb{L}}}_{\mathbb{P}}$ is well-behaved.\footnote{This result was obtained independently by Christian Delhomme in as yet unpublished work. The author thanks him for his private communication.}
\end{thm}
\begin{proof}
By Theorem \ref{Thm:TreesAreAMR} we have that $\amr=\sscatt^{\clos{\mathbb{L}}}_{\mathbb{P}}(Q')$. Suppose we have a bad $\sscatt^{\clos{\mathbb{L}}}_{\mathbb{P}}(Q')$-array $f$ then define $g(X)$ to be a decomposition tree for $f(X)$. First we note that by Definition \ref{Defn:Tx}, it is clear that the leaves of the tree $g(X)$ are coloured by some element of $\AAA$ used to construct $f(X)$. Thus the colours of the tree $f(X)$ are colours of colours of elements of $g(X)$. Now by Theorem \ref{Thm:Limit} we have that $g$ is a bad array with range in $\sscatt^{\clos{\RR}}_{\PP}(\MM\cup \AAA)$.

We know that $\clos{\RR}=\omega$ and $\PP$ has injective morphisms. By Lemma \ref{Thm:TreesAritiesWellBehaved} we know that $\PP$ is well-behaved; hence by Theorem \ref{Thm:Kriz}, there is a witnessing bad $\MM\cup \AAA$-array $g'$. Using Lemma \ref{Thm:TreesAritiesWellBehaved} we know that $\MM$ is bqo, hence we can restrict $g'$ using Theorem \ref{Thm:U bqo}, to find a witnessing bad array $f'$ to $\AAA$. Since the colours from each $g(X)$ that were in $\AAA$ were points of $f(X)$ we now know that $f'$ is a witnessing bad array for $f$. Then since $\AAA=Q^1\cup \{\emptyset\}$ we can restrict again to find a witnessing bad array to $Q^1$, which clearly gives a witnessing bad array to $Q$ just passing to colours. Hence any bad array admits a witnessing bad array, i.e. $\sscatt^{\clos{\mathbb{L}}}_{\mathbb{P}}$ is well-behaved.
\end{proof}

\section{Constructing partial orders} \label{Section:PO}
In this section we aim to show that some large classes of partial orders are well-behaved. Theorems \ref{Thm:MLPisWB} and \ref{Thm:TLPWell-Behaved} together, essentially already give us that some classes of partial orders are well-behaved. The challenge now is to characterise the partial orders that we have constructed.

Our method is similar to Thomass\'{e}'s in \cite{Nfree}, but expanded to make use of the two generalisations of section \ref{Section:OpConstruction}. This allows us to generalise the bqo class of partial orders not only to allow embeddings of the $N$ partial order, but also into the transfinite. In order to prove his main result of \cite{Nfree}, Thomass\'{e} used a structured tree theorem: that $\sscatt^\ctbl_{\CH{2}}(\AC{3})$ preserves bqo. We will use what is essentially the same method, but with the more general trees $\sscatt^{\clos{\mathbb{L}}}_{\mathbb{P}}$ for general well-behaved classes of partial orders $\mathbb{P}$ and linear orders $\mathbb{L}$. So that, in particular when $\mathbb{L}=\sscat$ and $\mathbb{P}=\{1,\AC{2},\CH{2}\}$ we obtain a transfinite version of Thomass\'{e}'s result. We will prove the general theorem (Theorem \ref{Thm:MLPisWB2}) in this section, and explore applications of this theorem in Section \ref{Section:Cors}.

\subsection{Intervals and indecomposable partial orders}
We will now borrow some definitions from \cite{Nfree}. We want to define the indecomposable partial orders which will serve as building blocks for larger partial orders, in order to do so we first require the notion of an interval.
\begin{defn}
Suppose that $a,b,c\in x\in \CC$. We say that $a$ \emph{shares the same relationship} to $b$ and $c$, and write $\SSR{a}{b}{c}$ iff for all $R\in \{<,>,\perp\}$ we have $$a R b\mbox{ iff }a R c.$$
\end{defn}

\begin{defn}\label{Defn:Interval}
Let $P$ be a partial order and $I\subseteq P$, then we call $I\neq \emptyset$ an \emph{interval} of $P$ if $\forall x,y\in I$ and $\forall p\in P\setminus I$ we have $\SSR{p}{x}{y}.$
\end{defn}
\begin{defn}\label{Defn:Indecomp}
Let $P$ be a partial order. Then $P$ is called \emph{indecomposable} if every interval of $P$ is either $P$ itself, or a singleton. $P$ is called \emph{decomposable} if it is not indecomposable.
\end{defn}
\begin{prop}\label{Prop:HetaSSR}
Let $\eta$ be a composition sequence of length $r$ and $b_0,b_1,b_2\in H_\eta$, such that for each $i\in \{0,1,2\}$ we have $b_i\in a^\eta_{j_i}$ for $j_i\in r$ then $$j_0<j_1,j_2\longrightarrow\SSR{b_0}{b_1}{b_2}.$$
\end{prop}
\begin{proof}
This is clear by the definition of $H_\eta$.
\end{proof}

\begin{lemma}\label{Lemma:UnionIntersectionIntervals}
Let $\langle I_j:j\in r\rangle$ be a chain of intervals of a partial order $P$ under $\supseteq$. Then $\bigcup_{j\in r}I_j$ and $\bigcap_{j\in r}I_j$ are intervals.
\end{lemma}
\begin{proof}
%Let $\langle I_j:j\in r\rangle$ be a chain of intervals of $P$ under $\supseteq$. 
Let $a\in P\setminus \bigcup_{j\in r}I_j$ and $b,c\in \bigcup_{j\in r}I_j$. Then $b,c\in I_i$ for some $i\in r$, we know that $I_i$ is an interval hence $\SSR{a}{b}{c}$ as required. The case of intersection is similar.
\end{proof}

\begin{defn}
Let $P$ be a partial order and $I\subseteq P$ be an interval. We define the new partial order $P/I$ to be the partial order obtained from $P$ by removing all but one point of $I$. This is well-defined since $I$ is an interval. If $\mathcal{I}$ is a set of disjoint intervals of $P$, then let $P/\mathcal{I}$ be the partial order obtained from $P$ by removing all but one point of each element of $\mathcal{I}$.
\end{defn}
%\begin{prop}
%Let $I$ be an interval of the partial order $P$, then $P=\sum_{p\in P/I}u_p$, where $u_p=\{p\}$ if $p\notin I$ and $u_p=I$ if $p\in I$.
%\end{prop}
%\begin{proof}
%Clear by Definition \ref{Defn:Sums}.
%\end{proof}

\begin{defn}\label{Defn:MaxCompSeq}
Let $\eta=\langle \langle f_i,s_i\rangle: i\in r\rangle$ be a composition sequence, $k\in \dom(f^\eta)$ and $x\in \CC_\alpha$. We call $\langle \eta,k\rangle$ \emph{maximal} for $x$ iff for each $i\in r$ we have $\arity(f_i)$ is indecomposable and there is a maximal chain $\langle I_j:j\in r\rangle$ of intervals of $x$ under $\supseteq$ such that for some 
$$k=\langle q_{i,u}:i\in r, u\in a^\eta_i\rangle \in \dom(f^{\eta})_{<\alpha},$$
we have $x=f^{\eta}(k)$, and if $\eta_j=\langle \langle f_i,s_i\rangle:i\geq j\rangle$ then we have for all $j\in r$, $$I_j= f^{\eta_j}(\langle q_{i,u}:i\geq j,u\in a^\eta_i\rangle).$$

\end{defn}
\begin{lemma}\label{Lemma:r'infimum}
Let $\eta$ be a composition sequence of length $r$, and $k\in \dom(f^\eta)$ be such that $\langle \eta, k\rangle$ is maximal for $x=f^\eta(k)$. Then any $r'\subseteq r$ has an infimum and a supremum in $r$.
\end{lemma}
\begin{proof}
Let $\eta=\langle \langle f_i,s_i\rangle: i\in r\rangle$, $\eta_j=\langle \langle f_i,s_i\rangle:i\geq j\rangle$, $k=\langle q_{i,u}:i\in r, u\in a^\eta_i\rangle$ and $I_j=f^{\eta_j}(\langle q_{i,u}:i\geq j,u\in a^\eta_i\rangle)$. Consider $I=\bigcup_{i\in r'}I_i$, this is an interval by Lemma \ref{Lemma:UnionIntersectionIntervals}. 

For any $j\in r$ we have that $I_j$ is comparable with $I$ under $\subseteq$, because if $j>i'$ for some $i'\in r'$ then $I\supseteq I_{i'}\supseteq I_j$, and if $j<i'$ for all $i'\in r'$ then for all such $i'$ we have $I_j\supseteq I_{i'}$ and this implies $I_j\supseteq I$. Thus since $\eta$ was maximal, we have that $\langle I_i:i\in r\rangle$ is a maximal chain and thus $I=I_{j_0}$ for some $j_0\in r$. Clearly then $j_0$ is a greatest lower bound of $r'$, i.e. $r'$ has an infimum in $r$.

For the supremum, consider $\bigcap_{i\in r'}I_i$, then the argument is similar.
\end{proof}

\subsection{Generalised $\sigma$-scattered partial orders}
We will now begin to characterise the partial orders that we were constructing earlier. For the rest of this section we let:
\begin{enumerate}
	\item $Q$ be an arbitrary quasi-order;
	\item $Q'$ be $Q$ with an added minimal element $-\infty$;
	\item $\mathbb{P}$ be a class of non-empty partial orders that is closed under non-empty subsets;
	\item $\mathbb{L}$ be a class of linear orders that is closed under subsets;
	\item $\OA=\OA^Q_{\mathbb{L},\mathbb{P}}=\langle \CC,\AAA,\PP,\MM,\RR\rangle$ as in Example \ref{Ex:POCONSTRUCTION}.
%	\item $\CC$ be the class all of $Q'$-coloured partial orders;
%	\item $\AAA=Q'^1$;
%	\item $\PP=\mathbb{P}$, some class of partial orders that is closed under taking nonempty subsets;
%	\item $\MM$ be the set of sums over elements of $\mathbb{P}$, as defined in Definition \ref{Defn:Sums}, inheriting colours;
%	\item $\RR=\clos{\mathbb{L}}$, for $\mathbb{L}$ some class of non-empty linear orders closed under non-empty subsets.
\end{enumerate}
\begin{defn}\label{Defn:Qhat}
We let $-2^{<\omega}$ be the partial order obtained by reversing the order on $2^{<\omega}$. We also define the partial order $\Qhat$ as follows. Elements of $\Qhat$ are finite sequences of elements of $2$, for $s,t\in \Qhat$, we define $s< t$ iff there are some sequences $u,s',t'$ such that $s=u\con \langle 0 \rangle \con s'$ and $t=u\con \langle 1 \rangle \con t'$. (See Figure \ref{Fig:Qhat}.)
\end{defn}
\begin{figure}
\centering
%\vspace{-14pt}
\begin{tikzpicture}[thick,scale=0.8]
\draw [->-] (1,4) -- (-1,4);
\draw [-<-] (-1.5,3) -- (-0.5,3);
\draw [-<-] (0.5,3) -- (1.5,3);
\draw [-<-] (-1.5,3) -- (-0.5,3);
\draw [-<-] (-0.5,3) -- (0.5,3);
\draw [-<-] (0.5,3) -- (1.5,3);

\draw [-<-] (-1.75,2) -- (-1.25,2);
\draw [-<-] (-1.25,2) -- (-0.75,2);
\draw [-<-] (-0.75,2) -- (-0.25,2);
\draw [-<-] (-0.25,2) -- (0.25,2);
\draw [-<-] (0.25,2) -- (0.75,2);
\draw [-<-] (0.75,2) -- (1.25,2);
\draw [-<-] (1.25,2) -- (1.75,2);

\draw [-<-] (-0.5,3) -- (1,4);
\draw [-<-] (-1.25,2) -- (-0.5,3);
\draw [-<-] (-0.25,2) -- (0.5,3);
\draw [-<-] (0.75,2) -- (1.5,3);

\draw [fill] (0,5) circle [radius=0.06];
\draw [fill] (1,4) circle [radius=0.06];
\draw [fill] (-1,4) circle [radius=0.06];
\draw [fill] (-1.5,3) circle [radius=0.06];
\draw [fill] (-0.5,3) circle [radius=0.06];
\draw [fill] (1.5,3) circle [radius=0.06];
\draw [fill] (0.5,3) circle [radius=0.06];
\draw [fill] (-1.75,2) circle [radius=0.06];
\draw [fill] (-1.25,2) circle [radius=0.06];
\draw [fill] (-0.75,2) circle [radius=0.06];
\draw [fill] (-0.25,2) circle [radius=0.06];
\draw [fill] (1.75,2) circle [radius=0.06];
\draw [fill] (1.25,2) circle [radius=0.06];
\draw [fill] (0.75,2) circle [radius=0.06];
\draw [fill] (0.25,2) circle [radius=0.06];
%\node [left] at (0,0) {$0$};
%\node [left] at (0,1) {$1$};
%\node [right] at (1,0) {$2$};
%\node [right] at (1,1) {$3$};

%%\node [above] at (0,5) {$\langle \rangle$};
%%\node [above] at (1,4) {$\langle 1\rangle$};
%%\node [above] at (-1,4) {$\langle 0\rangle$};
%%\node [above] at (-1.5,3) {$\langle 0,0\rangle$};
%%\node [above] at (-0.5,3) {$\langle 0,1\rangle$};
%%\node [above] at (0.5,3) {$\langle 1,0\rangle$};
%%\node [above] at (1.5,3) {$\langle 1,1\rangle$};

\node [below] at (0,2) {$\vdots$};
\end{tikzpicture}
%\vspace{-10pt}
\caption{The partial order $\Qhat$.}\label{Fig:Qhat}
\label{Fig:Antichain}
\end{figure}
The following definition of $\scat_\mathbb{P}^\mathbb{L}$ will characterise the orders of $\amrs$ (see Theorem \ref{Thm:PPL=amrs}). We think of these as a generalisation of scattered linear orders $\scat$. We will think of $\sscat_\mathbb{P}^\mathbb{L}$ as a generalisation of $\sigma$-scattered orders $\sscat$. Our aim is to show that $\amrs=\scat_\mathbb{P}^\mathbb{L}$ and $\amr=\sscat_\mathbb{P}^\mathbb{L}$. 
\begin{defn}\label{Defn:PPL}
We define $\scat_\mathbb{P}^\mathbb{L}$ to be the class of non-empty partial orders $X$ %\in \scat_\mathbb{P}^\mathbb{L}$ 
with the following properties.
\begin{enumerate}[(i)]
	\item \label{PPL1}If $Y\leq X$ is indecomposable, then $Y\in \mathbb{P}$.
	\item \label{PPL2}Every chain of intervals of $X$ with respect to $\supseteq$ has order type in $\clos{\mathbb{L}}$.
	\item \label{PPL3}$2^{<\omega}$, $-2^{<\omega}$ and $\Qhat$ do not embed into $X$.
\end{enumerate}
We let $\pscat_\mathbb{P}^\mathbb{L}$ be the class of those non-empty $X$ satisfying (\ref{PPL1}) and (\ref{PPL2}). We call a sequence $(x_n)_{n\in \omega}$ \emph{increasing} if for each $n\in \omega$, we have that $x_n\in \scat_\mathbb{P}^\mathbb{L}$ and $x_{n+1}$ is an $x_n$-sum of some partial orders from $\scat_\mathbb{P}^\mathbb{L}$ (so we consider $x_n\subseteq x_{n+1}$ similarly to Remark \ref{Rk:PoLimits}). We let $\sscat_\mathbb{P}^\mathbb{L}$ be the class containing all of $\scat_\mathbb{P}^\mathbb{L}$ and unions of increasing sequences.
\end{defn}
\begin{remark}\label{Rk:MLPSeqs}
For any limiting sequence $(x_n)_{n\in \omega}$, let $y_n$ be the underlying set of $x_n$ for each $n\in \omega$. Then $(y_n)_{n\in \omega}$ is an increasing sequence. Furthermore, for any union $y$ of an increasing sequence, we could construct a limiting sequence $(x_n)_{n\in \omega}$ with a limit $x$ whose underlying set is $y$ and has any $Q'$-colouring that we desire.
\end{remark}
In order to prove that $\amr=\sscat^\mathbb{L}_\mathbb{P}$, we first require several lemmas. Firstly given an $x\in \amr$, we now want to fine tune the construction of a decomposition tree for $x$ by making more explicit the choice of $\eta$ used to construct a decomposition function. (See Lemma \ref{Lemma:decompfn} and Definition \ref{Defn:Tx}.) We want to only take $\eta$ that are maximal in the sense of Definition \ref{Defn:MaxCompSeq}.

\begin{lemma}\label{Lemma:POsMaxChainsForTx}
Let $x\in \CC_\alpha\cap \pscat_\mathbb{P}^\mathbb{L}(Q)$. Then there is a composition sequence $\eta$ and $k\in \dom(f^\eta)$, such that $\langle \eta,k\rangle$ is maximal for $x$.
%
%Then there is a maximal chain $\langle I_j:j\in r\rangle$ of intervals of $x$, a composition sequence $\eta=\langle \langle f_i,s_i\rangle :i\in r\rangle$ with $\eta_j=\langle \langle f_i,s_i\rangle:i\geq j\rangle$, and some $k=\langle k_i:i\in r\rangle \in \dom(f^{\eta})\cap \RR(\PP(\CC_{<\alpha}))$ such that $x=f^{\eta}(k)$ and for $k^j=\langle k_i:i\geq j\rangle $ we have $I_j= f^{\eta_j}(k^j)$. 
\end{lemma}
\begin{proof}
Let $x\in \CC_\alpha$ so by Proposition \ref{Prop:Reduction}, there is some $\eta$ and $k=\langle q_u:u\in A^\eta\rangle\in \dom(f^\eta)_{<\alpha}$ such that $x=f^{\eta}(k)=\sum_{u\in H_\eta}q_u$. We will define a maximal chain $\langle I'_i:i\in r'\rangle$ of intervals of $x$, a composition sequence $\eta'$ and some $k'\in \dom(f^{\eta'})$, so that $\langle \eta',k'\rangle$ is maximal for $x$.

Let $r$ be the length of $\eta$, then for $j\in r$ define $$I_j=\{d\in x\mid d\in q_u, u\in a^\eta_i, i\geq j\}.$$
By Proposition \ref{Prop:HetaSSR} and since $x$ is an $H_\eta$-sum, we have for all $j\in r$ that $I_j$ is an interval of $x$; and for $i<j$ we have $I_i\supset I_j$. So $\langle I_j:j\in r\rangle$ is a chain of intervals under $\supseteq$. Consider a maximal chain $\langle I'_i:i\in r'\rangle$ of intervals under $\supseteq$ that includes all of the $I_j$. For $i\in r$ we define the set 
$$D_i=\sum_{u\in a^\eta_i}q_u=I_i\setminus\bigcup\{I_{i'}\mid i'<i\}$$
and for $j\in r'$ we define the set 
$$D'_j=I'_j\setminus\bigcup\{I'_{j'}\mid j'<j\}.$$
Then $D'_j\subseteq x=\bigcup_{i\in r}D_i$, and moreover since each of the $I_i$ were equal to some $I'_j$ it must be that each $D_i$ is a union of some of the $D'_j$. %is the union of those $D'_j$ such that $\bigcup_{i'<i} I_{i'}\subset I'_j$. 
In particular, for all $j\in r'$ there is some $i\in r$ such that $D'_j\subseteq D_i$. 
%
%By the definitions it is not too difficult to see that $D_i=\sum_{e\in a_i^\eta}x_e$. 
Therefore, for some $b_j\subseteq a_i^\eta$ and some $q'_u\subseteq q_u$ for $u\in b_j$, we can write $$D'_j=\sum_{u\in b_j}q'_u.$$

For $j\in r'$, if $j=\max(r')$ we put $b'_j=b_j$, otherwise we can set $b'_j=b_j\sqcup\{s'_j\}$. We define the order on $b'_j$ by inheriting from $b_j$, and for $u\in b'_j$ with $u\neq s'_j$ we put
\begin{itemize}
	\item $u< s'_j$ iff there is some $w\in x\setminus \bigcup_{j'\leq j}D'_{j'}$ and $v\in q'_w$ such that $v<w$ and 
	\item $u> s'_j$ iff there is some $w\in x\setminus \bigcup_{j'\leq j}D'_{j'}$ and $v\in q'_w$ such that $v>w$.
\end{itemize}
This order is well-defined since $\langle I'_i:i\in r'\rangle$ was a chain of intervals.

%Now suppose that for some $j\in r$ there is an interval $Z$ of $b'_j$ such that $1<|Z|$ and $s'_j\notin Z$. Then let $c_j$ be the partial order obtained by replacing $Z$ with a single point $e_Z$ (which is well-defined since $Z$ was an interval). Then we can write $D'_j=\sum_{e\in c_j}y_e$ where $y_e=x'_e$ if $e\neq e_Z$ and $y_{e_Z}=\sum_{e\in Z}x'_e$.

Let $\mathcal{Z}=\{Z_\gamma:\gamma\in \kappa\}$ be the set of all non-singleton unions of maximal chains of intervals of $b'_j$ that do not contain $s'_j$; so by Lemma \ref{Lemma:UnionIntersectionIntervals} for each $\gamma\in \kappa$ we have $Z_\gamma$ is an interval. Now let $a_j=b'_j/\mathcal{Z}$, letting $e_\gamma\in a_j$ be the point remaining from $Z_\gamma$. So the only non-singleton intervals of $a_j$ contain $s'_j$. For each $u\in a_j\setminus (\{e_\gamma\mid \gamma<\kappa\}\cup \{s'_j\})$ define $y_u=q'_u$ %if $e\in a_j\setminus \{e_\gamma\mid \gamma<\kappa\}$ 
and for each $\gamma\in \kappa$ define $y_{e_\gamma}=\sum_{u\in Z_\gamma}q'_u$. Thus we can write $$D'_j=\sum_{u\in a_j\setminus \{s'_j\}}y_u.$$

We set $f'_j=\sum_{a_j}$ and $\eta'=\langle \langle f'_j,s'_j\rangle :j\in r'\rangle$. For each $i\in r$, $D_i$ was a union of some $D'_j$, so since $\langle I'_i:i\in r'\rangle$ was a maximal chain of intervals, we can see that $x=\bigcup_{j\in r'}D'_j$. We chose $\eta'$ precisely so that $H_{\eta'}$ consisted of copies of the $a_j$ arranged in the same order as the $D'_j$ were arranged; we also know that $D'_j$ was a sum of the $y_u$ and therefore $x$ is an $H_{\eta'}$ sum of the $y_{u}$. In other words $x=f^{\eta'}(\langle y_u:u\in A^{\eta'}\rangle)$. Since for each $u\in A^\eta$ we have $q'_u\subseteq q_u\in \CC_{<\alpha}$, we know that $q'_u\in \CC_{<\alpha}$ (by applying the same construction to $q'_u$ as to $q_u$, taking smaller sums where necessary we see that $q'_u$ must have rank $<\alpha$). %maybe a lemma of this instead
Let $\eta_j=\langle \langle f'_i,s'_i\rangle:i\geq j\rangle$ then we have $I'_j=f^{\eta_j}(\langle q'_u:u\in a^\eta_i, i\geq j\rangle)$, since the $I'_j$ were just the corresponding portion of the $H_{\eta'}$ sum. 

Now we set $k'=\langle y_u:u\in A^{\eta'}\rangle\in \dom (f^{\eta'})$ and thus in order to show that $\langle \eta',k'\rangle$ is maximal, it remains only to show that $a_i$ is indecomposable for each $i\in r'$. So suppose for contradiction that there is an $i\in r'$ with $a_i$ decomposable, so we can let $Z$ be an interval of $a_i$ with $1<|Z|<|a_i|$. Thus we have from before that $s'_i\in Z$.
%Let $a\in \arity(f_i)$ let $P_a=\col_{k_i}(a)$ if $a\in a^\eta_i$, and if $a\notin a^\eta_i$ then $a=s_i$, in which case set $P_a=\bigcup_{j>i} I_j$.
Now set $$J=\left(\sum_{u\in Z\setminus \{s'_i\}}y_u \right) \cup \bigcup_{j>i} I'_j.$$
 Since $Z$ is an interval of $a_i$ and we are taking the sum, it is simple to verify that $J$ is an interval of $x$. Moreover, $J$ is a strict subset of $I'_i=(\sum_{u\in a_i\setminus \{s'_i\}}y_u)\cup \bigcup_{j>i} I'_j$ which strictly contains $\bigcup_{j>i} I'_j$. But then existence of such a $J$ contradicts that $\langle I'_j:j\in r\rangle$ was a maximal chain of intervals. Thus no such $Z$ exists and each $a_i$ is indecomposable. This completes the proof.
\end{proof}

\begin{remark}\label{Rk:POsMaxEtas}When we constructed decomposition functions (Lemma \ref{Lemma:decompfn}), the only condition on the composition sequence $\eta$ that we used at each stage is that when $\rank(x)=\alpha$ we have $x=f^{\eta}(k)$ for some $k\in \dom(f^\eta)_{<\alpha}$. Hence by Lemma \ref{Lemma:POsMaxChainsForTx}, we can always assume without loss of generality that at every stage of the construction of a decomposition tree for $x$ (Definition \ref{Defn:Tx}) we chose $\eta$ with a corresponding maximal chain of intervals so that $\eta$ was maximal. Doing so makes no difference to the results of Section \ref{Section:WBConstruction}, since the choice of $\eta$ was arbitrary so long as $\eta$ satisfies the condition above. 
%
%Suppose that we fixed an enumeration of $x=\{x_i\mid i\in |x|\}$, and at each stage chose a maximal chain of intervals containing the point of $x$ that was least in the enumeration.
%
%
For the rest of this section, we always assume that any decomposition tree was constructed by choosing such maximal composition sequences.
\end{remark}

\begin{lemma}\label{Lemma:IndecompLabels}
If $x\in \amr\cap \pscat_\mathbb{P}^\mathbb{L}(Q)$ has a decomposition tree $T$, then for every non-leaf $t\in T$, $\range(l_t)$ is indecomposable.
\end{lemma}
\begin{proof}
Let $t\in T_x$ then either $t$ is a leaf or $t=\vec{p}\con \langle i\rangle$ for some $\vec{p}$ and $i$. %By the definition of a decomposition tree, if we have 
If $\eta(\vec{p})=\langle \langle f_i^{\vec{p}},s_i^{\vec{p}}\rangle : i\in r(\vec{p})\rangle$ then $\range(l_t)=\arity(f^{\vec{p}}_i)$. Following Remark \ref{Rk:POsMaxEtas}, we have assumed that $\eta(\vec{p})$ is maximal, and therefore $\arity(f^{\vec{p}}_i)$ is indecomposable, which gives the lemma.
%and suppose $\range(l_t)$ were decomposable. But then without loss of generality by Lemma \ref{Lemma:POsMaxChainsForTx} we had that $t$ was added to $T$ as part of $H(\eta)$ for a composition sequence $\eta$, such that for some $k$ and $x'$, $\langle \eta,k\rangle$ was maximal for $x'$. Thus if $\eta=\langle \langle f_i,s_i\rangle:i\in r\rangle$ then for some $i\in r$ we have $\range(l_t)=\arity(f_i)$. Then by Lemma \ref{Lemma:MaxCSIndecompArities}, we have that $\range(l_t)$ is indecomposable.
%Therefore there is some interval $Z'\subset Z$ with $|Z'|>1$. We had that $t$ was added to $T_x$ as the central chain of some $H(x')\subseteq T_x$, with $x'\subseteq x$. Let $i$ be such that $t$ is the $i$th element of this central chain. without loss of generality by Lemma \ref{Lemma:POsMaxChainsForTx}, there is a maximal chain of intervals $\langle I_j:j\in r\rangle$ of this $x'$ such that $x'=f^{\eta}(k)$ and $I_j=f^{\eta_j}(k^j)$ for each $j\in r$.
%
%For $p\in \range{l_t}$ let $x(t,p)\subseteq x'$ be as from Lemma \ref{Lemma:x(t,p)}. Now set $$J=\sum_{p\in Z'} x(t,p)\subseteq x'$$ it is simple to verify that $J$ is an interval of $x'$, since $Z'$ was an interval, and we are taking the sum of other orders. %lemma for this?
%Moreover, $J$ is a strict subset of the interval $I_i=\sum_{p\in Z}x(t,p)$ since $Z'\subset Z$. We also have since $|Z'|>1$ that $\bigcup_{j>i}I_j$ is a strict subset of $J$. It is simple to verify that $\bigcup_{j>i}I_j$ is also an interval of $x'$. Thus $\langle I_j:j\in r\rangle$ was not maximal, because it could have contained $J$. This is a contradiction.
\end{proof}

\subsection{Pathological decomposition trees}

We now continue with lemmas that mainly address which types of embeddings of $2^{<\omega}$ can appear within decomposition trees, and what affect such embeddings will have on the partial order. First we have a lemma that tells us that we can always find a copy of $2^{<\omega}$ in decomposition trees for elements of $\amri$.

\begin{lemma}\label{Lemma:XnotinAmrsThenDecompTreeNotScat}
Let $x\in \pscat_\mathbb{P}^\mathbb{L}(Q)$ with $x\notin \amrs$, then any decomposition tree for $x$ is not scattered.
\end{lemma}
\begin{proof}
We will prove this by induction on the scattered rank of possible decomposition trees for $x$. Since the only non-empty scattered tree of rank $0$ is a single point, clearly $x$ has no scattered decomposition tree of rank $0$. % Note that decomposition trees are not disjoint unions of subtrees.
Suppose for any $y\in \pscat_\mathbb{P}^\mathbb{L}(Q)\setminus \amrs$ that $y$ has no decomposition tree $T'$ with $\Srank(T')<\alpha$ and that $x\notin \amrs$ has a scattered decomposition tree $T$ with $\Srank(T)=\alpha$. Thus there is a chain $\zeta$ of $T$ such that $T$ is a $\zeta$-tree-sum of the lower ranked trees ${}^p\downs t$ $(t\in \zeta, p\in a^\zeta_t)$.%\footnote{By construction, decomposition trees are not disjoint unions of other trees.}
%Now for $i\in \zeta$:
%\begin{itemize}
%	\item let $f_i=\col(i)$;
%	\item when $i$ is not a maximum, let $j\in \zeta\cap \downs i$ and $s'_i=s_i=l_i(j)$;
%	\item when $i$ is a maximum, let $s'_i=\emptyset$ and let $s_i$ be any element of $\range(l_i)$;
%	\item let $r=\ot(\zeta)$,
%	\item let $k_i=\langle x(i,p):p\in \range(l_i)\setminus s'_i\rangle$.
%\end{itemize}
%This allows us to define $\eta=\langle \langle f_i,s_i\rangle: i\in r\rangle$. 

Then $x=f^{\eta(\zeta)}(k^\zeta)$ by Lemma \ref{Lemma:ChangeOfPos}. For each $t\in \zeta$ and $p\in a^\zeta_t$ we have a decomposition tree $T_{x(t,p)}={}^p\downs t$ for $x(t,p)$, with $\Srank(T_{x(t,p)})<\alpha$. Thus by the induction hypothesis it must be that $x(t,p)\in \amrs$. But then since $x=f^{\eta(\zeta)}(k^\zeta)$ we see that $x$ is $f^{\eta(\zeta)}$ applied to elements of $\amrs$ hence it must be that $x\in \amrs$. This is a contradiction, therefore $x$ has no scattered decomposition tree of rank $\alpha$. This completes the induction and gives the lemma.
%
% there must be some $t\in \zeta$ and $p\in \range(l_t)$ such that $x(t,p)\notin \amrs$. But $x(t,p)$ had a decomposition tree ${}^p\downs t$ which was one of the lower ranked trees. Hence by the induction hypothesis, ${}^p\downs t$ is not scattered. Hence $2^{<\omega}$ embeds into ${}^p\downs t\subseteq T$, which implies that $T$ is not scattered.
\end{proof}

\begin{lemma}\label{Lemma:POinDcompAppearinT}
Let $x\in \amr$, and $T$ be a decomposition tree for $x$. If $y\subseteq x$ is indecomposable and $|y|>1$, then there exists $t\in T$ such that the underlying partial order of $y$ embeds into $\range(l_t)$.
\end{lemma}
\begin{proof}
First we suppose $x\in \amrs$. If $\rank(x)=0$ then $x$ is a single point, hence $|y|\leq |x|=1$ which gives the base case. %If $\rank(x)=1$ then $x=f^{\eta}(k)$ for some $k\in \RR(\PP(\AAA)))$. Thus $x$ is an $H_\eta$ sum of single points, since $y$ is indecomposable then, it must be that some $f_i$, a function in the composition sequence $\eta$, was a sum over some partial order $P$ into which $y$ embeds. Thus $H(x)$, used in the construction of $T$, contained a point $t$ such that $\range(l_t)=P$.
Suppose now that for some $\alpha\in \On$ we have the lemma for all $x_0\in \CC_{<\alpha}$ and that some $x\in\CC_\alpha$ has an indecomposable subset $y\subset x$ with $|y|>1$. Then we know that $x=f^{\eta}(k)=\sum_{u\in H_\eta}q_u$ for some $k=\langle q_u:u\in A^\eta\rangle\in \dom(f^\eta)_{<\alpha}$. %Thus $x=\sum_{u\in A^{\eta}}q_u$ so i
If any of these $q_u$ contains a subset isomorphic to a copy of $y$; then since a decomposition tree for $x$ is a $\zeta$-tree-sum of decomposition trees for the $q_u$ (for some chain $\zeta$), by the induction hypothesis we are done. 

So suppose for every $u\in A^\eta$, $q_u$ does not have a subset isomorphic to $y$. We claim that if $\eta=\langle \langle f_i,s_i\rangle:i\in r\rangle$ then for some $i\in r$, and some partial order $P$ with $y\leq P$, we have $f_i=\sum_P$. If for some $i\in r$ at least two points of $y$ were in $\sum_{u\in a^\eta_i}q_u$ and another point of $y$ was in $\sum_{u\in a^\eta_j}q_{u}$ for $j<i$, then the interval $$y\cap \bigcup_{j'>j}\sum_{u\in a^\eta_i}q_u\subset y$$ shows that $y$ is decomposable. So we can let $i$ be least such that $y\cap a^\eta_i\neq \emptyset$, and we know that there is at most one point $v\in y$ contained inside some $q_u$ for $u\in a^\eta_j$ with $j>i$; in this case we let $\varphi(v)=s_i$.
If two points of $y$ were in some $q_u\subset x$ then there is another point of $y$ not in here, so that the interval $q_u\cap y\subset y$ shows that $y$ is decomposable. Thus all $w\in y\setminus \{v\}$ are inside $\sum_{u\in a^\eta_i}q_u$, with at most one point in each $q_u$. So we let $\varphi(w)=u$ whenever $w\in q_u$. Thus we have defined $\varphi:y\rightarrow \arity(f_i)$ and it is simple to verify that $\varphi$ is an embedding. 
Now without loss of generality we have $\eta=\eta(\langle \rangle)$ so that $\arity(f_i)=\range(l_{\langle i\rangle})$. This completes the proof for $x\in \amrs$.
%
%
%
%
% Thus again $H(x)$, used in the construction of $T$, contained a point $t$ such that $\range(l_t)=P$.

Suppose that $x\in \amri$ is the limit of $(x_n)_{n\in \omega}$. Then we claim that for some $n\in \omega$ there is some $y'\subseteq x_n$ such that the underlying orders of $y'$ and $y$ are isomorphic. Let $n$ be least such that $|y\cap x_n|>1$. If $y$ is not a subset of $x_n$ then for some $m>n$, some points of $x_n$ were replaced by larger partial orders; equating the original point to a point of this partial order as in Remark \ref{Rk:PoLimits}. If at least two points of $y$ were in a partial order that replaced a single point, then these points form a proper interval of $y$ which contradicts that $y$ is indecomposable. So let $y'$ consist of those points of $x_n$ that are either in $y$ or are replaced by a partial order containing a point of $y$ in some $x_m$ $(m>n)$. Thus the underlying orders of $y'$ and $y$ are isomorphic. Hence for some $T_n$ a decomposition tree for $x_n$ we have $t\in T_{n}$ and the underlying partial order of $y'=y$ embeds into $\range(l_t)$. Hence also $t\in T$. This completes the proof.
%
% is a subset of some $x_n$, since otherwise we would have added $y$ throughout multiple nested sums, which is impossible because $y$ is indecomposable. Hence for some $T_n$ a decomposition tree for $x_n$ we have $t\in T_{n}$ and the underlying partial order of $y$ embeds into $\range(l_t)$. Hence also $t\in T$. This completes the proof.
\end{proof}

We next define the $\{\CH{2}, \AC{2}\}$-structured, $\{\sum_{\CH{2}}, \sum_{\AC{2}}\}$-coloured trees $\Btree^+$, $\Btree^-$, $\Qhattree$, $\Qtree$ and $\Atree$ as in Figure \ref{Fig:PthTrees}. These will be \emph{pathological decomposition trees} for the partial orders $2^{<\omega}$, $-2^{<\omega}$, $\Qhat$, $\mathbb{Q}$ and $\AC{\aleph_0}$ respectively (we deal with these in \ref{SubSubSection:PathPOs}). It will turn out that $\Qtree$ and $\Atree$ can never be decomposition trees under the assumption granted by Remark \ref{Rk:POsMaxEtas}.

\begin{figure}
\centering
    \begin{tikzpicture}[scale=0.4]

%B+TREE
\node [above] at (0,5) {$\Btree^+$};
\draw [fill] (0,5) circle [radius=0.06];
\draw (0,5) -- (-2,4);
\draw [fill] (-2,4) circle [radius=0.06];
\draw (0,5) -- (2,4);
\draw [fill] (2,4) circle [radius=0.06];
\draw (2,4) -- (3,3);
\draw [fill] (3,3) circle [radius=0.06];
\draw (2,4) -- (1,3);

\draw [fill] (1,3) circle [radius=0.06];
\draw (1,3) -- (1.5,2);
\draw [fill] (1.5,2) circle [radius=0.06];
\draw (1,3) -- (0.5,2);

\draw [fill] (0.5,2) circle [radius=0.06];
\draw (3,3) -- (3.5,2);
\draw [fill] (3.5,2) circle [radius=0.06];
\draw (3,3) -- (2.5,2);

\draw [fill] (2.5,2) circle [radius=0.06];
\draw (3.5,2) -- (3.25,1);
\draw [fill] (3.25,1) circle [radius=0.06];
\draw (3.5,2) -- (3.75,1);

\draw [fill] (3.75,1) circle [radius=0.06];
\draw (1.5,2) -- (1.25,1);
\draw [fill] (1.25,1) circle [radius=0.06];
\draw (1.5,2) -- (1.75,1);

\draw [fill] (1.75,1) circle [radius=0.06];
\draw (3.25,1) -- (3.125,0);
\draw (3.25,1) -- (3.375,0);
\draw (3.75,1) -- (3.625,0);
\draw (3.75,1) -- (3.875,0);

\draw (1.25,1) -- (1.125,0);
\draw (1.25,1) -- (1.375,0);
\draw (1.75,1) -- (1.625,0);
\draw (1.75,1) -- (1.875,0);
\node at (0,4) {$<$};
\node at (2,3) {$\perp$};
\node at (1,2) {$<$};
\node at (3,2) {$<$};
\node [scale=0.7] at (3.5,1) {$\perp$};
\node [scale=0.7] at (1.5,1) {$\perp$};
\end{tikzpicture}
\hspace{2pt}
\begin{tikzpicture}[scale=0.4]
%B-TREE
\node [above] at (9-0,5) {$\Btree^-$};
\draw [fill] (9-0,5) circle [radius=0.06];
\draw (9-0,5) -- (9--2,4);
\draw [fill] (9--2,4) circle [radius=0.06];
\draw (9-0,5) -- (9-2,4);
\draw [fill] (9-2,4) circle [radius=0.06];
\draw (9-2,4) -- (9-3,3);
\draw [fill] (9-3,3) circle [radius=0.06];
\draw (9-2,4) -- (9-1,3);
\draw [fill] (9-1,3) circle [radius=0.06];
\draw (9-1,3) -- (9-1.5,2);
\draw [fill] (9-1.5,2) circle [radius=0.06];
\draw (9-1,3) -- (9-0.5,2);
\draw [fill] (9-0.5,2) circle [radius=0.06];
\draw (9-3,3) -- (9-3.5,2);
\draw [fill] (9-3.5,2) circle [radius=0.06];
\draw (9-3,3) -- (9-2.5,2);
\draw [fill] (9-2.5,2) circle [radius=0.06];
\draw (9-3.5,2) -- (9-3.25,1);
\draw [fill] (9-3.25,1) circle [radius=0.06];
\draw (9-3.5,2) -- (9-3.75,1);
\draw [fill] (9-3.75,1) circle [radius=0.06];
\draw (9-1.5,2) -- (9-1.25,1);
\draw [fill] (9-1.25,1) circle [radius=0.06];
\draw (9-1.5,2) -- (9-1.75,1);
\draw [fill] (9-1.75,1) circle [radius=0.06];
\draw (9-3.25,1) -- (9-3.125,0);
\draw (9-3.25,1) -- (9-3.375,0);
\draw (9-3.75,1) -- (9-3.625,0);
\draw (9-3.75,1) -- (9-3.875,0);

\draw (9-1.25,1) -- (9-1.125,0);
\draw (9-1.25,1) -- (9-1.375,0);
\draw (9-1.75,1) -- (9-1.625,0);
\draw (9-1.75,1) -- (9-1.875,0);
\node at (9-0,4) {$<$};
\node at (9-2,3) {$\perp$};
\node at (9-1,2) {$<$};
\node at (9-3,2) {$<$};
\node [scale=0.7] at (9-3.5,1) {$\perp$};
\node [scale=0.7] at (9-1.5,1) {$\perp$};
\end{tikzpicture}
\hspace{2pt}
\begin{tikzpicture}[scale=0.4]

%QhatTREE
\node [above] at (0,5) {$\Qhattree$};
\draw [fill] (0,5) circle [radius=0.06];
\draw (0,5) -- (-2,4);
\draw [fill] (-2,4) circle [radius=0.06];
\draw (0,5) -- (2,4);
\draw [fill] (2,4) circle [radius=0.06];
\draw (2,4) -- (3,3);
\draw [fill] (3,3) circle [radius=0.06];
\draw (2,4) -- (1,3);

\draw [fill] (1,3) circle [radius=0.06];
\draw (1,3) -- (1.5,2);
\draw [fill] (1.5,2) circle [radius=0.06];
\draw (1,3) -- (0.5,2);

\draw [fill] (0.5,2) circle [radius=0.06];
\draw (3,3) -- (3.5,2);
\draw [fill] (3.5,2) circle [radius=0.06];
\draw (3,3) -- (2.5,2);

\draw [fill] (2.5,2) circle [radius=0.06];
\draw (3.5,2) -- (3.25,1);
\draw [fill] (3.25,1) circle [radius=0.06];
\draw (3.5,2) -- (3.75,1);

\draw [fill] (3.75,1) circle [radius=0.06];
\draw (1.5,2) -- (1.25,1);
\draw [fill] (1.25,1) circle [radius=0.06];
\draw (1.5,2) -- (1.75,1);

\draw [fill] (1.75,1) circle [radius=0.06];
\draw (3.25,1) -- (3.125,0);
\draw (3.25,1) -- (3.375,0);
\draw (3.75,1) -- (3.625,0);
\draw (3.75,1) -- (3.875,0);

\draw (1.25,1) -- (1.125,0);
\draw (1.25,1) -- (1.375,0);
\draw (1.75,1) -- (1.625,0);
\draw (1.75,1) -- (1.875,0);
\node at (0,4) {$\perp$};
\node at (2,3) {$<$};
\node at (1,2) {$\perp$};
\node at (3,2) {$\perp$};
\node [scale=0.7] at (3.5,1) {$<$};
\node [scale=0.7] at (1.5,1) {$<$};
\end{tikzpicture}
\hspace{2pt}
\begin{tikzpicture}[scale=0.4]
%QTREE
\node [above] at (0,5-6) {$\Qtree$};
\draw [fill] (0,5-6) circle [radius=0.06];
\draw (0,5-6) -- (-2,4-6);
\draw [fill] (-2,4-6) circle [radius=0.06];
\draw (0,5-6) -- (2,4-6);
\draw [fill] (2,4-6) circle [radius=0.06];
\draw (2,4-6) -- (3,3-6);
\draw [fill] (3,3-6) circle [radius=0.06];
\draw (-2,4-6) -- (-3,3-6);
\draw [fill] (-3,3-6) circle [radius=0.06];

\draw (2,4-6) -- (1,3-6);
\draw [fill] (1,3-6) circle [radius=0.06];
\draw (-2,4-6) -- (-1,3-6);
\draw [fill] (-1,3-6) circle [radius=0.06];

\draw (1,3-6) -- (1.5,2-6);
\draw [fill] (1.5,2-6) circle [radius=0.06];
\draw (-1,3-6) -- (-1.5,2-6);
\draw [fill] (-1.5,2-6) circle [radius=0.06];

\draw (1,3-6) -- (0.5,2-6);
\draw [fill] (0.5,2-6) circle [radius=0.06];
\draw (-1,3-6) -- (-0.5,2-6);
\draw [fill] (-0.5,2-6) circle [radius=0.06];

\draw (3,3-6) -- (3.5,2-6);
\draw [fill] (3.5,2-6) circle [radius=0.06];
\draw (-3,3-6) -- (-3.5,2-6);
\draw [fill] (-3.5,2-6) circle [radius=0.06];

\draw (3,3-6) -- (2.5,2-6);
\draw [fill] (2.5,2-6) circle [radius=0.06];
\draw (-3,3-6) -- (-2.5,2-6);
\draw [fill] (-2.5,2-6) circle [radius=0.06];

\draw (3.5,2-6) -- (3.25,1-6);
\draw [fill] (3.25,1-6) circle [radius=0.06];

\draw (2.5,2-6) -- (2.25,1-6);
\draw [fill] (2.25,1-6) circle [radius=0.06];
\draw (2.5,2-6) -- (2.75,1-6);
\draw [fill] (2.75,1-6) circle [radius=0.06];

\draw (-2.5,2-6) -- (-2.25,1-6);
\draw [fill] (-2.25,1-6) circle [radius=0.06];
\draw (-2.5,2-6) -- (-2.75,1-6);
\draw [fill] (-2.75,1-6) circle [radius=0.06];

\draw (0.5,2-6) -- (0.25,1-6);
\draw [fill] (0.25,1-6) circle [radius=0.06];
\draw (0.5,2-6) -- (0.75,1-6);
\draw [fill] (0.75,1-6) circle [radius=0.06];

\draw (-0.5,2-6) -- (-0.25,1-6);
\draw [fill] (-0.25,1-6) circle [radius=0.06];
\draw (-0.5,2-6) -- (-0.75,1-6);
\draw [fill] (-0.75,1-6) circle [radius=0.06];

\draw (-3.5,2-6) -- (-3.25,1-6);
\draw [fill] (-3.25,1-6) circle [radius=0.06];

\draw (3.5,2-6) -- (3.75,1-6);
\draw [fill] (3.75,1-6) circle [radius=0.06];
\draw (-3.5,2-6) -- (-3.75,1-6);
\draw [fill] (-3.75,1-6) circle [radius=0.06];

\draw (1.5,2-6) -- (1.25,1-6);
\draw [fill] (1.25,1-6) circle [radius=0.06];
\draw (-1.5,2-6) -- (-1.25,1-6);
\draw [fill] (-1.25,1-6) circle [radius=0.06];

\draw (1.5,2-6) -- (1.75,1-6);
\draw [fill] (1.75,1-6) circle [radius=0.06];
\draw (-1.5,2-6) -- (-1.75,1-6);
\draw [fill] (-1.75,1-6) circle [radius=0.06];

\draw (0.25,1-6) -- (0.125,0-6);
\draw (0.25,1-6) -- (0.375,0-6);
\draw (0.75,1-6) -- (0.625,0-6);
\draw (0.75,1-6) -- (0.875,0-6);

\draw (1.25,1-6) -- (1.125,0-6);
\draw (1.25,1-6) -- (1.375,0-6);
\draw (1.75,1-6) -- (1.625,0-6);
\draw (1.75,1-6) -- (1.875,0-6);

\draw (2.25,1-6) -- (2.125,0-6);
\draw (2.25,1-6) -- (2.375,0-6);
\draw (2.75,1-6) -- (2.625,0-6);
\draw (2.75,1-6) -- (2.875,0-6);

\draw (3.25,1-6) -- (3.125,0-6);
\draw (3.25,1-6) -- (3.375,0-6);
\draw (3.75,1-6) -- (3.625,0-6);
\draw (3.75,1-6) -- (3.875,0-6);

\draw (-0.25,1-6) -- (-0.125,0-6);
\draw (-0.25,1-6) -- (-0.375,0-6);
\draw (-0.75,1-6) -- (-0.625,0-6);
\draw (-0.75,1-6) -- (-0.875,0-6);

\draw (-1.25,1-6) -- (-1.125,0-6);
\draw (-1.25,1-6) -- (-1.375,0-6);
\draw (-1.75,1-6) -- (-1.625,0-6);
\draw (-1.75,1-6) -- (-1.875,0-6);

\draw (-2.25,1-6) -- (-2.125,0-6);
\draw (-2.25,1-6) -- (-2.375,0-6);
\draw (-2.75,1-6) -- (-2.625,0-6);
\draw (-2.75,1-6) -- (-2.875,0-6);

\draw (-3.25,1-6) -- (-3.125,0-6);
\draw (-3.25,1-6) -- (-3.375,0-6);
\draw (-3.75,1-6) -- (-3.625,0-6);
\draw (-3.75,1-6) -- (-3.875,0-6);

\node at (0,4-6) {$<$};
\node at (2,3-6) {$<$};
\node at (-2,3-6) {$<$};
\node at (1,2-6) {$<$};
\node at (3,2-6) {$<$};
\node at (-1,2-6) {$<$};
\node at (-3,2-6) {$<$};

\node [scale=0.7] at (3.5,1-6) {$<$};
\node [scale=0.7] at (1.5,1-6) {$<$};
\node [scale=0.7] at (0.5,1-6) {$<$};
\node [scale=0.7] at (2.5,1-6) {$<$};

\node [scale=0.7] at (-3.5,1-6) {$<$};
\node [scale=0.7] at (-1.5,1-6) {$<$};
\node [scale=0.7] at (-0.5,1-6) {$<$};
\node [scale=0.7] at (-2.5,1-6) {$<$};
\end{tikzpicture}
\hspace{2pt}
\begin{tikzpicture}[scale=0.4]
%ATREE
\node [above] at (9+0,5-6) {$\Atree$};
\draw [fill] (9+0,5-6) circle [radius=0.06];
\draw (9+0,5-6) -- (9+-2,4-6);
\draw [fill] (9+-2,4-6) circle [radius=0.06];
\draw (9+0,5-6) -- (9+2,4-6);
\draw [fill] (9+2,4-6) circle [radius=0.06];
\draw (9+2,4-6) -- (9+3,3-6);
\draw [fill] (9+3,3-6) circle [radius=0.06];
\draw (9+-2,4-6) -- (9+-3,3-6);
\draw [fill] (9+-3,3-6) circle [radius=0.06];

\draw (9+2,4-6) -- (9+1,3-6);
\draw [fill] (9+1,3-6) circle [radius=0.06];
\draw (9+-2,4-6) -- (9+-1,3-6);
\draw [fill] (9+-1,3-6) circle [radius=0.06];

\draw (9+1,3-6) -- (9+1.5,2-6);
\draw [fill] (9+1.5,2-6) circle [radius=0.06];
\draw (9+-1,3-6) -- (9+-1.5,2-6);
\draw [fill] (9+-1.5,2-6) circle [radius=0.06];

\draw (9+1,3-6) -- (9+0.5,2-6);
\draw [fill] (9+0.5,2-6) circle [radius=0.06];
\draw (9+-1,3-6) -- (9+-0.5,2-6);
\draw [fill] (9+-0.5,2-6) circle [radius=0.06];

\draw (9+3,3-6) -- (9+3.5,2-6);
\draw [fill] (9+3.5,2-6) circle [radius=0.06];
\draw (9+-3,3-6) -- (9+-3.5,2-6);
\draw [fill] (9+-3.5,2-6) circle [radius=0.06];

\draw (9+3,3-6) -- (9+2.5,2-6);
\draw [fill] (9+2.5,2-6) circle [radius=0.06];
\draw (9+-3,3-6) -- (9+-2.5,2-6);
\draw [fill] (9+-2.5,2-6) circle [radius=0.06];

\draw (9+3.5,2-6) -- (9+3.25,1-6);
\draw [fill] (9+3.25,1-6) circle [radius=0.06];

\draw (9+2.5,2-6) -- (9+2.25,1-6);
\draw [fill] (9+2.25,1-6) circle [radius=0.06];
\draw (9+2.5,2-6) -- (9+2.75,1-6);
\draw [fill] (9+2.75,1-6) circle [radius=0.06];

\draw (9+-2.5,2-6) -- (9+-2.25,1-6);
\draw [fill] (9+-2.25,1-6) circle [radius=0.06];
\draw (9+-2.5,2-6) -- (9+-2.75,1-6);
\draw [fill] (9+-2.75,1-6) circle [radius=0.06];

\draw (9+0.5,2-6) -- (9+0.25,1-6);
\draw [fill] (9+0.25,1-6) circle [radius=0.06];
\draw (9+0.5,2-6) -- (9+0.75,1-6);
\draw [fill] (9+0.75,1-6) circle [radius=0.06];

\draw (9+-0.5,2-6) -- (9+-0.25,1-6);
\draw [fill] (9+-0.25,1-6) circle [radius=0.06];
\draw (9+-0.5,2-6) -- (9+-0.75,1-6);
\draw [fill] (9+-0.75,1-6) circle [radius=0.06];

\draw (9+-3.5,2-6) -- (9+-3.25,1-6);
\draw [fill] (9+-3.25,1-6) circle [radius=0.06];

\draw (9+3.5,2-6) -- (9+3.75,1-6);
\draw [fill] (9+3.75,1-6) circle [radius=0.06];
\draw (9+-3.5,2-6) -- (9+-3.75,1-6);
\draw [fill] (9+-3.75,1-6) circle [radius=0.06];

\draw (9+1.5,2-6) -- (9+1.25,1-6);
\draw [fill] (9+1.25,1-6) circle [radius=0.06];
\draw (9+-1.5,2-6) -- (9+-1.25,1-6);
\draw [fill] (9+-1.25,1-6) circle [radius=0.06];

\draw (9+1.5,2-6) -- (9+1.75,1-6);
\draw [fill] (9+1.75,1-6) circle [radius=0.06];
\draw (9+-1.5,2-6) -- (9+-1.75,1-6);
\draw [fill] (9+-1.75,1-6) circle [radius=0.06];

\draw (9+0.25,1-6) -- (9+0.125,0-6);
\draw (9+0.25,1-6) -- (9+0.375,0-6);
\draw (9+0.75,1-6) -- (9+0.625,0-6);
\draw (9+0.75,1-6) -- (9+0.875,0-6);

\draw (9+1.25,1-6) -- (9+1.125,0-6);
\draw (9+1.25,1-6) -- (9+1.375,0-6);
\draw (9+1.75,1-6) -- (9+1.625,0-6);
\draw (9+1.75,1-6) -- (9+1.875,0-6);

\draw (9+2.25,1-6) -- (9+2.125,0-6);
\draw (9+2.25,1-6) -- (9+2.375,0-6);
\draw (9+2.75,1-6) -- (9+2.625,0-6);
\draw (9+2.75,1-6) -- (9+2.875,0-6);

\draw (9+3.25,1-6) -- (9+3.125,0-6);
\draw (9+3.25,1-6) -- (9+3.375,0-6);
\draw (9+3.75,1-6) -- (9+3.625,0-6);
\draw (9+3.75,1-6) -- (9+3.875,0-6);

\draw (9+-0.25,1-6) -- (9+-0.125,0-6);
\draw (9+-0.25,1-6) -- (9+-0.375,0-6);
\draw (9+-0.75,1-6) -- (9+-0.625,0-6);
\draw (9+-0.75,1-6) -- (9+-0.875,0-6);

\draw (9+-1.25,1-6) -- (9+-1.125,0-6);
\draw (9+-1.25,1-6) -- (9+-1.375,0-6);
\draw (9+-1.75,1-6) -- (9+-1.625,0-6);
\draw (9+-1.75,1-6) -- (9+-1.875,0-6);

\draw (9+-2.25,1-6) -- (9+-2.125,0-6);
\draw (9+-2.25,1-6) -- (9+-2.375,0-6);
\draw (9+-2.75,1-6) -- (9+-2.625,0-6);
\draw (9+-2.75,1-6) -- (9+-2.875,0-6);

\draw (9+-3.25,1-6) -- (9+-3.125,0-6);
\draw (9+-3.25,1-6) -- (9+-3.375,0-6);
\draw (9+-3.75,1-6) -- (9+-3.625,0-6);
\draw (9+-3.75,1-6) -- (9+-3.875,0-6);

\node at (9+0,4-6) {$\perp$};
\node at (9+2,3-6) {$\perp$};
\node at (9+-2,3-6) {$\perp$};
\node at (9+1,2-6) {$\perp$};
\node at (9+3,2-6) {$\perp$};
\node at (9+-1,2-6) {$\perp$};
\node at (9+-3,2-6) {$\perp$};

\node [scale=0.7] at (9+3.5,1-6) {$\perp$};
\node [scale=0.7] at (9+1.5,1-6) {$\perp$};
\node [scale=0.7] at (9+0.5,1-6) {$\perp$};
\node [scale=0.7] at (9+2.5,1-6) {$\perp$};

\node [scale=0.7] at (9+-3.5,1-6) {$\perp$};
\node [scale=0.7] at (9+-1.5,1-6) {$\perp$};
\node [scale=0.7] at (9+-0.5,1-6) {$\perp$};
\node [scale=0.7] at (9+-2.5,1-6) {$\perp$};

\end{tikzpicture}

\caption{The structured trees $\Btree^+$, $\Btree^-$, $\Qhattree$, $\Qtree$ and $\Atree$.}\label{Fig:PthTrees}
\end{figure}
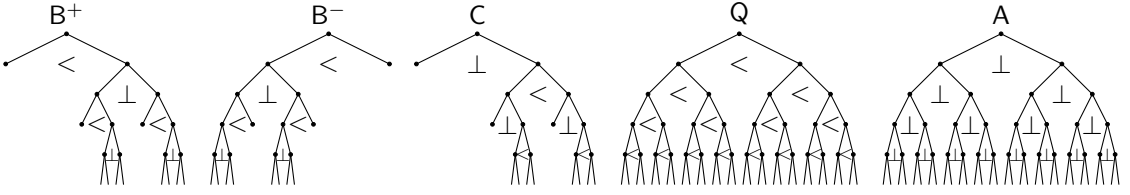
\begin{defn}
$\Btree^+$ has underlying set consisting of all finite sequences $s=\langle s_i:i\leq n\rangle$ of elements of $\{0,1,2,3\}$ such that:
\begin{itemize}
	\item if $s\neq \langle \rangle$ then $s_0\in\{2,3\}$,
	\item if $s_i\in \{0,1\}$ and $i<n$ then $s_{i+1}\in \{2,3\}$,
	\item if $s_i=2$ and $i<n$ then $s_{i+1}\in \{0,1\}$,
	\item if $s_i=3$ then $i=n$.
\end{itemize}
Thus $\Btree^+$ is a tree under $\is$. If $n$ is even, then we set $\col(s)=\sum_{\CH{2}}$ and label so that $$l_s(s\con \langle 3\rangle)=\min(\CH{2})\mbox{ and }l_s(s\con \langle 2\rangle \con t)=\max(\CH{2})$$ for every possible sequence $t$. If $n$ is odd, the we set $\col(s)=\sum_{\AC{2}}$ and $\range(l_s)=\AC{2}$.
We define the tree $\Btree^-$ in the same way as $\Btree^+$, with the only difference that $$l_s(s\con \langle 3\rangle)=\max(\CH{2})\mbox{ and }l_s(s\con \langle 2\rangle \con t)=\min(\CH{2})$$ for every possible sequence $t$.
We also define the tree $\Qhattree$ in the same way as $\Btree^+$, but change the labels and colours as follows. If $n$ is even, then we set $\col(s)=\sum_{\AC{2}}$ and label so that $\range(l_s)=\AC{2}$. If $n$ is odd then we set $\col(s)=\sum_{\CH{2}}$ and for every possible sequence $t$
$$l_s(s\con \langle 0\rangle\con t)=\min(\CH{2})\mbox{ and }l_s(s\con \langle 1\rangle\con t)=\max(\CH{2}).$$

We now define $\Qtree$ as a copy of $2^{<\omega}$, coloured and labelled so that for each $t\in \Qtree$ we have $\col(t)=\sum_{\CH{2}}$ and $\range(l_t)=\CH{2}$.
Finally we define $\Atree$ as a copy of $2^{<\omega}$, coloured and labelled so that for each $t\in \Atree$ we have $\col(t)=\sum_{\AC{2}}$ and $\range(l_t)=\AC{2}$.
\end{defn}

\begin{lemma}\label{Lemma:AtreeA}
Let $x\in \amr$ have a decomposition tree $T$. Then if $\Btree^+,\Btree^-\not \leq T$ and $\Atree\leq T$, then there is a $\down$-closed subset $A\subseteq T$ such that $\Atree\leq A$ and for each $t\in A$, we have $\range(l_t)=\AC{2}$.
\end{lemma}

\begin{proof}
Suppose $\Btree^+,\Btree^-\not \leq T$ and $\Atree\leq T$, so let $\varphi:\Atree\rightarrow T$ be a witnessing embedding.
For each $t=\varphi(s)\in \im(\varphi)$, we have that $P_t=\range(l_t)$ is some partial order that embeds $\AC{2}$. Let $a_t=l_t(\varphi(s\con \langle 0 \rangle))$ and $b_t=l_t(\varphi(s\con \langle 1 \rangle))$. Hence $a_t,b_t\in P_t$ with $a_t\perp b_t$. Now for each $t$, either $P_t=\{a_t,b_t\}$ or there is some $c_t\in P_t$ such that without loss of generality (swapping the names of $a_t$ and $b_t$ if necessary) one of the following cases occurs:
\begin{enumerate}%[(i)]
	\item \label{AtreeProof1}$a_t,b_t\perp c_t$,
%	\item $a_t<c_t$, $b_t\perp c_t$,
	\item \label{AtreeProof2}$a_t<c_t$,
	\item \label{AtreeProof3}$a_t>c_t$.
%	\item $a_t>c_t$, $b_t\perp c_t$,
%	\item $a_t,b_t<c_t$,
%	\item $a_t,b_t>c_t$.
\end{enumerate}
%Any other configuration would contradict $a_t\perp b_t$.

Suppose that there is some $u_0\in \im(\varphi)$ such that for every $u\in \im(\varphi)\cap \down u_0$ there is $t(u)\geq u$ and some $c_{t(u)}$ satisfying case \ref{AtreeProof2}, (i.e. $a_{t(u)}<c_{t(u)}$). %Suppose that for some $v_0\in A$, and any $v\geq v_0$ there is some $w_v\geq v$ and some $a\in U\cap \downs w_v$ and $b\in \downs w_v$ such that $l_{w_v}(a)<l_{w_v}(b)$. 
Set $\tau(\langle \rangle)=t(u_0)$. Suppose we have defined $\tau(s)$ for some $s\in \Btree^-$ of length $n$. To simplify notation, we let $s'$ be $s$ with its last element removed and $\pi=\tau(s)$. Now suppose further that whenever $s\neq s'\con \langle 3\rangle$, we have we have $\tau(s)\in \im(\varphi)$; and whenever $s=\langle \rangle$, or $s=s'\con \langle i \rangle$ for $i\in \{0,1\}$ we have that $\tau(s)$ satisfies case \ref{AtreeProof2}.

If $s=s'\con \langle 2 \rangle$ we let $\delta_0$ and $\delta_1$ be elements of $$\im(\varphi)\cap {}^{a_{\pi}}\downs \pi\hspace{5pt}\mbox{ and }\hspace{5pt}\im(\varphi)\cap {}^{b_{\pi}}\downs \pi$$ respectively, and set $\tau(s\con \langle 0 \rangle)=t(\delta_0)$ and $\tau(s\con \langle 1 \rangle)=t(\delta_1)$. These exist and are incomparable since $\Atree$ was a copy of $2^{<\omega}$. If $s=\langle \rangle$ or $s=s'\con \langle i \rangle$ for $i\in \{0,1\}$, then pick the values of $\tau(s\con \langle 2\rangle)$ and $\tau(s\con \langle 3\rangle)$ to be elements of 
$$\im(\varphi)\cap {}^{a_{\pi}} \downs \pi\hspace{5pt}\mbox{ and }\hspace{5pt}{}^{c_{\pi}} \downs \pi$$
respectively.
%\begin{itemize} 
%	\item $\tau(s\con \langle 2\rangle)\in \im(\varphi)\cap {}^{a_{t(\pi)}} \downs \pi$ and% such that $l_{\tau(s)}(\delta)=a_{ \tau(s)}$ and
%	\item $\tau(s\con \langle 3\rangle)\in {}^{c_{t(\pi)}} \downs \pi$.% such that $l_{\tau(s)}(\theta)=c_{ \tau(s)}$.
%\end{itemize}
So by construction, the map $\tau:\Btree^-\rightarrow T$ is an embedding, which is a contradiction. Thus no such $u_0$ exists and there is $u_1\in \im(\varphi)$ such that no $t>u_1$ satisfies case \ref{AtreeProof2}. 

Now using that $\Btree^+\not \leq T$, we can apply a similar argument to the tree $\down u_1$ (in place of $T$) and case \ref{AtreeProof3} (in place of case \ref{AtreeProof2}). So we find a $u_2\in \im(\varphi)$ such that for every $t>u_2$ we have that $t$ does not satisfy cases \ref{AtreeProof2} or \ref{AtreeProof3} for any choice of $c_t$. Let $A=\down u_2$, so that for each $t\in A$ we have either $P_t=\{a_t,b_t\}$ or $t$ satisfies case \ref{AtreeProof1}, for any choice of $c_t$. Hence $P_t$ is an antichain, and therefore by Lemma \ref{Lemma:IndecompLabels} we always have $P_t=\{a_t,b_t\}=\AC{2}$, since any indecomposable antichain is either $1$ or $\AC{2}$. It remains to check that $\Atree\leq A$, but this is clear since $A=\down u_2$, with $u_2\in \im(\varphi)$.
\end{proof}
\begin{lemma}\label{Lemma:Atree2}
Let $T$ be a decomposition tree for $x\in \amr\cap \pscat_\mathbb{P}^\mathbb{L}(Q)$. If $\Btree^+,\Btree^-\not \leq T$, then $\Atree \not \leq T$.
\end{lemma}
\begin{proof}
Suppose for contradiction that $\Btree^+,\Btree^-\not \leq T$ and $\Atree\leq T$.  %By the reasoning of Remark \ref{Rk:POsMaxEtas} we can assume without loss of generality that each chain used to construct $T$ was maximal. 
By Lemma \ref{Lemma:AtreeA} there is some $\down$-closed subset $A\subseteq T$ such that for each $t\in A$ we have $\range(l_t)=\AC{2}$.
Consider $y=x(t,p)$ for arbitrary $t\in A$ and $p\in \range(l_t)$.
We claim that $y$ is an antichain. Suppose not, then there is some chain $y'\subseteq y$ of size $2$. Hence by Lemma \ref{Lemma:POinDcompAppearinT}, since $y'$ is an indecomposable subset of $y$ with $|y'|>1$, there must be some $t\in T_y\subseteq A$ such that $y'$ embeds into $P_t=\range(l_t)$. But this is a contradiction since each $P_t$ was an antichain, therefore indeed we have that $y$ is an antichain.

Let $\vec{q}$ be shortest such that there is $t'=\vec{q}\con \langle i_0\rangle\in {}^{p}\downs t$ for some $i_0$, let $w\in a^{\eta(\vec{q})}_{i_0}$, then set $z=x(t',w)$. Then $z$ has decomposition tree ${}^{w}\downs t'$ which is a $\down$-closed subset of $A$ and hence $|T_z|>1$ and therefore $|z|>1$.
%
%
%Pick a leaf $b\in T^1_y\setminus S^1_y$, which exists since $T_y$ was a $\down$-closed subset of $A$, and therefore embeds $2^{<\omega}$. Then set $z=\col(b)\subseteq y$, so that $T_z$ is a $\down$-closed subset of $T_y$, and so in particular we have $$|z|\geq |T_y| >1. $$
Let $\eta(\vec{q})=\eta=\langle \langle f_i,s_i\rangle:i\in r\rangle$ and for $j\in r$ let $\eta_j=\langle \langle f_i,s_i\rangle :i\geq j\rangle$. By Remark \ref{Rk:POsMaxEtas} can assume that there is a maximal chain of intervals $\langle I_j:j\in r\rangle$ of $y$, and some $k=\langle q_u:u\in A^\eta\rangle\in \dom(f^\eta)_{<\alpha}$ such that $x=f^\eta(k)$ and $I_j=f^{\eta_j}(\langle q_u:u\in a^\eta_i, i\geq j\rangle)$ for each $j\in r$. % we can assume that %these $\eta$ and $k$ were used in the construction of $T_y$, i.e. that 
%$\eta=\eta(\vec{q})$. 

%Since $b\in T_y^1\setminus S^1_y$, it was the case that for some $j\in r$ and $\pi \in k^j$ we have $z=\col(\pi)$. %We also know that 
Thus we have that $z=q_w$ so that $z\subseteq I_{i_0}=\sum_{u\in H_{\eta_{i_0}}} q_u$, since $w\in H_{\eta_{i_0}}$. For each $i\in r$ we have $f_i=\sum_{\AC{2}}$, therefore $H_\eta$ is an antichain of size $|r|$. Let $I=I_j\setminus\{\rho\}$ for some $\rho\in z$ then since $|z|>1$ we have $$\bigcup_{i>i_0}I_i\subset I\subset I_{i_0}.$$
But $I\subseteq y$ is also an interval, since $y$ is an antichain. Therefore $\langle I_j:j\in r\rangle$ was not a maximal chain of intervals of $y$. This contradiction gives the lemma.
\end{proof}

\begin{lemma}\label{Lemma:QtreeC}
Let $x\in \amr$ have a decomposition tree $T$. Then if $\Qhattree\not \leq T$ and $\Qtree\leq T$, then there is a $\down$-closed subset $A\subseteq T$ such that $\Qtree\leq A$ and for each $t\in A$, we have $\range(l_t)=\CH{2}$.
\end{lemma}
\begin{proof}
Similar to Lemma \ref{Lemma:AtreeA}, replacing $\perp$ with $\not \perp$ and $<$ with $\perp$.
\end{proof}
\begin{lemma}\label{Lemma:Qtree2}
Let $x\in \amr\cap \pscat^{\mathbb{L}}_{\mathbb{P}}(Q)$ have a decomposition tree $T$. Then if $\Qhattree\not \leq T$ then $\Qtree \not \leq T$.
\end{lemma}
\begin{proof}
Suppose for contradiction that $\Qhattree\not \leq T$ and $\Qtree\leq T$. By the reasoning of Remark \ref{Rk:POsMaxEtas} we can assume without loss of generality that each chain used to construct $T$ was maximal. By Lemma \ref{Lemma:QtreeC} there is some $\down$-closed subset $A\subseteq T$ such that for each $t\in A$ we have $\range(l_t)=\CH{2}$. We can now proceed as in Lemma \ref{Lemma:Atree2}, replacing the word chain with antichain, and vice versa.
\end{proof}

\begin{lemma}\label{Lemma:POsBasicallyThmLimit}
If $x\in \amr \cap \pscat^{\mathbb{L}}_{\mathbb{P}}(Q)$ has a decomposition tree $T$ such that $\Btree^+\leq T$ $($resp. $\Btree^-$, $\Qhattree)$, then $2^{<\omega}\leq x$ $($resp. $-2^{<\omega}$, $\Qhat)$.
\end{lemma}
\begin{proof}
Let $\varphi:\Btree^+\rightarrow T$ be an embedding. Let $W$ be the subset of $\Btree^+$ consisting of $\langle \rangle$, $u\con \langle 0 \rangle$ and $u\con \langle 1 \rangle$ for all possible sequences $u$. For each $s\in W$, let $t_s=\varphi(s)$, and $u_s=l_{t_s}(\varphi(s\con \langle 3 \rangle))$. Then let $y_s$ be an element of $$x(t_s,l_{t_s}(u_s))\subseteq x.$$
We now define an embedding $\psi:2^{<\omega}\rightarrow x$. Given $a=\langle a_0,a_1,...,a_{n-1}\rangle\in 2^{<\omega}$, if $a=\langle \rangle$ then let $a'=a$, and if $a\neq \langle \rangle$ let $a'=\langle 2,a_0,2,a_1,...,2,a_{n-1}\rangle$. Now set $\psi(a)=y_{a'}$,
 so if $a,b\in 2^{<\omega}$ are such that $a\leq b$ then $a\is b$ so that $a'\is b'$ and thus $t_{a'}\leq t_{b'}$ and $l_{t_{a'}}(u_{a'})\leq l_{t_{b'}}(u_{b'})$ which means that $\psi(a)\leq \psi(b)$. If $a\perp b$ then similarly $t_{a'}\leq t_{b'}$ and $l_{t_{a'}}(u_{a'})\perp l_{t_{b'}}(u_{b'})$ so that $\psi(a)\perp \psi(b)$. %We also have that $\psi$ is injective since if $t_{a'}\neq t_{b'}$ then $y_{a'}\neq y_{b'}$. 
Thus $\psi$ is an embedding and witnesses $2^{<\omega}\leq x$.
The cases for $-2^{<\omega}$ and $\Qhat$ are similar.
\end{proof}

\subsection{Pathological partial orders}\label{SubSubSection:PathPOs}
\begin{defn}
Set $\pth=\{2^{<\omega}, -2^{<\omega}, \Qhat\}$. We call elements of $\pth$ \emph{pathological} partial orders.
\end{defn}
\begin{lemma}\label{Lemma:SumsPathological}
Suppose we have $x\in \CC$ and for each $i\in x$, we have $x_i\in \CC$. Suppose that for $y\in \pth$ and all $z\in \{x\}\cup \{x_i\mid i\in x\}$ we have $y\not \leq z$. Then $y\not \leq \sum_{i\in x}x_i.$
\end{lemma}
\begin{proof}
Fix $y\in \pth$ and suppose that $y\leq \sum_{i\in x}x_i$, with $\varphi$ a witnessing embedding. For all $t\in y$, let $i_t$ be such that $\varphi(s)\in x_{i_t}$. 
Let $F$ be a finite subset of $x$, $i\in F$, and $s\in y$. 

We claim that there is some $s'\in y$ with $s\is s'$ such that for all $z\in y$ with $s'\is z$ we have $i_z\neq i$. So suppose for contradiction that for all $s'\in y$ with $s\is s'$ there is some $Z(s')\in y$ with $s'\is Z(s')$ and $i_{Z(s')}=i$. We define $\psi:y\rightarrow x_i$ inductively as follows. Let $\psi(\langle \rangle)=Z(s)$ and if for $t\in y$ we have defined $\psi(t)=\varphi(t')$ then for $m\in \{0,1\}$ let $\psi(t\con \langle m \rangle)=\varphi(Z(t'\con \langle m \rangle)).$
It is easily verified then that $\psi$ is an embedding, which is a contradiction and we have the claim. Applying the claim repeatedly for each $i\in F$, we then see that for all $s\in y$ there is $s_F\in y$ with $s\is s_F$ such that for all $i\in F$ and all $z\in y$ with $s_F\is z$ we have $i_z\neq i$.

Now let $\mu(\langle \rangle)=\varphi(\langle \rangle)$. Suppose inductively we have defined $\mu$ on some sequences $t\in y$ so that $\mu(t)=\varphi(t')$ for some $t'$, and let $G$ be the set of these $t$ such that $\mu(t)$ is already defined. Let $u\in y$ be the lexicographically least element of $y\setminus G$ and let $v\in y$ and $m\in \{0,1\}$ be such that $u=v\con \langle m\rangle$; so $\mu(v)$ is already defined. Now let $u'=v'\con \langle m\rangle$ and let $\mu(u)=\varphi(s_G)$. Now $\mu$ is an embedding and $i_{\mu(t)}$ is distinct for distinct $t$. So let $\mu':y\rightarrow x$ be given by $\mu'(t)=i_{\mu(t)}$, then $\mu'$ is an embedding, which is a contradiction.
\end{proof}

\begin{prop}\label{Prop:PthSSR}
Let $y\in \pth$, and $s,s',t\in y$ be such that $s\is s'$ and $s$ and $t$ are incomparable under $\is$. Then $\neg \SSR{s}{s'}{t}$ and $\neg \SSR{s'}{s}{t}$.
\end{prop}
\begin{proof}
Let $y\in \pth$ and $s,s',t$ be as described. Suppose that $y=2^{<\omega}$, then since $2^{<\omega}$ is just ordered by $\is$ we have $s\leq s'$ and $s\perp t$ hence $\neg \SSR{s}{s'}{t}$. We also have that $s'\perp t$, and therefore $\neg \SSR{s'}{s}{t}$.
If $y=-2^{<\omega}$ then we have $s\geq s'$, $s\perp t$ and $s'\perp t$, and again we can conclude that $\neg \SSR{s}{s'}{t}$ and $\neg \SSR{s'}{s}{t}$. If $y=\Qhat$ then we have $s\perp s'$, and either $t>s$ and $t>s'$ or $t<s$ and $t<s'$. Hence again we can conclude $\neg \SSR{s}{s'}{t}$ and $\neg \SSR{s'}{s}{t}$.
\end{proof}

\begin{lemma}\label{Lemma:pthdontembed_Heta}
Suppose that no element of $\pth$ embeds into any element of $\mathbb{P}$. Then for all $y\in \pth$, and for every composition sequence $\eta$ we have $y\not \leq H_\eta$.
\end{lemma}
\begin{proof}
Let $\eta=\langle \langle f_i,s_i\rangle:i\in r\rangle$ be a composition sequence, then %by Lemma \ref{Lemma:IndecompLabels}, 
for each $i\in r$, we have that $\arity(f_i)\in \PP$. %is indecomposable. 
Since $\PP=\mathbb{P}$ we know that $$(\forall i\in r)(\forall y\in \pth), y\not \leq \arity(f_i).$$
Suppose that for some $y\in \pth$, we have $y\leq H_\eta$ with $\varphi$ a witnessing embedding. 
We claim that for any $a=\varphi(s)\in a^\eta_i$, there are $a_0=\varphi(s_0)\in a^\eta_{i_0}$ and $a_1=\varphi(s_1)\in a^\eta_{i_1}$ with $s\is s_0,s_1$, and $s_0,s_1$ incomparable under $\is$, with $i\neq i_0\neq i_1 \neq i$.
Suppose not, then for some $s\in y$ and $i\in r$, we have $\varphi(s\con t)\subseteq a^\eta_i$ for every sequence $t$. This is a contradiction, since %for each $y\in \pth$ we can see that 
then $y$ embeds into $\{t'\in y\mid s\is t'\}\leq a^\eta_i\leq \arity(f_i)$. This gives the claim.

Let $a=\varphi(\langle \rangle)$ and choose $a_0=\varphi(s_0)\in a^\eta_{i_0}$ and $a_1=\varphi(s_1)\in a^\eta_{i_1}$ as in the claim. Now choose $a_{00}=\varphi(s_{00})\in a^\eta_{i_{00}}$ and $a_{11}=\varphi(s_{11})\in a^\eta_{i_{11}}$ by applying the claim to $a_0$ and $a_1$. Notice that by a similar argument to before, we can assume that every element of $I=\{i,i_0,i_1,i_{00},i_{11}\}$ is distinct, otherwise we would be able to embed $y$ into $\arity(f_j)$ for some $j\in I$.

We now use Proposition \ref{Prop:HetaSSR} in the following cases:
\begin{itemize}
	\item $i_0<i_1,i_{00}$ which implies $\SSR{a_0}{a_1}{a_{00}}$,
	\item $i_{00}<i_0,i_1$ which implies $\SSR{a_{00}}{a_0}{a_1}$,
	\item $i_1<i_0,i_{00}$ and $i_1<i_0,i_{11}$ which implies $\SSR{a_1}{a_0}{a_{11}}$,
	\item $i_1<i_0,i_{00}$ and $i_{11}<i_0,i_1$ which implies $\SSR{a_{11}}{a_0}{a_1}$,
	\item $i_1<i_0,i_{00}$ and $i_0<i_1,i_{11}$ which implies $i_0<i_1<i_0$ which is a contradiction.
\end{itemize}
Now note that any of the first four cases contradict Proposition \ref{Prop:PthSSR}. So we have a contradiction in every case, and our assumption that $y\leq H_\eta$ must have been false. This gives the lemma.
\end{proof}
\begin{lemma}\label{Lemma:PathologicalDontEmbedAmrs}
If $x\in \amrs$ and for all $y\in \pth$, $z\in \mathbb{P}$, we have $y\not \leq z$, then for all $y'\in \pth$,  $y'\not \leq x$.
\end{lemma}
\begin{proof}
Clearly this holds when $\rank(x)=0$, since then $x$ is a single point and so does not embed any element of $\pth$. Suppose the statement holds for all $x'\in \CC_{<\alpha}$ for some $\alpha\in \On$, and that $x\in \CC_\alpha$.
Then $x=f^{\eta}(k)=\sum_{u\in H_\eta}q_u$ for some composition sequence $\eta$ and $k=\langle q_u:u\in A^\eta\rangle\in \dom(f^\eta)_{<\alpha}$. 
By the induction hypothesis, for each $u\in H_\eta$ and $y\in \pth$ we know that $y\not \leq q_u$, and by Lemma \ref{Lemma:pthdontembed_Heta} we know that $y\not \leq H_\eta$. So by Lemma \ref{Lemma:SumsPathological}, we have that $y\not \leq x$ as required.
\end{proof}

\subsection{Characterising the construction}

\begin{lemma}\label{Lemma:Indecomp1Sums}
Suppose that:
\begin{itemize}
	\item $U$ is an indecomposable partial order with $|U|>2$;
	\item $P=\sum_{u\in U}P_u$ for some non-empty partial orders $P_u (u\in U)$;
	\item $I\subseteq P$ is an interval of $P$ with $I\cap P_{v_0}\neq \emptyset$ and $I\cap P_{v_1}\neq \emptyset$ for some $v_0\neq v_1\in U$.
\end{itemize}
Then $I=P$.
\end{lemma}
\begin{proof}
First we claim that $$J=\{u\mid I\cap P_u\neq \emptyset\}$$ is an interval of $U$. To see this, we let $v\in U\setminus J$, $u_0,u_1\in J$, $a\in P_v$ and $b_0\in I\cap P_{u_0}$, $b_1\in I\cap P_{u_1}$. Then $a\notin I$ (otherwise $v\in J$), so we have $\SSR{a}{b_0}{b_1}$ because $I$ was an interval. But this implies $\SSR{v}{u_0}{u_1}$ since $P=\sum_{u\in U}P_u$, hence we have the claim that $J$ is an interval of $U$.

Since $U$ was indecomposable, we either have $|J|=1$ or $J=U$. If $|J|=1$ then this contradicts our assumption that $I\cap P_{v_0}\neq \emptyset$ and $I\cap P_{v_1}\neq \emptyset$ for some $v_0\neq v_1\in U$. Hence $J=U$.

Suppose for contradiction that there is some $a\in P\setminus I$, then $a\in P_v$ for some $v\in U$. For arbitrary $u_0,u_1\in U=I'$ with $v\notin \{u_0,u_1\}$ we have $b_0\in I\cap P_{u_0}$ and $b_1\in I\cap P_{u_1}$. Since $a\in P\setminus I$, $b_0,b_1\in I$ and $I$ is an interval, we then know that $\SSR{a}{b_0}{b_1}$. So since $P=\sum_{u\in U}P_u$ and $v\notin \{u_0,u_1\}$ we see that $\SSR{v}{u_0}{u_1}$. But then since $u_0$ and $u_1$ were arbitrary, it must be that $U\setminus \{v\}$ is an interval. Hence since $U$ was indecomposable we have $|U\setminus\{v\}|=1$ which means $|U|=2$ which is a contradiction.
\end{proof}
\begin{lemma}\label{Lemma:Indecomp2Sums}
Suppose that:
\begin{itemize}
	\item $U=\{0,1\}$ with $0R1$ for some $R\in \{<,>,\perp\}$, so $U\in \{\AC{2},\CH{2}\}$;
	\item $P=\sum_{u\in U}P_u$ for some non-empty partial orders $P_u, (u\in U)$;
	\item $I\subseteq P$ is an interval of $P$ with $I\cap P_0\neq \emptyset$ and $I\cap P_1\neq \emptyset$;
	\item there are no non-empty subsets $K_0,K_1$ of $P_0$ with $K_0\cap K_1 = \emptyset$ and $P_0=K_0\cup K_1$, such that for all $a_0\in K_0$ and $a_1\in K_1$ we have $a_0 R a_1$.
\end{itemize}
Then $P_0\subseteq I$.
\end{lemma}
\begin{proof}
Suppose that $P_0\not \subseteq I$ so we can let $K_0=I\cap P_0$, $K_1=P_0\setminus I$. So there are $a\in I\cap P_0$ and $b\in P_0\setminus I$, such that $a\perp b$ if $R\neq \perp$ and $a\not \perp b$ if $R=\perp$, so in either case we have $\neg(bRa)$. Let $c\in P_1\cap I$, then $b Rc$ since $P=\sum_{u\in U}P_u$ and $b\in P_0$. So we have shown $\neg\SSR{b}{a}{c}$. But $a,c\in I$ and $b\notin I$ with $I$ an interval, hence we have $\SSR{b}{a}{c}$, which is a contradiction.
\end{proof}
In the next definition and following few lemmas (\ref{Defn:IntervalPs}, \ref{Lemma:IntervalAritySizes}, \ref{Lemma:IntervalClassification}, \ref{Lemma:SumsChainsInL} and \ref{Lemma:SumsIndecompSubsetsInP}) we fix a composition sequence $\eta=\langle \langle f_i,s_i\rangle:i\in r\rangle$ and $k\in \dom(f^\eta)$ so that $\langle \eta, k\rangle$ is maximal for $P\in \amrs$.
\begin{defn}\label{Defn:IntervalPs}
%For $i\in r$ and $u\in a^\eta_i$, let $P_{i,u}=\col_{k_i}(u)\in \CC$, and let 
Let $k=\langle P_{i,u}:i\in r, u\in a^\eta_i\rangle$ and define:
\begin{itemize}
	\item $P_i=\sum_{u\in a^\eta_i}P_{i,u}$;
	\item $P_{\geq i}=\bigcup_{j\geq u}P_j\subseteq f^\eta(k)$;
	\item $P_{>i}=P_{\geq i}\setminus P_i$.
\end{itemize}
\end{defn}
\begin{lemma}\label{Lemma:IntervalAritySizes}
Suppose that $I$ is an interval of $P$ and $i\in r$ be such that $I\cap P_i\neq \emptyset$ and $I\cap P_{>i}\neq \emptyset$. If $|\arity(f_i)|>2$ then $P_{\geq i}\subseteq I$, and if $|\arity(f_i)|=2$ then $P_i\subseteq I$.
\end{lemma}

\begin{proof}
We have that $P=f^\eta(k)$ is an $H_\eta$-sum of the non-empty partial orders $P_{i,a}$.
By definition of $H_\eta$ we have that $P_{\geq i}$ is an $\arity(f_{i})$-sum of the $P_{i,u}$ in position $u\in a^\eta_{i}$, and $P_{>i}$ in position $s_i$. We know that $\arity(f_{i})$ is indecomposable since $\langle \eta,k\rangle$ was maximal. So since $I\cap P_{i}\neq \emptyset$, we can apply Lemma \ref{Lemma:Indecomp1Sums} to see that either $|\arity(f_{i})|\leq 2$ or $P_{\geq i}\subseteq I$.

%If $|\arity(f_{j_0})|=1$ then the sum over $j_0$ makes no difference to the sum over $H_\eta$ so without loss of generality assume this does not occur.

If $|\arity(f_i)|=2$ then set $R\in\{<,>,\perp\}$ such that for $a\in a^\eta_i$ we have $a R s_i$. Then using Lemma \ref{Lemma:Indecomp2Sums}, we either have either that $P_i\subseteq I$ or there are non-empty $K_0,K_1\subseteq P_i$ with $K_0\cap K_1=\emptyset$ and $P_i=K_0\cup K_1$ such that for all $a_0\in K_0$ and $a_1\in K_1$ we have $a_0 R a_1$. Consider $J=K_1\cup P_{>i}$, let $a\in P\setminus J$ and $z_0,z_1\in J$ then either $a\in P\setminus P_{\geq j}$ in which case we have $\SSR{a}{z_0}{z_1}$ or $a\in K_0$ in which case $a R z_0$ and $a R z_1$, so that $\SSR{a}{z_0}{z_1}$. Thus $J$ is an interval. But we have that $P_{>i}\subset J\subset P_{\geq i}$, which contradicts that $\eta$ was maximal. So it must be that $P_i\subseteq I$.
\end{proof}

\begin{lemma}\label{Lemma:IntervalClassification}
Let $I$ be an interval of $P$, then there are $j_0,j_1\in r$ and an $X\subseteq P_{j_1}$ either empty or an interval of $P_{j_1}$ such that $I=P_{\geq j_0}$ or $I=(P_{\geq j_0}\setminus P_{\geq j_1})\cup X.$
\end{lemma}
\begin{proof}
%Let $\eta=\langle \langle f_i,s_i\rangle:i\in r\rangle$ be as described. We have that $P=f^\eta(k)$ is an $H_\eta$-sum of the non-empty partial orders $P_{i,a}$.

%Suppose that $I$ is an interval of $P$ and $i\in r$ be such that $I\cap P_i\neq \emptyset$ and $I\cap P_{>i}\neq \emptyset$ (if no such $i$ existed then set $j_0=j_1$ and $X=I$ and we are done). Let $j_0=\inf\{i\mid I\cap P_i\neq \emptyset\}\in r$, this exists by Lemma \ref{Lemma:r'infimum}. By definition of $H_\eta$ we have that $P_{\geq j_0}$ is an $\arity(f_{j_0})$-sum of the $P_{j_0,u}$ in position $u\in a^\eta_{j_0}$, and $P_{>j_0}$ in position $s_{j_0}$. We know that $\arity(f_{j_0})$ is indecomposable since $\langle \eta,k\rangle$ was maximal. So since $I\cap P_{j_0}\neq \emptyset$, we can apply Lemma \ref{Lemma:Indecomp1Sums} to see that either $|\arity(f_{j_0})|\leq 2$ or $P_{\geq j_0}\subseteq I$.

If for any $i\in r$ we have $|\arity(f_{i})|=1$ then $f_i=\sum_1$ which makes no difference to the sum over $H_\eta$, so without loss of generality we assume this does not occur. %If $|\arity(f_{j_0})|>2$ then we have $P_{\geq j_0}\subseteq I$. By definition of $j_0$ clearly we also have $I\subseteq P_{\geq j_0}$. Hence $I=P_{\geq j_0}$ and we are done.

Let $j_0=\inf\{i\in r\mid I\cap P_i\neq \emptyset\}$, and $j_1=\sup\{i\in r\mid I\cap P_i\neq \emptyset\}$, then $j_0,j_1\in r$ by Lemma \ref{Lemma:r'infimum}. Suppose that for some $n\in r$ with $j_0\leq n< j_1$ we have $|\arity(f_n)|\geq 2$. We claim that in this case $I=P_{\geq i_0}$. For each such $n$ we have by Lemma \ref{Lemma:IntervalAritySizes} that $P_{\geq n}\subseteq I$. Therefore for all $m\in r$ with $j_0\leq m<n$ we have again by Lemma \ref{Lemma:IntervalAritySizes} that $P_m\subseteq I$, hence $I=P_{\geq j_0}$ by definition of $j_0$.

Suppose now that the only $n\in r$ such that $j_0\leq n\leq j_1$ with $|\arity(f_n)|>2$ is $n=j_1$, or that no such $n$ exists. Let $X=I\cap P_{j_1}$, then $X$ is either empty or an interval of $P_{j_1}$. If $j_0=j_1$ we are done, otherwise for each $j_0\leq m<j_1$ we have $|\arity(f_n)|=2$. We either have $P_{j_1}\neq \emptyset$ or $P_{j_1}=\emptyset$ in which case by definition of $j_1$ there must be some $m'\in r$ with $m<m'<j_1$ and $P_{m'}\neq \emptyset$. So in either case we can apply Lemma \ref{Lemma:IntervalAritySizes} to see that $P_{m}\subseteq I$. Using the definition of $j_0$ we also have that $I\subseteq P_{\geq j_0}$. Therefore we have shown that $I=(P_{\geq j_0}\setminus P_{\geq j_1})\cup X$ as required.
\end{proof}
\begin{lemma}\label{Lemma:SumsChainsInL}
Suppose that $\forall i\in r$ and $\forall a\in a^\eta_i$, every chain of intervals of $P_{i,a}$ under $\supseteq$ has order type in $\clos{\mathbb{L}}$. Then every chain of intervals of $f^\eta(k)$ under $\supseteq$ has order type in $\clos{\mathbb{L}}$.
\end{lemma}
\begin{proof}
Let $\langle J_\alpha:\alpha\in \sigma\rangle$ be a chain of intervals of $P$ under $\supseteq$. If for some $\alpha \in \sigma$ we have some $j_\alpha\in r$ with $P_{\geq j_\alpha}= J_\alpha$ then for every $\beta\leq \alpha$, by Lemma \ref{Lemma:IntervalClassification} and since $J_\beta \supseteq J_\alpha$, there must be some $j_\beta$ with $J_\beta=P_{\geq j_\beta}$. Hence the order type of the chain $\langle J_\beta:\beta \leq \alpha\rangle$ must be a subset of the order type of $\{i\in r\mid j\leq j_\alpha\}$, and hence this order type is in $\clos{\mathbb{L}}$. Therefore since $\clos{\mathbb{L}}$ is closed under finite sums, it only remains to show that the final segment $$F=\{ J_\gamma\mid (\forall i\in r),J_\gamma \neq P_{\geq i}\}$$ has order type in $\clos{\mathbb{L}}$ under $\supseteq$.

Let $J_\gamma\in F$, then by Lemma \ref{Lemma:IntervalClassification}, there are some $j_\gamma, j'_\gamma\in r$ and $X$ an interval of $P_{j'_\gamma}$ such that $$J_\gamma=(P_{\geq j_\gamma}\setminus P_{\geq j'_\gamma})\cup X_\gamma$$
Now if $\gamma<\upsilon$ then $j_\gamma\leq j_\upsilon$, $j'_\gamma\geq j'_\upsilon$ and $X_\gamma \supseteq X_\upsilon$, and if $j'_\gamma>j'_\upsilon$ then $X_\upsilon\subseteq J_\gamma$.
Let $F_\delta=\{J_\gamma\in F\mid j_\gamma=j_\delta\}$. Then $\ot(F)$ is an $\ot(\{j_\delta \mid \delta \in \sigma\})$-sum of the $\ot(F_\delta)$. But $\{j_\delta \mid \delta\in \sigma \}\subseteq r$ hence this subset has order type in $\clos{\mathbb{L}}$. So it remains only to show that each $\ot(F_\delta)$ is in $\clos{\mathbb{L}}$. We also have that each $\ot(F_\delta)$ is an $\ot(\{j'_\tau \mid \tau \in \sigma, j_\tau=j_\delta\})$-sum of the order types of $F_{\delta, \tau}=\{X_\lambda\mid \lambda\in \sigma, j_\lambda=j_\delta,j'_\lambda=j'_\tau\}$. But $\{j'_\tau \mid \tau \in \sigma, j_\tau=j_\delta\}\subseteq r$ and $F_{\delta,\tau}$ is a chain of intervals of $P_{\tau}$ under $\supseteq$, whence both are in $\clos{\mathbb{L}}$ and thus $\ot(F_\delta)\in \clos{\mathbb{L}}$ which gives the lemma.
% Let $$\delta=\ot(\{\langle j_\gamma, j'_\gamma\rangle\mid \gamma\in \sigma\}).$$ So we see that $\sigma$ is a $\delta$-sum of order types of chains of intervals of the $P_i$ $(i\in r)$. Since each such chain is in $\clos{\mathbb{L}}$ by assumption (since $P_i=P_{i,a}$ in this case), it remains only to show that $\delta\in \clos{\mathbb{L}}$. Let $\gamma=\ot(\{j_\gamma\mid \gamma\in \sigma\})$ so that $\delta$ is a $\gamma$-sum of subsets of $r$, but indeed $\gamma$ itself must be a subset of $r$. Hence $\delta\in \clos{\mathbb{L}}$ and thus $\sigma \in \clos{\mathbb{L}}$ as required.
\end{proof}

\begin{lemma}\label{Lemma:SumsIndecompSubsetsInP}
Suppose that for any indecomposable partial order $y$ and any $i\in r$ and $a\in a^\eta_i$ we have that $y\leq P_{i,a}\longrightarrow y\in \mathbb{P}$. Then for any indecomposable $y'$ we have $y'\leq f^\eta(k)\longrightarrow y'\in\mathbb{P}$.
\end{lemma}
\begin{proof}
Since $\langle \eta,k\rangle$ was maximal for $P$ we know that for each $i\in r$ we have $\arity(f_i)$ is indecomposable. Now any subset $A$ of $H_\eta=\bigsqcup_{i\in r} a^\eta_i$ that contains points of $a^\eta_i$ and at least two points of $\bigcup_{j>i} a^{\eta}_j$ is such that $A\cap \bigcup_{j>i} a^{\eta}_j$ is an interval with at least two points inside. So here $A$ cannot be indecomposable. Thus any indecomposable subset $I$ of $H_\eta$ is a subset of $a^\eta_i\sqcup a^\eta_j$ for some $i,j\in r$ with $i<j$, moreover, $|I\cap a^\eta_j|\leq 1$. So $I$ has the same order type as a subset of $\arity(f_i)$, which shows that $I$ has order type in $\mathbb{P}$.

Thus if we take a subset $A\subseteq P$ with at least two points inside a single $P_{i,a}$ and at least one point not in $P_{i,a}$ then $A\cap P_{i,a}$ is an interval which contradicts that $A$ is indecomposable. We know that $P$ is an $H_\eta$-sum of the partial orders $P_{i,a}$. So let $J$ be an indecomposable subset $J$ of the sum $P$. Then either $J$ is entirely contained within some $P_{i,a}$ and hence $J$ has order type in $\mathbb{P}$; or $J$ contains at most one point of each of the $P_{i,a}$ that it intersects, and hence has the same order type as an indecomposable order of $H_\eta$. Hence by the previous paragraph $J$ has order type in $\mathbb{P}$, which completes the proof.
\end{proof}

We are now ready to prove the following generalisation of Hausdorff's famous theorem on scattered linear orders (Theorem \ref{Thm:Hausdorff}).

\begin{thm}\label{Thm:PPL=amrs}
$\scat_\mathbb{P}^\mathbb{L}(Q')=\amrs$.
\end{thm}
\begin{proof}
First we claim that $\amrs\subseteq \scat_\mathbb{P}^\mathbb{L}(Q')$. By Lemma \ref{Lemma:PathologicalDontEmbedAmrs}, we have that $\amrs$ satisfies (\ref{PPL3}) of Definition \ref{Defn:PPL}. We will show (\ref{PPL1}) and (\ref{PPL2}) by induction on the rank of $x\in \amrs$. If $x$ has rank $0$ then $x$ is just a single point and so satisfies (\ref{PPL1}) and (\ref{PPL2}) since both $\mathbb{P}$ and $\mathbb{L}$ contained $1$. Now suppose that %for some $\alpha\in \On$ and 
any $y\in \CC_{<\alpha}$ satisfies (\ref{PPL1}) and (\ref{PPL2}). Then if $x\in \CC_{\alpha}$, we have that $x=f^{\eta}(k)$ where $k\in\dom(f^\eta)_{<\alpha}$. Hence by Lemma \ref{Lemma:SumsIndecompSubsetsInP} we have that $x$ satisfies (\ref{PPL1}), and by Lemma \ref{Lemma:SumsChainsInL} we have that $x$ satisfies (\ref{PPL2}). So indeed we have $\amrs\subseteq \scat_\mathbb{P}^\mathbb{L}(Q)$.

We will now show $\scat_\mathbb{P}^\mathbb{L}(Q')\subseteq \amrs$. Suppose we have some non-empty $Q'$-coloured partial order $x\notin \amrs$. We claim that either (\ref{PPL1}), (\ref{PPL2}) or (\ref{PPL3}) fails for $x$. Suppose that (\ref{PPL1}) and (\ref{PPL2}) hold, we will show that (\ref{PPL3}) fails. By Lemma \ref{Lemma:XnotinAmrsThenDecompTreeNotScat} we have that any decomposition tree $T$ for $x$ embeds $2^{<\omega}$. So let $T$ be a decomposition tree for $x$ and $B\subseteq T$ be a copy of $2^{<\omega}$. Then we can either find a copy of $\Atree$ inside $B$, or below some point of $B$ there is no point coloured by $\sum_{\AC{2}}$, and hence below this point there is a copy of $\Qtree$. Thus by Lemmas \ref{Lemma:Atree2} and \ref{Lemma:Qtree2} we have that either $\Btree^+\leq T$, $\Btree^-\leq T$ or $\Qhattree \leq T$. Hence by Lemma \ref{Lemma:POsBasicallyThmLimit} either $2^{<\omega}$, $-2^{<\omega}$ or $\Qhat$ embeds into $x$. Thus $x$ fails (\ref{PPL3}), and $x\notin \scat_\mathbb{P}^\mathbb{L}(Q')$, which gives the theorem.
\end{proof}

\begin{cor}\label{Cor:MPL=amr}
$\sscat_\mathbb{P}^\mathbb{L}(Q')=\amr$.
\end{cor}
\begin{proof}
We have that $\sscat_\mathbb{P}^\mathbb{L}(Q')$ is the class containing $\scat_\mathbb{P}^\mathbb{L}(Q')$ and countable unions of increasing sequences and $\amr$ is the class containing $\amrs$ and countable unions of limiting sequences. The result then follows from Theorem \ref{Thm:PPL=amrs}, considering Remark \ref{Rk:MLPSeqs}.%By Theorem \ref{Thm:PPL=amrs} we have that $\scat_\mathbb{P}^\mathbb{L}(Q')=\amrs$, hence $\sscat_\mathbb{P}^\mathbb{L}(Q')=\amr$.
\end{proof}

%
%\subsection{The class $\sscat_\mathbb{P}^\mathbb{L}$ of partial orders is well-behaved}

\begin{thm}\label{Thm:MLPisWB2}
If $\mathbb{L}$ and $\mathbb{P}$ are well-behaved, then $\sscat^\mathbb{L}_\mathbb{P}$ is well-behaved.
\end{thm}
\begin{proof}
By Corollary \ref{Cor:MPL=amr} we have that $\sscat_\mathbb{P}^\mathbb{L}(Q')=\amr$. Suppose there is a bad $\sscat_\mathbb{P}^\mathbb{L}(Q)$-array; so by theorems \ref{Thm:TLPWell-Behaved} and \ref{Thm:MLPisWB}, we have a witnessing bad $Q$-array. Hence $\sscat^\mathbb{L}_\mathbb{P}$ is well-behaved.
\end{proof}

\section{Corollaries and conclusions}\label{Section:Cors}
We now present some applications of Theorem \ref{Thm:MLPisWB2}.

\begin{thm}[K\v{r}\'{i}\v{z} \cite{Kriz}]\label{Thm:sscatWB}
$\sscat$, the class of $\sigma$-scattered linear orders is well-behaved.
\end{thm}
\begin{proof}
Set $\mathbb{L}=\On\cup \On^*$, set $\mathbb{P}=\{1,\CH{2}\}$. Then $\clos{\mathbb{L}}=\scat$ by Theorem \ref{Thm:Hausdorff}. $\mathbb{L}$ is well-behaved, by theorems \ref{Thm:U bqo} and \ref{Thm:OnWB}. We also have that $\mathbb{P}$ is well-behaved by Lemma \ref{Lemma:FinitePosWB}. Hence by Theorem \ref{Thm:MLPisWB2} we have that $\sscat^{\mathbb{L}}_{\mathbb{P}}$ is well-behaved. 

We claim that $\sscat^{\mathbb{L}}_{\mathbb{P}}=\sscat$. Recall the definition of $\scat^{\mathbb{L}}_{\mathbb{P}}$, \ref{Defn:PPL}. Let $X\in \scat$, then clearly $X$ satisfies (\ref{PPL1}) and (\ref{PPL3}) since if either failed then $X$ would contain two incomparable elements. A chain of intervals of $X$ does not embed $\mathbb{Q}$, otherwise $X$ would embed $\mathbb{Q}$. Hence every chain of intervals of $X$ under $\supseteq$ has scattered order type. If we take the chain of intervals of $X$ under $\supseteq$ consisting of final segments, this chain has order type precisely the same as $X$. So we have that $\scat^{\mathbb{L}}_{\mathbb{P}}=\scat$ which implies that $\sscat^{\mathbb{L}}_{\mathbb{P}}=\sscat$, and therefore $\sscat$ is well-behaved.
\end{proof}

\begin{defn}
Let $\mathbb{P}$ be a set of countable indecomposable partial orders. Then $\mathscr{C}_\mathbb{P}$ is the class of countable partial orders such that every indecomposable subset is in $\mathbb{P}$.
\end{defn}

\begin{thm}\label{Thm:CPsubsetMLP}
If $\ctbl \subseteq \clos{\mathbb{L}}$ then $\mathscr{C}_\mathbb{P}\subseteq \sscat^\mathbb{L}_\mathbb{P}$.
\end{thm}
\begin{proof}
Suppose that $\ctbl \subseteq \clos{\mathbb{L}}$. Let $X\in \mathscr{C}_\mathbb{P}$, we claim that $X\in \sscat^\mathbb{L}_\mathbb{P}$. Fix an enumeration of $X$ in order type $\omega$, so that $X=\{x_n:n\in \omega\}$. (If $|X|<\aleph_0$ then the argument is essentially the same.) We will write $X$ as the countable union of some limiting sequence $(X_n)_{n\in \omega}$ that we will define.

Let $X_{\langle\rangle}=X$, $\PS_0=\{\langle \rangle\}$ and suppose for some sequence ${\vec{t}}$ of elements of $\bigcup \clos{\mathbb{L}}\times \bigcup \mathbb{P}$ we have defined some $X_{\vec{t}}\subseteq X$. 
When $|X_{\vec{t}}|>1$, pick a maximal chain $\langle I^{\vec{t}}_i:i\in r_{\vec{t}}\rangle$ of intervals of $X_{\vec{t}}$ that contains $\{x_n\}$, where $n$ is least such that $x_n\in X_{\vec{t}}$. %For $i\in r_s$, let $P^s_i=I^s_i\setminus \bigcup_{j>i}I^s_j$, so the $P^s_i$ $(i\in r_s)$ form a partition of $X_s$.
Since $X_{\vec{t}}\subseteq X$ we have that $r_{\vec{t}}\in \clos{\mathbb{L}}$. For each $i\in r_{\vec{t}}$, let $\mathcal{D}_i^{\vec{t}}$ be the set of unions of maximal chains of intervals of $I^{\vec{t}}_i$ that do not contain $I^{\vec{t}}_i$.

We claim that $a_i=I^{\vec{t}}_i/\mathcal{D}_i^{\vec{t}}$ is indecomposable. If $Z\neq a_i$ is a non-singleton interval of $a_i$, then let $Z'\subseteq I^{\vec{t}}_i$ be the union of the sets of $\mathcal{D}_i^{\vec{t}}$ that contain points of $Z$. Then $Z'$ is an interval of $I^{\vec{t}}_i$ not equal to $I^{\vec{t}}_i$, and not in any of the maximal chains used to form $\mathcal{D}_i^{\vec{t}}$; this is a contradiction which gives the claim. Now, we see that $a_i$ is an indecomposable subset of $X$ and thus $a_i\in \mathbb{P}$.

Now let $f^{\vec{t}}_i=\sum_{a_i}$ and $s^{\vec{t}}_i$ be the point of $a_i$ that was also contained in $\bigcup_{j>i}I^{\vec{t}}_j\in \mathcal{D}_i^{\vec{t}}$. We set $\eta({\vec{t}})=\langle \langle f^{\vec{t}}_i,s^{\vec{t}}_i\rangle:i\in r_{\vec{t}}\rangle$. For $i\in r_{\vec{t}}$ and $u\in a^\eta_i\subseteq a_i$, let $X_{{\vec{t}}\con \langle i,u\rangle}$ be the element of $\mathcal{D}_i^{\vec{t}}$ that contains the point $u$. Then if $k_{\vec{t}}=\langle X_{{\vec{t}}\con \langle i,u\rangle}:i\in r_{\vec{t}}, u\in a^\eta_i\rangle$ we have by construction that $\langle \eta, k_{\vec{t}}\rangle$ is maximal for $X_{\vec{t}}$.

For $m\in \omega$ we let $\PS_{m+1}=\PS_m\cup \{{\vec{t}}\con \langle i,u\rangle\mid {\vec{t}}\in \PS_m, i\in r_{\vec{t}}, u\in a^\eta_i, |X_{{\vec{t}}}|>1\}$, $\DS_m$ be the set of leaves of $\PS_m$ and $\FS_m=\{\eta({\vec{t}})\mid {\vec{t}}\in \PS_m\setminus \DS_m\}$. Then $\FS_m$ is an admissible composition set and $\PS_m$ is a set of position sequences for $\FS_m$. Now let $d_m=\langle d_m^{\vec{p}}:\vec{p}\in \DS_m\rangle \in \dom(f^{\FS_m})$ be such that each $d_m^{\vec{p}}$ is a single point, coloured either by $-\infty$ if $|X_{\vec{p}}|>1$, and coloured by $\col_{X_{\vec{p}}}(x)$ when $X_{\vec{p}}=\{x\}$. Now define $X_m=g^{\FS_m}(d_m)$, so that by construction $(X_m)_{m\in \omega}$ is a limiting sequence. Moreover, for each $m\in \omega$ we can consider the underlying set of each $X_m$ as a subset of the underlying set of $X$. Note also that the underling set of $X_m$ contains the point $x_m$, and since the relevant chain of intervals contained $\{x_m\}$ the colour of this point in $X_m$ will be the same as from $X$. Thus the limit of $(X_m)_{m\in \omega}$ is precisely $X$, so that $X\in \sscat^\mathbb{L}_\mathbb{P}$ as required.
\end{proof}

\begin{cor}\label{Cor:C_PisWB}
If $\mathbb{P}$ is a well-behaved set of countable indecomposable partial orders, then $\mathscr{C}_\mathbb{P}$ is well-behaved.\footnote{This result was obtained independently by Christian Delhomme in as yet unpublished work. The author thanks him for his private communication.}
\end{cor}
\begin{proof}
By Theorems \ref{Thm:sscatWB}, \ref{Thm:CPsubsetMLP} and \ref{Thm:MLPisWB2}.
\end{proof}
\begin{remark}
If $\mathbb{P}$ is a set of finite partial orders then let $\tilde{\mathbb{P}}$ be the class of countable partial orders whose every finite restriction is in $\mathbb{P}$. In \cite{PouzetApps}, Pouzet asked: if $\mathbb{P}$ preserves bqo, then is $\tilde{\mathbb{P}}$ bqo? As we have seen, well-behaved is a more useful concept than preserving bqo, so we modify the question to if $\mathbb{P}$ well-behaved. Corollary \ref{Cor:C_PisWB} brings us closer to a result of this kind, however fails to account for possible infinite indecomposable subsets of orders in $\tilde{\mathbb{P}}$. If we could prove that for any infinite indecomposable order $X$ the set of finite indecomposable subsets of $X$ is not well-behaved, then we would answer this version of Pouzet's question positively.
\end{remark}

\begin{defn}
For $n\in \omega$ let $\mathscr{I}_n$ denote the set of indecomposable partial orders whose cardinality is at most $n$.
\end{defn}

\begin{thm}\label{Thm:POTHM}
For any $n\in \omega$, the class $\sscat^\sscat_{\mathscr{I}_n}$ is well-behaved.
\end{thm}
\begin{proof}
$\mathscr{I}_n$ is a finite set of finite partial orders so by Lemma \ref{Lemma:FinitePosWB}, is well-behaved. Furthermore, $\sscat$ is well-behaved by Theorem \ref{Thm:sscatWB}, so using Theorem \ref{Thm:MLPisWB2} completes the proof.
\end{proof}
\begin{cor}
For any $n\in \omega$, the class $\ctbl_{\mathscr{I}_n}$ is well-behaved.
\end{cor}
\begin{proof}
By Theorems \ref{Thm:CPsubsetMLP} and \ref{Thm:POTHM}.
\end{proof}
Note that $\bigcup_{n\in \omega}\mathscr{I}_n(\AC{2})$ contains an antichain as in Figure \ref{Fig:Antichain}, and as described by Pouzet in \cite{PouzetApps}. Hence $\bigcup_{n\in \omega}\mathscr{I}_n$ does not preserve bqo and is certainly not well-behaved. Thus Theorem \ref{Thm:POTHM} is maximal in some sense.
\begin{figure}
\centering
\begin{tikzpicture}[thick]
\draw [->-] (0,0) -- (0.5,1);
\draw [-<-] (0.5,1) -- (1,0);
\draw [fill] (0,0) circle [radius=0.06];
\draw [fill] (0.5,1) circle [radius=0.06];
\draw [fill] (1,0) circle [radius=0.06];
\node [below] at (0,0) {$1$};
\node [above] at (0.5,1) {$0$};
\node [below] at (1,0) {$1$};
\end{tikzpicture}
\hspace{14pt}
\begin{tikzpicture}[thick]
\draw [->-] (0,0) -- (0.5,1);
\draw [-<-] (0.5,1) -- (1,0);
\draw [->-] (1,0) -- (1.5,1);
\draw [fill] (0,0) circle [radius=0.06];
\draw [fill] (0.5,1) circle [radius=0.06];
\draw [fill] (1,0) circle [radius=0.06];
\draw [fill] (1.5,1) circle [radius=0.06];
\node [below] at (0,0) {$1$};
\node [above] at (0.5,1) {$0$};
\node [below] at (1,0) {$0$};
\node [above] at (1.5,1) {$1$};
\end{tikzpicture}
\hspace{14pt}
\begin{tikzpicture}[thick]
\draw [->-] (0,0) -- (0.5,1);
\draw [-<-] (0.5,1) -- (1,0);
\draw [->-] (1,0) -- (1.5,1);
\draw [-<-] (1.5,1) -- (2,0);
\draw [fill] (0,0) circle [radius=0.06];
\draw [fill] (0.5,1) circle [radius=0.06];
\draw [fill] (1,0) circle [radius=0.06];
\draw [fill] (1.5,1) circle [radius=0.06];
\draw [fill] (2,0) circle [radius=0.06];
\node [below] at (0,0) {$1$};
\node [above] at (0.5,1) {$0$};
\node [below] at (1,0) {$0$};
\node [above] at (1.5,1) {$0$};
\node [below] at (2,0) {$1$};
\end{tikzpicture}
\hspace{14pt}
\begin{tikzpicture}[thick]
\draw [->-] (0,0) -- (0.5,1);
\draw [-<-] (0.5,1) -- (1,0);
\draw [->-] (1,0) -- (1.5,1);
\draw [-<-] (1.5,1) -- (2,0);
\draw [->-] (2,0) -- (2.5,1);
\draw [fill] (0,0) circle [radius=0.06];
\draw [fill] (0.5,1) circle [radius=0.06];
\draw [fill] (1,0) circle [radius=0.06];
\draw [fill] (1.5,1) circle [radius=0.06];
\draw [fill] (2,0) circle [radius=0.06];
\draw [fill] (2.5,1) circle [radius=0.06];
\node [below] at (0,0) {$1$};
\node [above] at (0.5,1) {$0$};
\node [below] at (1,0) {$0$};
\node [above] at (1.5,1) {$0$};
\node [below] at (2,0) {$0$};
\node [above] at (2.5,1) {$1$};
\node at (3.5,0.5) {$...$};
\end{tikzpicture}
\caption{An antichain of $\bigcup_{n\in \omega}\mathscr{I}_n(\AC{2})$.}
\label{Fig:Antichain}
\end{figure}
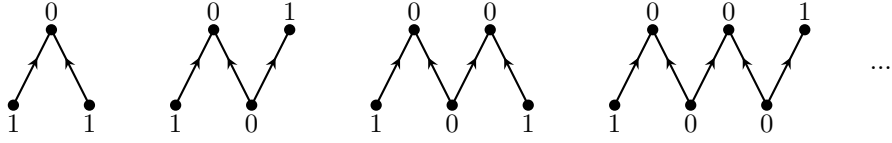
In order to improve this result we would like to know the answers to questions such as:
\begin{enumerate}
	\item Is there consistently a well-behaved class of linear orders larger than $\sscat$? E.g. is the class of Aronszajn lines from \cite{Aron} well-behaved under PFA?
	\item Is there an infinite well-behaved class of indecomposable partial orders?
	\item Is there an infinite indecomposable partial order $P$ such that $\{P\}$ is well-behaved?
\end{enumerate}
A positive answer to any of these questions would immediately improve Theorem \ref{Thm:POTHM}. 
%\newpage
\section*{Acknowledgements}
The author thanks Mirna D\v{z}amonja, Yann Pequignot and Rapha\"{e}l Carroy for extremely helpful discussions and remarks.

\end{document}